\documentclass[11pt,a4paper,leqno]{amsart}
\usepackage{a4wide}
\usepackage{amsmath,amstext, amsthm, amssymb}
\usepackage{latexsym, amsthm, amsfonts, amsmath,amssymb,graphicx,xcolor}
\usepackage{color}

\usepackage{latexsym}
\usepackage{graphicx}
\usepackage{amsmath}
\usepackage{amssymb}
\usepackage[colorlinks=true, linkcolor=blue, 
hyperfootnotes=true,citecolor=blue,urlcolor=black]{hyperref}

\newtheorem{theo}{Theorem}[section]
\newtheorem{lemma}[theo]{Lemma}
\newtheorem{prop}[theo]{Proposition}

\theoremstyle{definition}
\newtheorem{defin}[theo]{Definition}
 
\theoremstyle{remark}
\newtheorem{rmk}[theo]{Remark}

\numberwithin{equation}{section}

\def\R{\mathbb{R}}
\def\Z{\mathbb{Z}}
\def\C{\mathbb{C}}

\def\n{\eta}

\def\m{\mu}
\def\n'{\nu}

\begin{document}
\title[Perturbed $q$-difference-differential equations with holomorphic coefficients]{On the multiple-scale analysis for some linear partial $q$-difference and differential equations with holomorphic coefficients}
\author{Thomas Dreyfus}
\address{IRMA, Universit\'e de Strasbourg, 7 rue Ren\'e Descartes, 67084 Strasbourg, France\\}
\email{dreyfus@math.unistra.fr}
\author{Alberto Lastra}
\address{University of Alcal\'{a}, Departamento de F\'{i}sica y Matem\'{a}ticas,
Ap. de Correos 20, E-28871 Alcal\'{a} de Henares (Madrid), Spain}
\email{alberto.lastra@uah.es}
\author{St\'ephane Malek}
\address{University of Lille 1, Laboratoire Paul Painlev\'e,
59655 Villeneuve d'Ascq cedex, France}
\email{Stephane.Malek@math.univ-lille1.fr}


\date{\today}

\bibliographystyle{amsalpha}
\sloppy
\begin{abstract}
The analytic and formal solutions of certain family of $q$-difference-differential equations under the action of a complex perturbation parameter is considered. The previous study~\cite{lamaq} provides information in the case when the main equation under study is factorizable, as a product of two equations in the so-called normal form. Each of them gives rise to a single level of $q$-Gevrey asymptotic expansion. In the present work, the main problem under study does not suffer any factorization, and a different approach is followed. More precisely, we lean on the technique developed in~\cite{dreyfus}, where the first author makes distinction among the different $q$-Gevrey asymptotic levels by successive applications of two $q$-Borel-Laplace transforms of different orders both to the same initial problem and which can be described by means of a Newton polygon.
\end{abstract} 

\maketitle

%\tableofcontents
%\thispagestyle{empty}

%\newpage

\section{Introduction}

This work is devoted to the study of a family of linear $q$-difference-differential problems in the complex domain. It can be arranged into a series of works dedicated to the asymptotic study of holomorphic solutions to different kinds of $q$-difference-differential problems involving irregular singularities such as \cite{lamaq2}, \cite{lama2}, \cite{lamaq}, \cite{lamasa2},  and \cite{ma0}. The study of $q$-difference and $q$-difference-differential equations in the complex domain is a promising and fruitful domain of research. In the literature, one may find other interesting approaches to these problems. We refer to~\cite{taya2} as a reference, and contributions in the framework of nonlinear $q$-analogs of Briot-Bouquet type partial differential equations in~\cite{ya}. We provide~\cite{ta,ya2} as novel studies in this direction. 

The study of $q$-difference equations has also been under study in different applications in the last years. Some advances in this respect are~\cite{pr1,pr2,pr3}, and the references therein.

The main aim of this work is to study a family of $q$-difference-differential equations of the form
\begin{multline}\label{e1}
Q(\partial_z)\sigma_{q,t}u(t,z,\epsilon)\\
=(\epsilon t)^{d_{D_1}}\sigma_{q,t}^{\frac{d_{D_1}}{k_1}+1}R_{D_1}(\partial_{z})u(t,z,\epsilon)+(\epsilon t)^{d_{D_2}}\sigma_{q,t}^{\frac{d_{D_2}}{k_2}+1}R_{D_2}(\partial_{z})u(t,z,\epsilon)\\
+\sum_{\ell=1}^{D-1}\epsilon^{\Delta_\ell}t^{d_{\ell}}\sigma_{q,t}^{\delta_{\ell}}(c_{\ell}(t,z,\epsilon)R_{\ell}(\partial_z)u(t,z,\epsilon))+\sigma_{q,t}f(t,z,\epsilon),
\end{multline}
where $D,D_1,D_2$ are integer numbers larger than 3, $Q$, $R_{D_1}$, $R_{D_2}$ and $R_{\ell}$ for $\ell=1,\ldots, D-1$ are polynomials of complex coefficients, and $\Delta_{\ell}\ge0$, $\delta_\ell,d_\ell\ge1$ are nonnegative integers for every $1\le \ell\le D-1$. The numbers $d_{D_1}$ and $d_{D_2}$ are positive integers. $q$ stands for a real number with $q>1$.

We consider the dilation operator $\sigma_{q,t}$ acting on variable $t$, i.e. $\sigma_{q,t}(t\mapsto f(t)):=f(qt)$, and the generalization of its composition given by 
$$\sigma_{q,t}^{\gamma}(t\mapsto f(t)):=f(q^{\gamma}t),$$
for any $\gamma\in\R$. We also fix positive integer numbers $k_1$ and $k_2$ with 
$$1\le k_1<k_2.$$

As in the former work~\cite{ma} of the third author, the coefficients $c_{\ell}(t,z,\epsilon)$
and the forcing term $f(t,z,\epsilon)$ represent bounded holomorphic functions in the
vicinity of the origin in $\mathbb{C}^{2}$ w.r.t $(t,\epsilon)$ and on a horizontal
strip $H_{\beta} = \{ z \in \mathbb{C} / |\mathrm{Im}(z)| < \beta \}$ of width
$2\beta>0$ relatively to the space variable $z$. However, a new additional constraint
is required on the growth of the Taylor expansion of each $c_{\ell}$ according to the
mixed variable $t\epsilon$, see (\ref{e119}). It implies that the functions $c_{\ell}(t,z,\epsilon)$ can be extended as entire functions in the monomial $\epsilon t$ in the whole plane
$\mathbb{C}$ with so-called $q$-exponential growth of some order related to $k_{1}$
and $k_{2}$ (this terminology will be explained later in the paper).

Two singularly perturbed terms on the right hand side of equation (\ref{e1}) are distinguished. This makes a crucial difference with respect to the previous work~\cite{lamaq} in which only one term appears, whilst the multi-level $q$-Gevrey asymptotic behavior comes from the forcing term. More precisely, in that previous work we focused on families of $q$-difference-differential equations that can be factorized as a product of two operators in so-called normal forms each enjoying one single level of $q$-Gevrey asymptotics. In the present work, the appearance of these two terms would cause a multilevel $q$-Gevrey phenomenon in the study of the asymptotic solution of (\ref{e1}) regarding the perturbation parameter. %However, a direct application of $q$-Borel-Laplace summation procedure of order $k_2$ would fail, as it is observed at the beginning of Section~\ref{seccion42}. 
Our approach is to follow a two-step procedure of summation of the formal solution, which makes the two $q$-Gevrey asymptotic orders emerge.

Another important difference compared to our previous contribution~\cite{lamaq} is that we are now able to handle holomorphic
coefficients in time $t$ whilst only polynomial coefficients were considered in~\cite{lamaq}. This
fact relies on new technical bounds for a $q$-analog of the convolution of order $k$
displayed in Proposition~\ref{prop3}.

The point of view we use here is similar to the one performed in the work of the first author, see~\cite{dreyfus}, and is related to direct constraints on the shape of the main equation via a possible description by a Newton polygon. It is important to stress that this approach is specific to the $q$-difference case. Namely, such a direct procedure for producing two different Gevrey levels in the differential case for the problem stated in the work~\cite{lama} is impossible due to very strong restrictions related to a formula used in the proof and appearing in~\cite{taya13}, see formula (8.7) p. 3630. In that case, only a proposal via factoring the main equation did actually work, as performed in our joint work~\cite{lama0}.

%The point of view we use here is similar to the one performed in the work of the first author, see~\cite{dreyfus}, and is related to direct constraints on the shape of the main equation via a possible description by a Newton polygon. It is important to stress that this approach is specific to the $q$-difference case. Namely, such a direct procedure for producing two different Gevrey levels in the differential case for the problem stated in the work~\cite{lama} is impossible due to very strong restrictions related to a formula used in the proof and appearing in~\cite{taya13}, see formula (8.7) p. 3630. In that case, only a proposal via factoring the main equation did actually works, as performed in our joint work~\cite{lamaq}.

Let us briefly review the steps followed in order to achieve our main results in the present work.

Let $0\le p\leq \varsigma-1$. First, we apply $q$-Borel transformation of order $k_1$ to equation (\ref{e1}) in order to obtain our first auxiliary problem in a Borel plane, problem (\ref{e87}). A fixed point result in a complex Banach space of functions under an appropriate growth at infinity lead us to an analytic solution, $w_{k_1}^{\mathfrak{d}_p}(\tau,m,\epsilon)$ of (\ref{e87}). More precisely, $w_{k_1}^{\mathfrak{d}_p}(\tau,m,\epsilon)$ defines a continuous function defined in $(U_{\mathfrak{d}_p}\cup D(0,\rho))\times\R\times D(0,\epsilon_0)$, where $U_{\mathfrak{d}_p}$ is an infinite sector of bisecting direction $\mathfrak{d}_p$, and holomorphic with respect to the variables $\tau$ and $\epsilon$ in $(U_{\mathfrak{d}_p}\cup D(0,\rho))$ and $D(0,\epsilon_0)$, respectively. In addition to that, it holds that this function admits $q$-exponential growth of order $\kappa$ at infinity with respect to $\tau$ in $U_{\mathfrak{d}_p}$.%, i.e., there exists $C_{w_{k_1}^{\mathfrak{d}_p}}>0$ such that $$|w_{k_1}^{\mathfrak{d}_p}(\tau,m,\epsilon)|\le C_{w_{k_1}^{\mathfrak{d}_p}}\frac{1}{(1+|m|)^{\mu}}e^{-\beta|m|}\exp\left(\frac{\kappa}{2\log(q)}\log^2|\tau+\delta|+\alpha\log|\tau+\delta|\right),$$ for all $m\in\R$ and all $\tau\in (D(0,\rho)\cup U_{\mathfrak{d}_p})$ and $\epsilon\in D(0,\epsilon_0)$. 
This result is described in detail in Proposition~\ref{prop347}.

A second auxiliary problem in the Borel plane is constructed by applying the formal $q$-Borel transformation of order $k_2$ to the main problem (\ref{e1}). The second auxiliary equation is stated in (\ref{e412b}). A second fixed point result in another appropriate Banach space of functions allow us to guarantee the existence of an actual solution of the second auxiliary problem, $w^{\mathfrak{d}_p}_{k_2}(\tau,m,\epsilon)$, defined in $S_{\mathfrak{d}_{p}}\times\R\times D(0,\epsilon_0)$ and holomorphic with respect to $\tau$ and $\epsilon$ in $S_{\mathfrak{d}_p}$ and $D(0,\epsilon_0)$, respectively. Here, $S_{\mathfrak{d}_p}$ stands for an infinite sector with vertex at the origin and bisecting direction $\mathfrak{d}_p$. %Moreover, this function satisfies $$|w_{k_2}^{\mathfrak{d}_p}(\tau,m,\epsilon)|\le C_{w_{k_2}^{\mathfrak{d}_p}}\frac{1}{(1+|m|)^{\mu}}e^{-\beta|m|}\exp\left(\frac{k_2}{2\log(q)}\log^2|\tau|+\nu\log|\tau|\right),$$ for some $C_{w_{k_2}^{\mathfrak{d}_p}}>0$, and some $\nu\in\R$, valid for every $(\tau,m,\epsilon)\in S_{\mathfrak{d}_p}\times\R\times D(0,\epsilon_0)$. 
Moreover, this function suffers $q$-exponential growth of order $k_{2}$
at infinity w.r.t. $\tau$ on $S_{\mathfrak{d}_{p}}$. This statement is proved in Proposition~\ref{prop11}.

As a matter of fact, the key point in our reasoning is the link between the $q$-Laplace transform of order $\kappa$ with respect to $\tau$ variable of $w_{k_1}^{\mathfrak{d}_p}$ and $w_{k_2}^{\mathfrak{d}_p}$. In Proposition~\ref{prop12}, we guarantee that both functions coincide in the intersection of their domain of definition. This would entail that the function 
$\mathcal{L}_{q;1/\kappa}^d(w^d_{k_1}(\tau,m,\epsilon))$ can be continued along direction $\mathfrak{d}_p$, with $q$-exponential growth of order $k_2$, see Propoposition~\ref{prop12}.

The construction of the analytic solution of (\ref{e1}), $u^{\mathfrak{d}_p}(t,z,\epsilon)$, is obtained after the application of $q$-Laplace transformation of order $k_2$ and inverse Fourier transform, providing a holomorphic function defined in $\mathcal{T}\times H_{\beta'}\times\mathcal{E}_{p}$, where $\mathcal{T}$ is some well chosen bounded sector centered at 0 and
$\{ \mathcal{E}_{p} \}_{0 \leq p \leq \varsigma-1}$ represents a good covering in
$\mathbb{C}^{\ast}$ (see Definition~\ref{defin111}). This result is described in Theorem~\ref{teo872}. The following diagram illustrates the procedure to follow.
\begin{figure}[h]
	\centering
		\includegraphics[width=0.90\textwidth]{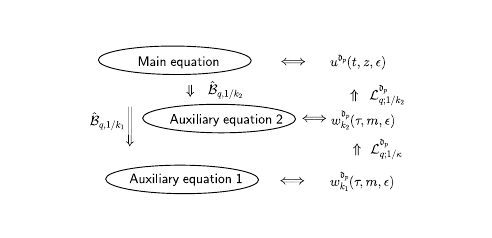}
		\caption{Scheme of the different Borel levels attained in the construction of the solution}
	\label{fig:a2}
\end{figure}
For the attainment of the asymptotic properties of the analytic solution we make use of a Ramis-Sibuya type theorem in two levels (see Theorem~\ref{teo1215}), and the properties held by the difference of two analytic solutions in the intersection of their domains, whenever it is not empty. The conclusion yields two different $q$-Gevrey levels of asymptotic behavior of the analytic solution with respect to the formal solution depending on the geometry of the problem. The final main result states the splitting of both, the formal and the analytic solutions to the problem under study, as a sum of three terms. More precisely, if $\mathbb{F}$ denotes the Banach space of holomorphic and bounded functions defined in $\mathcal{T}\times H_{\beta'}$, and $\hat{u}(t,z,\epsilon)$ stands for the formal power series solution of (\ref{e1}), then it holds that 
$$\hat{u}(t,z,\epsilon)=a(t,z,\epsilon)+\hat{u}_1(t,z,\epsilon)+\hat{u}_2(t,z,\epsilon),$$
where $a(t,z,\epsilon)\in\mathbb{F}\{\epsilon\}$ and $\hat{u}_1(t,z,\epsilon),\hat{u}_2(t,z,\epsilon)\in\mathbb{F}[[\epsilon]]$ and such that for every $0\le p\le \varsigma-1$, the function $u^{\mathfrak{d}_p}$ can be written in the form
$$u^{\mathfrak{d}_{p}}(t,z,\epsilon)=a(t,z,\epsilon)+u^{\mathfrak{d}_{p}}_1(t,z,\epsilon)+u^{\mathfrak{d}_{p}}_2(t,z,\epsilon),$$
where $\epsilon\mapsto u_1^{\mathfrak{d}_{p}}(t,z,\epsilon)$ is a $\mathbb{F}$-valued function that admits $\hat{u}_1(t,z,\epsilon)$ as its $q$-Gevrey asymptotic expansion of order $1/k_1$ on $\mathcal{E}_{p}$ and also $\epsilon\mapsto u_2^{\mathfrak{d}_{p}}(t,z,\epsilon)$ is a $\mathbb{F}$-valued function that admits $\hat{u}_2(t,z,\epsilon)$ as its $q$-Gevrey asymptotic expansion of order $1/k_2$ on $\mathcal{E}_{p}$. This corresponds to Theorem~\ref{teo1281}.
 
%One common drawback in this paper and in our last contribution~\cite{lamaq} compared to our prior study~\cite{ma} concerns the restriction on the coefficients $c_{l}(t,z,\epsilon)$ which are only allowed to depend polynomially on time. An analytic dependence would cause the appearance of convolution operators when studying the auxiliary equations in the Borel plane and are more delicated to handle. We postpone this for a future work.

%For the sake of clarity, we enclose the following table on the different elements involved in the construction of the solution of the problem.  
% $$\begin{array}{|l|l|l|l|}\hline
%\mathrm{Transformation}& \mathrm{Solution} & \mathrm{Equation} &  \mathrm{Domain \; of \; analyticity}\\ \hline
%&\hat{u}(t,z,\epsilon)& \eqref{epral}&\mathrm{Formal}\\\hline
%\mathcal{F} &\hat{U}(T,m,\epsilon)&\eqref{e59}& \mathrm{Formal}
%\\\hline
%\hat{\mathcal{B}}_{q;1/k_{1}}&w^{\mathfrak{d}_p}_{k_1}(\tau,m,\epsilon)&\eqref{e87}&\mathcal{R}_{\mathfrak{d}_p,\tilde{\delta}}\times\mathbb{R}\times D(0,\epsilon_0)\\\hline
%\mathcal{L}^{\mathfrak{d}_p}_{q;1/\kappa}&w_{k_2}^{\mathfrak{d}_p}(\tau,m,\epsilon)&\eqref{e412b}&S_{\mathfrak{d}_p}\times \mathbb{R}\times D(0,\epsilon_0)\\\hline
%\mathcal{L}^{\mathfrak{d}_p}_{q;1/k_{2}}&U^{\mathfrak{d}_p}(T,m,\epsilon)&\eqref{e59}&S_{\mathfrak{d}_p}\times\mathbb{R}\times D(0,\epsilon_0)\\\hline
%\mathcal{F}^{-1} &u^{\mathfrak{d}_{p}}(t,z,\epsilon)&\eqref{epral}&\mathcal{T}\times H_{\beta'}\times\mathcal{E}_p. 
%\\\hline 
%\end{array}  $$

The paper is organized as follows.

In Section \ref{sec21}, we define a weighted Banach space of continuous functions on the domain $(D(0,\rho)\cup U)\times \R$ with $q$-exponential growth on the unbounded sector $U$ with respect to the first variable, and exponential decay on $\R$ with respect to the second one. We study the continuity properties of several operators acting of this Banach space. Section \ref{sec22}. is concerned with the study of a second family of Banach spaces of functions with $q$-exponential growth on an infinite sector with respect to one variable and exponential decay on $\R$ with respect to the other variable. In Section \ref{sec3}, we recall the definitions and main properties of formal and analytic operators involved in the solution of the main equation. Namely, formal $q$-Borel transformation and an analytic $q$-Laplace transform of certain $q$-Gevrey orders, and inverse Fourier transform. In Section \ref{seccion41}. and Section \ref{seccion42}, we study the analytic solutions of two auxiliary problems in two different Borel planes and relate them via $q$-Laplace transformation (see Theorem~\ref{teo872}). In Section \ref{sec5}, we describe in detail the main problem under study, and construct its analytic solution and the rate of growth of the difference of two neighboring solutions in their common domain of definition. Finally, Section \ref{seccion6} deals with the existence of a formal solution of the problem, and studies the asymptotic behavior relating the analytic and the formal solutions through a multi-level $q$-Gevrey asymptotic expansion (Theorem~\ref{teo1281}). This result is attained with the application of a two-level $q$-version of Ramis-Sibuya theorem (Theorem~\ref{teo1215}).

\section{Auxiliary Banach spaces of functions}

In this section, we describe auxiliary Banach spaces of functions with certain growth and decay behavior. We also provide important properties of such spaces under certain operators.

Let $U_d$ be an open unbounded sector with vertex at the origin in $\mathbb{C}$, bisecting direction $d\in\mathbb{R}$ and positive opening. We take $\rho>0$ and consider $D(0,\rho)=\{\tau\in\mathbb{C}:|\tau|<\rho\}$.  

We fix real numbers $\beta,\mu,\delta>0$, $q>1$ and $\alpha\ge0$ through the whole section. We assume the distance from $U_d\cup D(0,\rho)$ to the real number $-\delta$ is strictly larger than 1. Let $k>0$. We denote $\overline{D}(0,\rho)$ the closure of $D(0,\rho)$.

The next definition of a Banach space of functions, and subsequent properties have already been studied in previous works. Analogous spaces were treated in~\cite{lamaq2,lm1}, inspired by the functional spaces appearing in~\cite{ra}. We refer the reader to~\cite{lamaq,ma} for some of the proofs of the following results, whose statements are included for the sake of completeness.

\subsection{First family of Banach spaces of functions with q-exponential growth and exponential decay}\label{sec21}

\begin{defin}\label{defi1}
Let $q>1$. We denote $\hbox{Exp}^{q}_{(k,\beta,\mu,\alpha,\rho)}$ the vector space of complex valued continuous functions $(\tau,m)\mapsto h(\tau,m)$ on $(U_d\cup\overline{D}(0,\rho))\times\mathbb{R}$, holomorphic with respect to $\tau$ on $U_d\cup D(0,\rho)$ and such that
$$\left\|h(\tau,m)\right\|_{(k,\beta,\mu,\alpha,\rho)}=\sup_{\substack{\tau\in U_d\cup\overline{D}(0,\rho),\\m\in\mathbb{R}}}(1+|m|)^{\mu}e^{\beta|m|}\exp\left(-\frac{k}{2}\frac{\log^2|\tau+\delta|}{\log(q)}-\alpha\log|\tau+\delta|\right)|h(\tau,m)|  $$
is finite. The space $(\hbox{Exp}^{q}_{(k,\beta,\mu,\alpha,\rho)},\left\|\cdot\right\|_{(k,\beta,\mu,\alpha,\rho)})$ is a Banach space.
\end{defin}

The proof of the following lemma is straightforward. 
\begin{lemma}\label{lema1}
Let $(\tau,m)\mapsto a(\tau,m)$ be a bounded continuous function on $(U_d\cup\overline{D}(0,\rho))\times\mathbb{R}$, holomorphic with respect to $\tau$ on $U_d\cup D(0,\rho)$. Then, it holds that
$$\left\|a(\tau,m)h(\tau,m)\right\|_{(k,\beta,\mu,\alpha,\rho)}\le\left(\sup_{\tau\in U_d\cup \overline{D}(0,\rho),m\in\mathbb{R}}|a(\tau,m)|\right)\left\|h(\tau,m)\right\|_{(k,\beta,\mu,\alpha,\rho)},$$
for every $h(\tau,m)\in \hbox{Exp}^{q}_{(k,\beta,\mu,\alpha,\rho)}$.
\end{lemma}

\begin{prop}\label{prop1}
Let $\gamma_1,\gamma_2,\gamma_3\ge0$ such that
$$
\gamma_1+k\gamma_3\ge\gamma_2.
$$
Let $a_{\gamma_1}(\tau)$ be a continuous function on $U_d\cup \overline{D}(0,\rho)$, holomorphic on $U_d\cup D(0,\rho)$, with
$$|a_{\gamma_1}(\tau)|\le\frac{1}{(1+|\tau|)^{\gamma_1}},$$
for every $\tau\in (U_d\cup \overline{D}(0,\rho))$. Then, there exists $C_1>0$, depending on $k,q,\alpha,\gamma_1,\gamma_2,\gamma_3$, such that
$$\left\|a_{\gamma_1}(\tau)\tau^{\gamma_2}\sigma_{q,\tau}^{-\gamma_3}f(\tau,m)\right\|_{(k,\beta,\mu,\alpha,\rho)}\le C_1  \left\|f(\tau,m)\right\|_{(k,\beta,\mu,\alpha,\rho)}$$
for every $f\in\hbox{Exp}^{q}_{(k,\beta,\mu,\alpha,\rho)}$.
\end{prop}

\begin{defin}
We write $E_{(\beta,\mu)}$ for the vector space of continuous functions $h:\mathbb{R}\to\mathbb{C}$ such that
$$\left\|h(m)\right\|_{(\beta,\mu)}=\sup_{m\in\mathbb{R}}(1+|m|)^{\mu}\exp(\beta|m|)|h(m)|<\infty.$$
 It holds that $(E_{(\beta,\mu)},\left\|\cdot\right\|_{(\beta,\mu)})$ is a Banach space.
\end{defin}

The Banach space $(E_{(\beta,\mu)},\left\|\cdot\right\|_{(\beta,\mu)})$ can be endowed with the structure of a Banach algebra with the following noncomutative product (see~Proposition 2 in \cite{ma} for further details).

\begin{prop}\label{prop2}
Let $Q(X), R(X)\in\mathbb{C}[X]$ be polynomials such that 
$$\deg(R)\ge\deg(Q),\qquad R(im)\neq0,$$
for all $m\in\mathbb{R}$. Let $m\mapsto b(m)$ be a continuous function in $\mathbb{R}$ such that
$$|b(m)|\le1/|R(im)|,\qquad m\in\R.$$
Assume that $\mu > \mathrm{deg}(Q)+1$. Then, there exists a
constant $C_{2}>0$ (depending on $Q(X),R(X),\mu$) such that
$$\left\|b(m)\int_{-\infty}^{+\infty}f(m-m_1)Q(im_1)g(m_1)dm_1\right\|_{(\beta,\mu)}\le C_2 \left\|f(m)\right\|_{(\beta,\mu)}\left\|g(m)\right\|_{(\beta,\mu)},$$
for every $f(m),g(m)\in E_{(\beta,\mu)}$. In the sequel, we adopt the notation
$$f(m)\star^{Q}g(m):=\int_{-\infty}^{+\infty}f(m-m_1)Q(im_1)g(m_1)dm_1,$$
for every $m\in\mathbb{R}$, extending the classical convolution product $\star$ for $Q\equiv 1$. As a result, $(E_{(\beta,\mu)},\left\|\cdot\right\|_{(\beta,\mu)})$ becomes a Banach algebra for the product $\star^{b,Q}$ defined by
$$f(m)\star^{b,Q}g(m):=b(m)f(m)\star^Qg(m).$$
\end{prop}

The next proposition is a slightly modified version of Proposition 3 in~\cite{ma}, adapted to the appearance of two different types of growth of the functions involved, which force holding some positive distance to the origin.
 
\begin{prop}\label{prop3}
Let $b(m),Q(X), R(X)$ be chosen as in Proposition~\ref{prop2}. We assume $1\le k\le \kappa$ is an integer. Let $c_h(m)\in E_{(\beta,\mu)}$ for $h\ge0$, such that
\begin{equation}\label{e300}
\left\|c_h\right\|_{(\beta,\mu)}\le C\left(\frac{1}{T}\right)^hq^{\frac{h^2}{2k}\left(1-\frac{\kappa}{k}\right)},\qquad h\ge0,
\end{equation}
for some $C>0$ and $T>q^{1/(2k)}$. Let $\varphi_{k}(\tau,m)$ be the power series
$$\varphi_{k}(\tau,m)=\sum_{h\ge0}c_h(m)\frac{\tau^h}{(q^{1/k})^{h(h-1)/2}}\in E_{(\beta,\mu)}[[\tau]],$$
which defines an entire function with respect to $\tau$ with values in $E_{(\beta,\mu)}$, in view of (\ref{e300}).

For every $f(\tau,m)\in\hbox{Exp}^{q}_{(\kappa,\beta,\mu,\alpha,\rho)}$, we define a $q$-analog of the convolution of order $k$ of $\varphi_k(\tau,m)$ and $f(\tau,m)$ as
$$\varphi_{k}(\tau,m)\star^{Q}_{q;1/k}f(\tau,m):=\sum_{h\ge0}\frac{\tau^h}{(q^{1/k})^{h(h-1)/2}}c_h(m)\star^{Q}(\sigma_{q,\tau}^{-\frac{h}{k}}f)(\tau,m).$$
Then, the function $b(m)\varphi_k(\tau,m)\star^{Q}_{q;1/k}f(\tau,m)$ belongs to $\hbox{Exp}^{q}_{(\kappa,\beta,\mu,\alpha,\rho)}$ and there exists $C_{3}>0$, depending on $\mu,q,\alpha,k,\kappa, Q(X),R(X),\delta$, $T$, such that
$$ \left\|b(m)\varphi_k(\tau,m)\star^{Q}_{q;1/k}f(\tau,m)\right\|_{(\kappa,\beta,\mu,\alpha,\rho)}\le C_3 C\left\|f(\tau,m)\right\|_{(\kappa,\beta,\mu,\alpha,\rho)}.$$
\end{prop}

\begin{proof} 
Let $f(\tau,m) \in \mathrm{Exp}_{(\kappa,\beta,\mu,\alpha,\rho)}^{q}$. From the very definition of the norm $\left\|\cdot\right\|_{(\kappa,\beta,\mu,\alpha,\rho)}$, we know that
\begin{multline*}
\left\|b(m)\varphi_{k}(\tau,m) \star_{q;1/k}^{Q} f(\tau,m)\right\|_{(\kappa,\beta,\mu,\alpha,\rho)}\\
= \sup_{\tau\in U_d\cup\overline{D}(0,\rho),m\in\mathbb{R}}
(1+|m|)^{\mu} e^{\beta |m|} \exp\left( - \frac{\kappa}{2} \frac{\log^{2}(|\tau+\delta|)}{\log(q)} - \alpha \log(|\tau+\delta|) \right)|b(m)|\\
\times \left| \sum_{h\ge 0} \frac{\tau^{h}}{(q^{1/k})^{\frac{h(h-1)}{2}}}\int_{-\infty}^{+\infty}
c_{h}(m-m_{1})Q(im_{1})f(\frac{\tau}{q^{h/k}},m_{1}) dm_{1} \right|=\sup_{\tau\in U_d\cup\overline{D}(0,\rho),m\in\mathbb{R}} L(\tau,m).
\end{multline*}

We first give upper estimates for 
$$\sup_{\tau\in \overline{D}(0,\rho),m\in\mathbb{R}}L(\tau,m).$$
By construction, there exist two constants $\mathfrak{Q},\mathfrak{R}>0$ such that
\begin{equation}
|Q(im_{1})| \leq \mathfrak{Q}(1 + |m_{1}|)^{\mathrm{deg}(Q)}, \ \  \ \
|R(im)| \geq \mathfrak{R}(1 + |m|)^{\mathrm{deg}(R)}, \label{bounds_Q_R}
\end{equation}
for all $m \in \mathbb{R}$. Using (\ref{bounds_Q_R}) and from Lemma 4 in \cite{ma2} (see also Lemma 2.2 from \cite{cota2}),
we get a constant $\tilde{C}_{2}>0$ with
\begin{multline}\label{bounds_int_b_Q_mu}
(1+|m|)^{\mu}|b(m)|\int_{-\infty}^{+\infty} \frac{|Q(im_{1})|}{(1 + |m-m_{1}|)^{\mu}(1+|m_{1}|)^{\mu}} dm_{1} \\ \leq
\sup_{m \in \mathbb{R}} \frac{\mathfrak{Q}}{\mathfrak{R}} (1 + |m|)^{\mu - \mathrm{deg}(R)}
\times \int_{-\infty}^{+\infty} \frac{1}{(1 + |m-m_{1}|)^{\mu} (1 + |m_{1}|)^{\mu - \mathrm{deg}(Q)}} dm_{1} \leq \tilde{C}_{2}
\end{multline}
provided that $\mu > \mathrm{deg}(Q)+1$.
From the definition of $\left\|f\right\|_{(\kappa,\beta,\mu,\alpha,\rho)}$ and $\left\|c_h\right\|_{(\beta,\mu)}$ for all $h\ge0$, (\ref{bounds_Q_R}), and (\ref{bounds_int_b_Q_mu}), we get that for every $\tau\in\overline{D}(0,\rho)$ and $m\in\R$, $L(\tau,m)$ is upper bounded by
\begin{multline*}
(1+|m|)^{\mu}e^{\beta|m|}\exp\left( - \frac{\kappa}{2} \frac{\log^{2}(|\tau+\delta|)}{\log(q)} - \alpha \log(|\tau+\delta|) \right)|b(m)|\\
 \times \sum_{h\ge0}\frac{|\tau|^h}{(q^{1/k})^{h(h-1)/2}} \int_{-\infty}^{+\infty} \left\|c_h\right\|_{(\beta,\mu)}\frac{1}{(1+|m-m_{1}|)^{\mu}} e^{-\beta|m-m_{1}|}|Q(im_1)|\\
 \frac{1}{(1+|m_1|)^{\mu}}e^{-\beta|m_1|} \exp\left( \frac{\kappa}{2} \frac{ \log^{2}(|\tau/q^{h/k}+\delta|) }{\log(q)} + \alpha \log( |\tau/q^{h/k}+\delta|) \right) dm_1 \left\|f(\tau,m)\right\|_{(\kappa,\beta,\mu,\alpha,\rho)}
\end{multline*}

\vspace{-1cm}

\begin{multline*}
\le  \tilde{C}_2\exp\left( - \frac{\kappa}{2} \frac{\log^{2}(|\tau+\delta|)}{\log(q)} - \alpha \log(|\tau+\delta|) \right) \sum_{h\ge0}\frac{|\tau|^h}{(q^{1/k})^{h(h-1)/2}}  \left\|c_h\right\|_{(\beta,\mu)}\\
 \times\exp\left( \frac{\kappa}{2} \frac{ \log^{2}(|\tau/q^{h/k}+\delta|) }{\log(q)} + \alpha \log( |\tau/q^{h/k}+\delta|) \right)\left\|f(\tau,m)\right\|_{(\kappa,\beta,\mu,\alpha,\rho)}\\
\le  \hat{C}_2 \sum_{h\ge0}\frac{\rho^h}{(q^{1/k})^{h(h-1)/2}}  \left\|c_h\right\|_{(\beta,\mu)} \left\|f(\tau,m)\right\|_{(\kappa,\beta,\mu,\alpha,\rho)},
\end{multline*}
with $\hat{C}_2=\tilde{C}_2\exp\left( \frac{\kappa}{2} \frac{ \log^{2}(\rho+\delta) }{\log(q)} + \alpha \log( \rho+\delta) \right)$. The assumption (\ref{e300}) on $\left\|c_h\right\|_{(\beta,\mu)}$ allows to conclude the result when restricting the domain on the variable $\tau$ to the subset $\overline{D}(0,\rho)$.

%due to explicit estimates on $\bar{D}(0,\rho)$ for
%$b(m)\varphi_{k}(\tau,m)\star_{q;1/k}^{Q}f(\tau,m)$ by the quantity
%$\sup_{\tau \in \bar{D}(0,\rho)} (1 + |m|)^{\mu}..|f(\tau,m)|$.
%The definition of the norm $\left\|\cdot\right\|_{(\kappa,\beta,\mu,\alpha,\rho}$ yields a bound in $\overline{D}(0,\rho)$. 
Let $\widetilde{U}_{d}$ be the complementary of $\overline{D}(0,\rho)$ in $U_{d}$. From what precedes we may take the supremum over $\widetilde{U}_{d}$ instead of $U_d\cup\overline{D}(0,\rho)$.

By inserting terms that correspond to the $||.||_{(\beta,\mu)}$ norm of $c_{h}(m)$ and to the
$||.||_{(\kappa,\beta,\mu,\alpha,\rho)}$ norm of $f(\tau/q^{h/k},m)$, we can give the bound estimates
\begin{multline*}
%||b(m)\varphi_{k}(\tau,m) \star_{q;1/k}^{Q} f(\tau,m)||_{(\kappa,\beta,\mu,\alpha,\rho)} \\
\sup_{\tau\in\tilde{U}_d,m\in\R}L(\tau,m)\\
\leq \sup_{\tau\in \widetilde{U}_{d},m\in\mathbb{R}}
(1+|m|)^{\mu} e^{\beta |m|} \exp\left( - \frac{\kappa}{2} \frac{\log^{2}(|\tau+\delta|)}{\log(q)} - \alpha \log(|\tau+\delta|) \right)|b(m)|\\
\times \sum_{h\ge0} \int_{-\infty}^{+\infty} \left( (1+|m-m_{1}|)^{\mu} e^{\beta|m-m_{1}|}
\frac{|c_{h}(m-m_{1})|}{(q^{1/k})^{h(h-1)/2}} |\tau|^{h} \right)\\
\times \left( |f(\frac{\tau}{q^{h/k}},m_{1})|
(1+|m_{1}|)^{\mu} e^{\beta|m_{1}|} \exp\left( -\frac{\kappa}{2} \frac{ \log^{2}(|\tau/q^{h/k}+\delta|) }{\log(q)} - \alpha \log( |\tau/q^{h/k}+\delta|) \right)  \right)\\
\times \left( \frac{e^{-\beta|m-m_{1}|}}{(1 + |m-m_{1}|)^{\mu}} \frac{|Q(im_{1})|}{(1+|m_{1}|)^{\mu}} e^{-\beta|m_{1}|}
\exp\left( \frac{\kappa}{2} \frac{ \log^{2}(|\tau/q^{h/k}+\delta|) }{\log(q)} + \alpha \log( |\tau/q^{h/k}+\delta|) \right)  \right) dm_{1}.
\end{multline*}

By means of the triangular inequality $|m| \leq |m-m_{1}| + |m_{1}|$, we deduce that
\begin{equation}
\sup_{\tau\in\tilde{U}_d,m\in\R}L(\tau,m)\le \hat{C} ||f(\tau,m)||_{(\kappa,\beta,\mu,\alpha,\rho)} \label{norm_q_conv_f<norm_f_q_exp}
\end{equation}
where
\begin{multline}\label{defin_hat_C_3}
\hat{C} =\sup_{\tau \in \widetilde{U}_{d},m \in \mathbb{R}}
(1+|m|)^{\mu} \exp\left( - \frac{\kappa}{2} \frac{\log^{2}(|\tau+\delta|)}{\log(q)} - \alpha \log(|\tau+\delta|) \right)|b(m)|\\
\times \sum_{h\ge0} ||c_{h}||_{(\beta,\mu)} \frac{|\tau|^{h}}{(q^{1/k})^{h(h-1)/2}}
\int_{-\infty}^{+\infty} \frac{|Q(im_{1})|}{(1 + |m-m_{1}|)^{\mu}(1+|m_{1}|)^{\mu}} dm_{1}\\
\times \exp\left( \frac{\kappa}{2} \frac{ \log^{2}(|\tau/q^{h/k}+\delta|) }{\log(q)} + \alpha \log( |\tau/q^{h/k}+\delta| )\right). 
\end{multline}
Again, we can also apply (\ref{bounds_int_b_Q_mu}) at this point. On the other hand, we can provide upper estimates on the following expression
\begin{multline}
\exp\left(\frac{\kappa}{2\log(q)}\left(\log^2\left(\left|\frac{\tau}{q^{h/k}}+\delta\right|\right)-\log^2\left(\left|\frac{\tau}{q^{h/k}}\right|\right)\right)\right)\\
\times\exp\left(\frac{\kappa}{2\log(q)}\left(\log^2\left(\left|\frac{\tau}{q^{h/k}}\right|\right)-\log^2\left(|\tau|\right)\right)\right)\\
\times\exp\left(\frac{\kappa}{2\log(q)}\left(\log^2\left(|\tau|\right)-\log^2\left(|\tau+\delta|\right)\right)\right)\\
\times\exp\left(\alpha\log\left(\left|\frac{\tau}{q^{h/k}}\right|\right)-\alpha\log(|\tau|)\right)\\
\times\exp\left(\alpha\log\left(\left|\frac{\tau}{q^{h/k}}+\delta\right|\right)-\alpha\log\left(|\frac{\tau}{q^{h/k}}|\right)\right)
\times\exp\left(\alpha\log(|\tau|)-\alpha\log(|\tau+\delta|)\right)\label{e365}\\
\end{multline}
as follows. The proof of Proposition 3~\cite{ma} can be applied to the second and fourth lines of (\ref{e365}) which yield
\begin{equation}\label{e366}
\exp\left(\frac{\kappa}{2\log(q)}\left(\log^2\left(\left|\frac{\tau}{q^{h/k}}\right|\right)-\log^2\left(|\tau|\right)\right)\right)= q^{\frac{h^2\kappa}{2k^2}}|\tau|^{-\frac{h\kappa}{k}},
\end{equation}
and
\begin{equation}\label{e367}
\exp\left(\alpha\log\left(\left|\frac{\tau}{q^{h/k}}\right|\right)-\alpha\log(|\tau|)\right)=q^{-\frac{\alpha h}{k}},
\end{equation}
respectively. It is straightforward to check that the expression in the fifth line of (\ref{e365}) is upper bounded by
\begin{equation}\label{e381}
C_{31}(q^{h/k})^{\alpha},
\end{equation}
for some positive constant $C_{31}$. 

We give upper bounds for the first line in (\ref{e365}). In the case that $|\tau/q^{h/k}|\le 1$, this expression is upper bounded by a constant which does not depend on $\tau$ nor $h$. Otherwise, we have
\begin{multline}
\exp\left(\frac{\kappa}{2\log(q)}\left(\log^2\left(\left|\frac{\tau}{q^{h/k}}+\delta\right|\right)-\log^2\left(\left|\frac{\tau}{q^{h/k}}\right|\right)\right)\right)\\
\le \exp\left(\frac{\kappa}{2\log(q)}\left(\log^2\left(\left|\frac{\tau}{q^{h/k}}\right|+\delta\right)-\log^2\left(\left|\frac{\tau}{q^{h/k}}\right|\right)\right)\right)\\
\le\sup_{x>1}\exp(\frac{\kappa}{2\log(q)}(\log^2(x+\delta))-\log^2(x))\le C_{32},\label{e390}
\end{multline}
for some $C_{32}>0$.
We finally provide upper bounds on the third line of (\ref{e365}). Taking into account that
\begin{equation}\label{e382}
\log^2|\tau|-\log^2|\tau+\delta|=-\log^2|1+\frac{\delta}{\tau}|-2\log|\tau|\log|1+\frac{\delta}{\tau}|\le C_{33},\quad \tau\in \widetilde{U}_{d}
\end{equation}
for some $C_{33}>0$.
From (\ref{defin_hat_C_3}), (\ref{bounds_int_b_Q_mu}), (\ref{e365}), (\ref{e366}), (\ref{e367}), (\ref{e381}), (\ref{e390}), and (\ref{e382}) we derive the existence of $\tilde{C}_{31},\tilde{C}_{32}>0$ such that
\begin{align*}
\hat{C}&\le \tilde{C}_{31} \sum_{h\ge0} ||c_{h}||_{(\beta,\mu)} \left(\sup_{\tau \in \widetilde{U}_{d}}|\tau|^{h(1-\frac{\kappa}{k})}\right)q^{\frac{h}{2k}+\frac{h^2}{2k}\left(\frac{\kappa}{k}-1\right)}\\
&\le \tilde{C}_{32} C\sum_{h\ge0}\left(\frac{q^{1/(2k)}}{T}\right)^{h}\le C_3C,
\end{align*}
which yields the result, when choosing $T>q^{1/(2k)}$.
\end{proof}

\textbf{Remark:} Observe that condition (\ref{e300}) on the coefficients $c_h$ is always satisfied in the case that only a finite number of $c_h$ is not identically zero, i.e. $\varphi_k\in E_{(\beta,\mu)}[\tau]$.

\subsection{Second family of Banach spaces of functions with q-exponential growth and exponential decay}\label{sec22}

The second family of auxiliary Banach spaces has already been studied in previous works, such as~\cite{lamaq,ma}. We refer to these references for the proofs of the related results.

Let $S_{d}$ be an infinite sector of bisecting direction $d$ and   
let $\nu\in\mathbb{R}$.

\begin{defin}
We write $\hbox{Exp}_{(k,\beta,\mu,\nu)}^q$ for the vector space of continuous functions ${(\tau,m)\mapsto h(\tau,m)}$ on $S_d\times\mathbb{R}$, and holomorphic with respect to $\tau$ on $S_d$, such that
$$\left\|h(\tau,m)\right\|_{(k,\beta,\mu,\nu)}=\sup_{\tau\in S_d,m\in\mathbb{R}}(1+|m|)^{\mu}e^{\beta|m|}\exp\left(-\frac{k\log^2|\tau|}{2\log(q)}-\nu\log|\tau|\right)|h(\tau,m)|$$
is finite. It holds that $(\hbox{Exp}_{(k,\beta,\mu,\nu)}^q,\left\|\cdot\right\|_{(k,\beta,\mu,\nu)})$ is a Banach space.
\end{defin}

\begin{rmk} Let $0\le \kappa_1\le\kappa_2$. For every $f\in \hbox{Exp}_{(\kappa_1,\beta,\mu,\nu)}^q$, it holds that $f\in \hbox{Exp}_{(\kappa_2,\beta,\mu,\nu)}^q$, and
$$\left\|f(\tau,m)\right\|_{(\kappa_2,\beta,\mu,\nu)}\leq \left\|f(\tau,m)\right\|_{(\kappa_1,\beta,\mu,\nu)}.$$
\end{rmk}

The proof of the following lemma is a straightforward consequence of the definition. 
\begin{lemma}\label{lema2}
Let $a(\tau,m)$ be a bounded continuous function on $S_d\times\mathbb{R}$, holomorphic on $S_d$ with respect to $\tau$. Then,
$$\left\|a(\tau,m)f(\tau,m)\right\|_{(k,\beta,\mu,\nu)}\le\sup_{\tau\in S_d,m\in\mathbb{R}}|a(\tau,m)|\left\|f(\tau,m)\right\|_{(k,\beta,\mu,\nu)}$$
for every $f(\tau,m)\in\hbox{Exp}_{(k,\beta,\mu,\nu)}^q$.
\end{lemma}

\begin{prop}\label{prop4} 
Let $\gamma_1,\gamma_2\ge0$ and $\gamma_3\in\mathbb{R}$ such that 
\begin{equation}\label{e185}
\gamma_1+k\gamma_3\ge\gamma_2.
\end{equation}
Let $a_{\gamma_1}(\tau)$ be holomorphic on $S_d$, with 
$$|a_{\gamma_1}(\tau)|\le\frac{1}{(1+|\tau|)^{\gamma_1}},\quad \tau\in S_d.$$
Then, there exists $C_4>0$, depending on $k,q,\nu,\gamma_1,\gamma_2,\gamma_3$ such that
$$\left\|a_{\gamma_1}(\tau)\tau^{\gamma_2}\sigma_{q,\tau}^{-\gamma_3}f(\tau,m)\right\|_{(k,\beta,\mu,\nu)}\le C_4  \left\|f(\tau,m)\right\|_{(k,\beta,\mu,\nu)}$$
for every $f\in\hbox{Exp}^{q}_{(k,\beta,\mu,\nu)}$.
\end{prop}
\begin{proof}
For every $f\in\hbox{Exp}^{q}_{(k,\beta,\mu,\nu)}$ we have 

\begin{multline*}
\left\|a_{\gamma_1}(\tau)\tau^{\gamma_2}\sigma_{q,\tau}^{-\gamma_3}f(\tau,m)\right\|_{(k,\beta,\mu,\nu)}\\
\le\sup_{\tau\in S_d,m\in\mathbb{R}}(1+|m|)^{\mu}e^{\beta|m|}\exp\left(-\frac{k\log^2|\tau|}{2\log(q)}-\nu\log|\tau|\right)\frac{|\tau|^{\gamma_2}}{(1+|\tau|)^{\gamma_1}}|f(\tau/q^{\gamma_3},m)|\\
\times \exp\left(-\frac{k\log^2|\tau/ q^{\gamma_3}|}{2\log(q)}-\nu\log|\tau /q^{\gamma_3}|\right)\exp\left(\frac{k\log^2|\tau /q^{\gamma_3}|}{2\log(q)}+\nu\log|\tau/ q^{\gamma_3}|\right)\\
\le \sup_{\tau\in S_d}q^{\frac{k\gamma_3^2-2\nu\gamma_3}{2}}\frac{|\tau|^{-k\gamma_3+\gamma_2}}{(1+|\tau|)^{\gamma_1}}  \left\|f(\tau,m)\right\|_{(k,\beta,\mu,\nu)}.
\end{multline*}
The result follows from the condition (\ref{e185}).
\end{proof}

Following the same lines of arguments as in Proposition 2.6, we deduce the next proposition
\begin{prop}\label{prop3b}
Let $b(m),Q(X), R(X), c_{h}$ for $h\ge0$ and $\varphi_k(\tau,m)$ be chosen as in Proposition~\ref{prop3}. For every $f(\tau,m)\in\hbox{Exp}^q_{(k,\beta,\mu,\nu)}$, it holds that $b(m)\varphi_k(\tau,m)\star^{Q}_{q;1/k}f(\tau,m)$ belongs to $\hbox{Exp}^q_{(k,\beta,\mu,\nu)}$ and there exists $C_{4}>0$, depending on $\mu,q,\nu,k,Q(X),R(X),$ such that
$$ \left\|b(m)\varphi_k(\tau,m)\star^{Q}_{q;1/k}f(\tau,m)\right\|_{(k,\beta,\mu,\nu)}\le C C_4\left\|f(\tau,m)\right\|_{(k,\beta,\mu,\nu)}.$$
\end{prop}

\section{Formal and analytic operators involved in the study of the problem}\label{sec3}

The main properties of some formal and analytic transformations are displayed for the sake of completeness. In this section, $\mathbb{E}$ stands for a complex Banach space.

The definition and main properties of the $q$-analog of Borel and Laplace transformation in several different orders can be found in~\cite{dreyfus,ra}. The proofs of the following results can be found in~\cite{ma}.

Let $q>1$ be a real number, and $k\ge1$ be a rational number. 
\begin{defin}
For every $\hat{a}(T)=\sum_{n\ge0}a_nT^n\in\mathbb{E}[[T]]$ we define the formal $q$-Borel transform of order $k$ of $\hat{a}(T)$ by
$$\hat{\mathcal{B}}_{q;1/k}(\hat{a}(T))(\tau)=\sum_{n\ge0}a_n\frac{\tau^n}{(q^{1/k})^{n(n-1)/2}}\in\mathbb{E}[[\tau]].$$
\end{defin}

\begin{prop}
Let $\sigma\in\mathbb{N}$ and $j\in\mathbb{Q}$. Then, it holds
$$\hat{\mathcal{B}}_{q;1/k}(T^{\sigma}\sigma_{q}^j\hat{a}(T))(\tau)=\frac{\tau^{\sigma}}{(q^{1/k})^{\sigma(\sigma-1)/2}}\sigma_{q}^{j-\frac{\sigma}{k}}\left(\hat{\mathcal{B}}_{q;1/k}(\hat{a}(T))(\tau)\right) ,$$
for every $\hat{a}(T)\in\mathbb{E}[[T]]$.
\end{prop}

The $q$-analog of Laplace transformation as it is shown was developed in~\cite{viziozhang}. The associated kernel of such transformation is the Jacobi theta function of order $k$ defined by
$$\Theta_{q^{1/k}}(x)=\sum_{n\in\mathbb{Z}}q^{-\frac{n(n-1)}{2k}}x^n,$$
which turns out to be a holomorphic function in $\mathbb{C}^{\star}$. It turns out to be a solution of the $q$-difference equation
\begin{equation}\label{e245}
\Theta_{q^{1/k}}(q^{\frac{m}{k}}x)=q^{\frac{m(m+1)}{2k}}x^{m}\Theta_{q^{1/k}}(x),
\end{equation}
for every $m\in\mathbb{Z}$, valid for all $x\in\mathbb{C}^{\star}$. As a matter of fact, the inverse of the Jacobi theta function of order $k$ is a function of $q$-Gevrey decrease of order $k$ at infinity in the sense that for every $\tilde{\delta}>0$ there exists $C_{q,k}>0$, not depending on $\tilde{\delta}$, such that
\begin{equation}\label{eq3}
\left|\Theta_{q^{1/k}}(x)\right|\ge C_{q,k}\tilde{\delta}\exp\left(\frac{k}{2}\frac{\log^2|x|}{\log(q)}\right)|x|^{1/2},
\end{equation}
for $x\in\mathbb{C}^{\star}$ under the condition that $|1+xq^{\frac{m}{k}}|>\tilde{\delta}$, for every $m\in\mathbb{Z}$. 

\begin{defin}\label{defi5}
Let $\rho>0$ and $U_d$ be an unbounded sector with vertex at $0$ and bisecting direction $d\in\mathbb{R}$. Let $f:D(0,\rho)\cup U_d\to\mathbb{E}$ be a holomorphic function, continuous on $\overline{D}(0,\rho)$, such that there exist $K>0$ and $\alpha\in\mathbb{R}$ with
$$\left\|f(x)\right\|_{\mathbb{E}}\le K\exp\left(\frac{k\log^2|x|}{2\log(q)}+\alpha\log|x|\right),\qquad x\in U_d,\quad |x|\ge \rho,$$
and
$$\left\|f(x)\right\|_{\mathbb{E}}\le K,\qquad x\in\overline{D}(0,\rho).$$
Set $\pi_{q^{1/k}}=\frac{\log(q)}{k}\prod_{n\ge0}(1-\frac{1}{q^{(n+1)/k}})^{-1}$. We define the $q$-Laplace transform of order $k$ of $f$ along direction $d$ by
$$\mathcal{L}^{d}_{q;1/k}(f(x))(T)=\frac{1}{\pi_{q^{1/k}}}\int_{L_{d}}\frac{f(u)}{\Theta_{q^{1/k}}\left(\frac{u}{T}\right)}\frac{du}{u},$$
where $L_{d}:=\{te^{id}:t\in(0,\infty)\}$.
\end{defin}

We refer the reader to Lemma 4 and Proposition 6 in~\cite{ma} for the proof of the next results. The algebraic property held by $q$-Laplace transformation would allow to commute some operators with respect to it.

\begin{lemma}\label{lem3}
Let $\tilde{\delta}>0$. Under the hypotheses of Definition~\ref{defi5}, we have that $\mathcal{L}_{q;1/k}^{d}(f(x))(T)$ defines a bounded and holomorphic function on $\mathcal{R}_{d,\tilde{\delta}}\cap D(0,r_1)$ for every $0<r_1\le q^{\left(\frac{1}{2}-\alpha\right)/k}/2$, where
\begin{equation}\label{e527}\mathcal{R}_{d,\tilde{\delta}}:=\left\{T\in\mathbb{C}^{\star}:\left|1+\frac{re^{id}}{T}\right|>\tilde{\delta},\hbox{ for all }r\ge0\right\}.
\end{equation}
A different choice for $d$ modulo $2\pi \Z$ would provide the same function due to Cauchy formula.
\end{lemma}

\begin{prop} Let $f$ be a function which satisfies the properties in Definition~\ref{defi5}, and let $\tilde{\delta}>0$. Then, for every $\sigma\ge0$ one has
$$T^{\sigma}\sigma_q^j(\mathcal{L}_{q;1/k}^{d}f(x))(T)=\mathcal{L}_{q;1/k}^{d}\left(\frac{x^{\sigma}}{(q^{1/k})^{\sigma(\sigma-1)/2}}\sigma_{q}^{j-\frac{\sigma}{k}}f(x)\right)(T),$$
for every $T\in\mathcal{R}_{d,\tilde{\delta}}\cap D(0,r_1)$, with $0<r_1\le q^{\left(\frac{1}{2}-\alpha\right)/k}/2$.
\end{prop}

Another operator which is used through the work is the inverse Fourier transform.

\begin{prop}\label{prop490}
Let $f \in E_{(\beta,\mu)}$ with $\beta > 0$, $\mu > 1$. The inverse Fourier transform of $f$ is defined by
$$ \mathcal{F}^{-1}(f)(x) = \frac{1}{ (2\pi)^{1/2} } \int_{-\infty}^{\infty} f(m) \exp( ixm ) dm $$
for all $x \in \mathbb{R}$. The function $\mathcal{F}^{-1}(f)$ extends to an analytic function on the strip
$$H_{\beta} = \{ z \in \mathbb{C} / |\mathrm{Im}(z)| < \beta \}. $$
Let $\phi(m) = im f(m) \in E_{(\beta,\mu - 1)}$. Then, we have
$$\partial_{z} \mathcal{F}^{-1}(f)(z) = \mathcal{F}^{-1}(\phi)(z)$$
for all $z \in H_{\beta}$.

Let $g \in E_{(\beta,\mu)}$ and let $\psi(m) = \frac{1}{(2\pi)^{1/2}}f \star g(m)$, the convolution product of $f$ and $g$, for all $m \in \mathbb{R}$.

From Proposition~\ref{prop2}, we know that $\psi \in E_{(\beta,\mu)}$. Moreover, we have
$$\mathcal{F}^{-1}(f)(z)\mathcal{F}^{-1}(g)(z) = \mathcal{F}^{-1}(\psi)(z)$$
for all $z \in H_{\beta}$.
\end{prop}

\section{Formal and analytic solutions to some auxiliary convolution initial value problems with complex parameters}\label{seccion4}

Let $1\le k_1<k_2$ and $D,D_1,D_2\ge 3$ be integers and define $\kappa^{-1}=k_{1}^{-1}-k_{2}^{-1}$. Observe that $\kappa>k_1$. Let $q>1$ be a real number. We also consider the positive integer numbers $d_{D_1},d_{D_2}$. For every $1\le \ell\le D-1$ we consider nonnegative integers $d_\ell,\delta_\ell\ge 1$ and $\Delta_\ell\ge0$. We assume that
\begin{equation}\label{e107}
\delta_1=1,\qquad \delta_\ell<\delta_{\ell+1},
\end{equation}
for $1\le \ell\le D-2$. We also assume that
\begin{equation}\label{e108a}
\Delta_\ell\ge d_\ell,\quad \frac{d_{D_1}-1}{\kappa}+\frac{d_\ell}{k_2}+1\ge\delta_\ell,\quad \frac{d_\ell}{k_1}+1\ge \delta_\ell,\quad  \frac{d_{D_2}-1}{k_2}\ge\delta_\ell-1,
\end{equation}
for every $1\le \ell\le D-1$, and also
\begin{equation}\label{e237}
k_1(d_{D_2}-1)>k_2d_{D_1}.
\end{equation}

Let $Q(X),R_{\ell}(X)\in\mathbb{C}[X]$ for $1\le \ell\le D-1$ and $R_{D_1},R_{D_2}\in\mathbb{C}[X]$, such that

\begin{equation}\label{e115b}
\deg(R_{D_2})= \deg(R_{D_1}),
\end{equation}
and 
\begin{equation}\label{e115a}
\deg(Q)\ge \deg(R_{D_j})\ge \deg(R_\ell), \quad \mu-1 > \deg(R_{D_j}), \quad Q(im)\neq0,\quad R_{D_j}(im)\neq0,
\end{equation}
for some $\mu>\deg(R_{D_j})+1$ with $j=1,2$, for all $m\in\mathbb{R}$, $1\le \ell\le D-1$.

We consider sequences of functions $m\mapsto F_n(m,\epsilon)$ and $m\mapsto C_{\ell,n}(m,\epsilon)$ for $n\ge0$ belonging to the Banach space $E_{(\beta,\mu)}$ for some $\beta>0$, depending holomorphically on $\epsilon\in D(0,\epsilon_0)$, for some $\epsilon_0>0$. %We also consider $m\mapsto C_{\ell}(T,m,\epsilon)=\sum_{j=0}^{p_1}C_{\ell,j}(m,\epsilon)T^j$, for $1\le \ell\le D-1$ in $E_{(\beta,\mu)}[T]$, depending holomorphically on $\epsilon\in D(0,\epsilon_0)$. 
We moreover assume there exist $\tilde{C}_{\ell},C_F,T_0>0$ such that
\begin{multline}\label{e119}
\left\|C_{\ell,n}\right\|_{(\beta,\mu)}\le \tilde{C}_\ell\left(\frac{1}{T_0}\right)^nq^{-\frac{n^2\kappa}{2k_1k_2}},\\
\left\|F_{n}\right\|_{(\beta,\mu)}\le C_F\left(\frac{1}{T_0}\right)^n,\qquad n\ge0.
\end{multline}

 We define the formal power series in $E_{(\beta,\mu)}[[T]]$ 
$$\hat{C}_{\ell}(T,m,\epsilon)=\sum_{n\ge 0}C_{\ell,n}(m,\epsilon)T^n,\qquad\hat{F}(T,m,\epsilon)=\sum_{n\ge 0}F_n(m,\epsilon)T^n.$$

 We consider the following initial value problem

\begin{multline}\label{e59}
Q(im)\sigma_{q,T}U(T,m,\epsilon)\\=T^{d_{D_1}}\sigma_{q,T}^{\frac{d_{D_1}}{k_1}+1}R_{D_1}(im)U(T,m,\epsilon)+T^{d_{D_2}}\sigma_{q,T}^{\frac{d_{D_2}}{k_2}+1}R_{D_2}(im)U(T,m,\epsilon)\\
+\sum_{\ell=1}^{D-1}\epsilon^{\Delta_\ell-d_\ell}T^{d_{\ell}}\sigma_{q,T}^{\delta_{\ell}}\left(\frac{1}{(2\pi)^{1/2}}\int_{-\infty}^{+\infty}\hat{C}_{\ell}(T,m-m_1,\epsilon)R_\ell(im_1)U(T,m_1,\epsilon)dm_1\right)\\
+\sigma_{q,T}\hat{F}(T,m,\epsilon).
\end{multline}

\begin{prop}
There exists a unique formal power series 
\begin{equation}\label{e64}
\hat{U}(T,m,\epsilon)=\sum_{n\ge0}U_n(m,\epsilon)T^n,
\end{equation}
solution of (\ref{e59}),  where the coefficients $U_{n}(m,\epsilon)$ belong to $E_{(\beta,\mu)}$, for $\beta>0$ and
$\mu > \mathrm{deg}(R_{D_{j}})+1$, $j\in \{1,2\}$, given above and depend holomorphically on $\epsilon \in D(0,\epsilon_{0})$.
\end{prop}
\begin{proof}
We plug the formal power series (\ref{e64}) into equation (\ref{e59}) to obtain a recursion formula for the coefficients $U_n$, for $n\ge0$. We have
\begin{multline*}
Q(im)U_n(m,\epsilon)q^n =\\R_{D_1}(im)U_{n-d_{D_1}}(m,\epsilon)q^{\left(\frac{d_{D_1}}{k_1}+1\right)(n-d_{D_1})}+R_{D_2}(im)U_{n-d_{D_2}}(m,\epsilon)q^{\left(\frac{d_{D_2}}{k_2}+1\right)(n-d_{D_2})}\\
+\sum_{\ell=1}^{D-1}\epsilon^{\Delta_\ell-d_\ell}q^{(n-d_{\ell})\delta_\ell}\left(\sum_{n_1+n_2=n-d_\ell}\frac{1}{(2\pi)^{1/2}}\int_{-\infty}^{+\infty}C_{\ell,n_1}(m-m_1,\epsilon)R_\ell(im_1)U_{n_2}(m_1,\epsilon)dm_1\right)\\
+F_n(m,\epsilon)q^n
\end{multline*} 
for every $n\ge \max\{d_{D_1},d_{D_2},\max_{1\le \ell\le D-1}d_{\ell}\}$. % Here $C_{\ell,n_1}\equiv 0$ in the case that $n_1\ge p_1$. 
Due to $C_{\ell,n},F_n\in E_{(\beta,\mu)}$ for every $n\ge0$ and $1\le \ell\le D-1$, we get $U_n\in E_{(\beta,\mu)}$ by recursion. We observe that Proposition 2.5. is applied in the recursion.
\end{proof}

\subsection{Analytic solutions of a first auxiliary problem in the $q$-Borel plane}\label{seccion41}

We proceed to multiply at both sides of equation (\ref{e59}) by $T^{k_1}$ and then apply the formal $q$-Borel transformation of order $k_1$ with respect to $T$. Let $\varphi_{k_1,\ell}(\tau,m,\epsilon)$ be the formal $q$-Borel transform of order $k_1$ of $\hat{C}_{\ell}(T,m,\epsilon)$ with respect to $T$, and $\Psi_{k_1}(\tau,m,\epsilon)$ the formal $q$-Borel transform of order $k_1$ of $\hat{F}(T,m,\epsilon)$ with respect to $T$. More precisely, we have
\begin{align}\label{e250}
\varphi_{k_1,\ell}(\tau,m,\epsilon)=\sum_{n\ge0}C_{\ell,n}(m,\epsilon)\frac{\tau^n}{(q^{1/k_1})^{n(n-1)/2}},\nonumber\\ \Psi_{k_1}(\tau,m,\epsilon)=\sum_{n\ge0}F_{n}(m,\epsilon)\frac{\tau^n}{(q^{1/k_1})^{n(n-1)/2}}.
\end{align}

%Assume that $\Psi_{k_1}(\tau,m,\epsilon)$ above, which defines a holomorphic function on $D(0,T_0)$, due to (\ref{e119}), with values in the Banach space $E_{(\beta,\mu)}$ can be analytically prolonged along direction $d\in\mathbb{R}$ in the sector $U_d$, with $q$-exponential growth of order $\kappa$. More precisely, we have
%\begin{equation}\label{e197}
%\left\|\Psi_{k_1}(\tau,m,\epsilon)\right\|_{(\kappa,\beta,\mu,\alpha,\rho)}= C_{\Psi_{k_1}}<\infty,
%\end{equation}
%for some  $C_{\Psi_{k_1}}>0$. Observe from (\ref{e119}) that $C_F$ tends to zero when $C_{\Psi_{k_1}}$ does. We also assume that $\Psi_{k_1}$ can not be prolongued to an entire function. More precisely, we assume that for every positive constants $\tilde{C}_F$ and $\tilde{T}_0$ there exists $n_0\ge0$ such that for every $n\ge n_0$ one has
%\begin{equation}\label{e332}
%\tilde{C}_F\left(\frac{1}{\tilde{T}_0}\right)^{n}(q^{1/k_1})^{n(n-1)/2}<\left\|F_n(m,\epsilon)\right\|_{(\beta,\mu)},
%\end{equation}
%for $\epsilon\in D(0,\epsilon_0)$. 
According to Lemma 5 of~\cite{ma}, the expression
$\Psi_{k_1}(\tau,m,\epsilon)$ represents an entire function of $q$-exponential growth of
order $k_1$ that belongs to the Banach space
$\mathrm{Exp}_{(\kappa,\beta,\mu,\alpha,\rho)}^{q}$ since $\kappa > k_{1}$, provided that
$\alpha$ satisfies $T_{0} > q^{\frac{1}{2k_{1}}}/q^{\alpha/k_{1}}$, for any
unbounded sector $U_{d}$ and any disc $D(0,\rho)$. More precisely, we have
\begin{equation}\label{e01}
||\Psi_{k_1}(\tau,m,\epsilon)||_{(\kappa,\beta,\mu,\alpha,\rho)} \leq C_{\Psi_{k_1}}
\end{equation}
for some constant $C_{\Psi_{k_1}} > 0$, for all $\epsilon \in D(0,\epsilon_{0})$.

In view of the properties of the $q$-Borel transformation of order $k_1$, we arrive at the equation
\begin{multline}\label{e87}
Q(im)\frac{\tau^{k_1}}{(q^{1/k_1})^{\frac{k_1(k_1-1)}{2}}}w_{k_1}(\tau,m,\epsilon) 
=R_{D_1}(im)\frac{\tau^{d_{D_1}+k_1}}{(q^{1/k_1})^{\frac{(d_{D_1}+k_1)(d_{D_1}+k_1-1)}{2}}}w_{k_1}(\tau,m,\epsilon) \\
+R_{D_2}(im)\frac{\tau^{d_{D_2}+k_1}}{(q^{1/k_1})^{\frac{(d_{D_2}+k_1)(d_{D_2}+k_1-1)}{2}}}\sigma_{q,\tau}^{d_{D_2}\left(\frac{1}{k_2}-\frac{1}{k_1}\right)}w_{k_1}(\tau,m,\epsilon)
+\frac{\tau^{k_1}}{(q^{1/k_1})^{\frac{k_1(k_1-1)}{2}}}\Psi_{k_1}(\tau,m,\epsilon)\\
+\sum_{\ell=1}^{D-1}\epsilon^{\Delta_\ell-d_\ell}\frac{\tau^{d_\ell+k_1}}{(q^{1/k_1})^{\frac{(d_{\ell}+k_1)(d_{\ell}+k_1-1)}{2}}}\sigma_{q,\tau}^{\delta_\ell-\frac{d_\ell}{k_1}-1}\left(\frac{1}{(2\pi)^{1/2}}\varphi_{k_1,\ell}(\tau,m,\epsilon)\star^{R_\ell}_{q;1/k_1}w_{k_1}(\tau,m,\epsilon)\right)
\end{multline} 
where $w_{k_1}(\tau,m,\epsilon)$ stands for the formal $q$-Borel transformation of order $k_1$ of $U(T,m,\epsilon)$ with respect to $T$. Observe the appearance only of negative powers of the dilation operator in each one of the terms in the sum of the right-hand side of the equation.

We assume an unbounded sector of bisecting direction $d_{Q,R_{D_1}}\in\mathbb{R}$ exists,
$$S_{Q,R_{D_1}}=\left\{z\in\mathbb{C}: |z|\ge r_{Q,R_{D_1}}, |\arg(z)-d_{Q,R_{D_1}}|\le \nu_{Q,R_{D_1}}\right\},$$
for some $r_{Q,R_{D_1}},\nu_{Q,R_{D_1}}>0$, in such a way that 
$$
\frac{Q(im)}{R_{D_1}(im)}\in S_{Q,R_{D_1}},
$$
for every $m\in\mathbb{R}$. We factorize
$$P_{m,1}(\tau)=\frac{Q(im)}{(q^{1/k_1})^{\frac{k_1(k_1-1)}{2}}}-\frac{R_{D_{1}}(im)}{(q^{1/k_1})^{\frac{(d_{D_1}+k_1)(d_{D_1}+k_1-1)}{2}}}\tau^{d_{D_1}}$$
in the form
$$P_{m,1}(\tau)=-\frac{R_{D_1}(im)}{(q^{1/k_1})^{\frac{(d_{D_1}+k_1)(d_{D_1}+k_1-1)}{2}}}\prod_{\ell=0}^{d_{D_1}-1}(\tau-q_{\ell}(m)),$$
with
$$
q_{\ell}(m)=e^{\frac{2i\pi\ell}{d_{D_1}}}\left(\frac{Q(im)}{R_{D_1}(im)}\right)^{1/d_{D_1}}q^{\frac{d_{D_1}+2k_1-1}{2k_{1}}},
$$
for every $0\le \ell\le d_{D_1}-1$. Let $U_{d}$ be an unbounded sector, and $\rho>0$ such that the following statements hold:
\begin{enumerate}\label{cond41}
\item[1)] There exists $M_1>0$ such that $|\tau-q_\ell(m)|\ge M_1(1+|\tau|)$ for every $0\le \ell\le d_{D_1}-1$, $m\in\mathbb{R}$, and $\tau\in U_d\cup \overline{D}(0,\rho)$. An appropriate choice of $r_{Q,R_{D_1}}$ and $\rho$ yields $|q_{\ell}(m)|>2\rho$ for every $m\in\mathbb{R}$, and $0\le \ell\le d_{D_1}-1$. In the case that $\nu_{Q,R_{D_1}}$ is small enough, the set $\{q_{\ell}(m):m\in\R,0\le \ell\le d_{D_1}-1\}$ stays at a positive distance to $U_d$, and it can be chosen with the property that $q_{\ell}(m)/\tau$ has positive distance to $1\in\mathbb{C}$ for every $\tau\in U_d$, $m\in\mathbb{R}$ and $0\le \ell\le d_{D_1}-1$.
\item[2)] There exists $M_2>0$ such that $|\tau-q_{\ell}(m)|\ge M_2|q_{\ell}(m)|$ for every $\ell\in\{0,\ldots,d_{D_1}-1\}$, $m\in\mathbb{R}$ and $\tau\in U_d\cup \overline{D}(0,\rho)$. This is a direct consequence of 1), for some small enough $M_2>0$.
\end{enumerate}

In order to prove the following upper estimates in (\ref{e113}), we make use of 1) for all $\ell\in\{0,\ldots,d_{D_1}-1\}$, except for one of them, say $\ell_0$, for which 2) is applied. The previous conditions yield the existence of $C_P>0$ such that
\begin{multline}\label{e113}
|P_{m,1}(\tau)|\\
\ge M_1^{d_{D_1}-1}M_2\frac{|R_{D_1}(im)|(1+|\tau|)^{d_{D_1}-1}}{(q^{1/k_1})^{\frac{(d_{D_1}+k_1)(d_{D_1}+k_1-1)}{2}}}\left(\frac{|Q(im)|}{|R_{D_1}(im)|}\right)^{1/d_{D_1}}q^{\frac{d_{D_1}+2k_1-1}{2k_{1}}}\\
\ge C_P(r_{Q,R_{D_1}})^{1/d_{D_1}}|R_{D_1}(im)|(1+|\tau|)^{d_{D_1}-1},
\end{multline}
for every $\tau\in U_d\cup \overline{D}(0,\rho)$, and $m\in\mathbb{R}$.

The next result states the existence and uniqueness of a solution of (\ref{e87}) in the space $\hbox{Exp}^{q}_{(\kappa,\beta,\mu,\alpha,\rho)}$, provided its norm in that space is small enough.
\begin{prop}\label{prop347}
Under the Assumptions (\ref{e107}), (\ref{e108a}), (\ref{e237}),  (\ref{e115b}) and (\ref{e115a}), there exist $r_{Q,R_{D_1}}>0$, a constant $\varpi>0$ and constants $\varsigma_\varphi,\varsigma_{\Psi}>0$ such that if
\begin{equation}\label{e256}
\tilde{C}_{\ell}\le \varsigma_{\varphi}\qquad C_{\Psi_{k_1}}\le \varsigma_{\Psi},
\end{equation}
for all $1\le \ell\le D-1$ (see~(\ref{e119}) and (\ref{e01})), then the equation (\ref{e87}) admits a unique solution $w^{d}_{k_1}(\tau,m,\epsilon)\in\hbox{Exp}^{q}_{(\kappa,\beta,\mu,\alpha,\rho)}$ with $\left\|w^{d}_{k_1}(\tau,m,\epsilon)\right\|_{(\kappa,\beta,\mu,\alpha,\rho)}\le\varpi$, for every $\epsilon\in D(0,\epsilon_0)$.
\end{prop}
\begin{proof}
Let $\epsilon\in D(0,\epsilon_0)$ and consider the operator $\mathcal{H}_{\epsilon}$ defined by
\begin{align*}
&\mathcal{H}^1_{\epsilon}(w(\tau,m)):= \frac{R_{D_2}(im)}{P_{m,1}(\tau)}\frac{\tau^{d_{D_2}}}{(q^{1/k_1})^{\frac{(d_{D_2}+k_1)(d_{D_2}+k_1-1)}{2}}}\sigma_{q,\tau}^{d_{D_2}\left(\frac{1}{k_2}-\frac{1}{k_1}\right)}w(\tau,m)\\
&+\sum_{\ell=1}^{D-1}\epsilon^{\Delta_\ell-d_{\ell}}\frac{\tau^{d_{\ell}}}{P_{m,1}(\tau)(q^{1/k_1})^{\frac{(d_\ell+k_1)(d_\ell+k_1-1)}{2}}}\sigma_{q,\tau}^{\delta_\ell-\frac{d_\ell}{k_1}-1}\left(\frac{1}{(2\pi)^{1/2}}\varphi_{k_1,\ell}(\tau,m,\epsilon)\star^{R_\ell}_{q;1/k_1}w(\tau,m)\right)\\
&+\frac{1}{P_{m,1}(\tau)(q^{1/k_1})^{\frac{k_1(k_1-1)}{2}}}\Psi_{k_1}(\tau,m,\epsilon).
\end{align*}
Note that a fixed point of $\mathcal{H}^1_{\epsilon}(w(\tau,m))$ will lead to a convenient  solution of (\ref{e87}).
To apply the fixed point theorem, we are going to prove successively two facts.
\begin{enumerate}
\item One may choose small enough $\varsigma_\varphi,\varsigma_{\Psi},\varpi>0$, and large enough $r_{Q,R_{D_1}}>0$ such that 
\begin{equation}\label{e130}
\mathcal{H}^1_{\epsilon}(\overline{B}(0,\varpi))\subseteq \overline{B}(0,\varpi),
\end{equation}
where $\overline{B}(0,\varpi)$ stands for the closed disc centered at 0, with radius $\varpi$ in the Banach space $\hbox{Exp}^{q}_{(\kappa,\beta,\mu,\alpha,\rho)}$.
\item It holds 
\begin{equation}\label{e131}
\left\|\mathcal{H}^1_\epsilon(w_1(\tau,m))-\mathcal{H}^1_\epsilon(w_2(\tau,m))\right\|_{(\kappa,\beta,\mu,\alpha,\rho)}\le \frac{1}{2}\left\|w_1(\tau,m)-w_2(\tau,m)\right\|_{(\kappa,\beta,\mu,\alpha,\rho)},
\end{equation} 
for every $w_1(\tau,m),w_2(\tau,m)\in\overline{B}(0,\varpi)$.\end{enumerate}
\pagebreak[2]
\begin{center}
Proof of (\ref{e130}).
\end{center}
 We first check (\ref{e130}).
Let $w(\tau,m)\in\hbox{Exp}^{q}_{(\kappa,\beta,\mu,\alpha,\rho)}$.

With (\ref{e108a}) and the definition of $\kappa$, we find that 
$d_{D_{1}}-1+\kappa (d_{\ell}/k_{1}+1-\delta_{\ell})\geq d_{\ell}$ and $d_{\ell}/k_{1}+1-\delta_{\ell}\ge0$.
Thus, taking into account assumptions (\ref{e107}), (\ref{e115a}), regarding (\ref{e113}) together with Proposition~\ref{prop1} and Proposition~\ref{prop3} we get
\begin{multline}
\left\|\epsilon^{\Delta_\ell-d_{\ell}}\frac{\tau^{d_{\ell}}}{P_{m,1}(\tau)(q^{1/k_1})^{\frac{(d_\ell+k_1)(d_\ell+k_1-1)}{2}}}\sigma_{q,\tau}^{\delta_\ell-\frac{d_\ell}{k_1}-1}\left(\frac{1}{(2\pi)^{1/2}}\varphi_{k_1,\ell}(\tau,m,\epsilon)\star^{R_\ell}_{q;1/k_1}w(\tau,m)\right)\right\|_{(\kappa,\beta,\mu,\alpha,\rho)}\\
\le \epsilon_0^{\Delta_{\ell}-d_\ell}\frac{C_1C_3\varsigma_{\varphi}}{(q^{1/k_1})^{\frac{(d_\ell+k_1)(d_\ell+k_1-1)}{2}}C_{P}(r_{Q,R_{D_1}})^{1/d_{D_1}}(2\pi)^{1/2}}\left\|w(\tau,m)\right\|_{(\kappa,\beta,\mu,\alpha,\rho)}\label{e257}.
\end{multline}
Gathering Lemma~\ref{lema1}, we get
\begin{align}
&\left\|\frac{1}{P_{m,1}(\tau)(q^{1/k_1})^{k_1(k_1-1)/2}}\Psi_{k_1}(\tau,m,\epsilon)\right\|_{(\kappa,\beta,\mu,\alpha,\rho)}\nonumber\\
&\le\frac{1}{(q^{1/k_1})^{k_1(k_1-1)/2}C_P(r_{Q,R_{D_1}})^{1/d_{D_1}}}\sup_{m\in\mathbb{R}}\frac{1}{|R_{D_1}(im)|}\varsigma_{\Psi}.\label{e287}
\end{align}
Condition (\ref{e115b}) and the application of Proposition~\ref{prop1} and Lemma~\ref{lema1} yields
\begin{align}
&\left\|\frac{R_{D_2}(im)}{P_{m,1}(\tau)}\frac{\tau^{d_{D_2}}}{(q^{1/k_1})^{\frac{(d_{D_2}+k_1)(d_{D_2}+k_1-1)}{2}}}\sigma_{q,\tau}^{d_{D_2}\left(\frac{1}{k_2}-\frac{1}{k_1}\right)}w(\tau,m)\right\|_{(\kappa,\beta,\mu,\alpha,\rho)}\nonumber\\
&\le \sup_{m\in\mathbb{R}}\frac{|R_{D_2}(im)|}{|R_{D_1}(im)|}\frac{C_1}{(q^{1/k_1})^{\frac{(d_{D_2}+k_1)(d_{D_2}+k_1-1)}{2}}C_P(r_{Q,R_{D_1}})^{1/d_{D_1}}}\varpi.\label{e307}
\end{align}
An appropriate choice of $r_{Q,R_{D_1}}>0$, $\varpi,\varsigma_\Psi,\varsigma_\varphi>0$ gives
\begin{align}
&\sum_{\ell=1}^{D-1}\epsilon_0^{\Delta_{\ell}-d_\ell}\frac{C_3\varsigma_{\varphi}C_1}{(q^{1/k_1})^{\frac{(d_\ell+k_1)(d_\ell+k_1-1)}{2}}C_{P}(r_{Q,R_{D_1}})^{1/d_{D_1}}(2\pi)^{1/2}}\varpi\nonumber\\
&+\frac{1}{(q^{1/k_1})^{k_1(k_1-1)/2}C_P(r_{Q,R_{D_1}})^{1/d_{D_1}}}\sup_{m\in\mathbb{R}}\frac{1}{|R_{D_1}(im)|}\varsigma_{\Psi}\nonumber\\
&+\sup_{m\in\mathbb{R}}\frac{|R_{D_2}(im)|}{|R_{D_1}(im)|}\frac{C_1\varpi}{(q^{1/k_1})^{\frac{(d_{D_2}+k_1)(d_{D_2}+k_1-1)}{2}}C_P(r_{Q,R_{D_1}})^{1/d_{D_1}}}\le \varpi.\label{e292}
\end{align}
Regarding (\ref{e257}), (\ref{e287}), (\ref{e307}) and (\ref{e292}) one concludes (\ref{e130}). 
\begin{center}
Proof of (\ref{e131}).
\end{center}
We proceed to prove (\ref{e131}). Let $w_1,w_2\in\hbox{Exp}^{q}_{(\kappa,\beta,\mu,\alpha,\rho)}$. We assume $\left\|w_\ell(\tau,m)\right\|_{(\kappa,\beta,\mu,\alpha,\rho)}\le \varpi,$ $\ell=1,2$, for some $\varpi>0$. Let $E(\tau,m):=w_1(\tau,m)-w_2(\tau,m)$. On one hand, from (\ref{e257}) one has

\begin{align}
&\left\|\frac{\epsilon^{\Delta_\ell-d_{\ell}}\tau^{d_{\ell}}}{P_{m,1}(\tau)(q^{1/k_1})^{\frac{(d_\ell+k_1)(d_\ell+k_1-1)}{2}}}\sigma_{q,\tau}^{\delta_\ell-\frac{d_\ell}{k_1}-1}\left(\frac{1}{(2\pi)^{1/2}}\varphi_{k_1,\ell}(\tau,m,\epsilon)\star^{R_\ell}_{q;1/k_1}E(\tau,m)\right)\right\|_{(\kappa,\beta,\mu,\alpha,\rho)}\nonumber\\
&\le \epsilon_0^{\Delta_{\ell}-d_\ell}\frac{C_3\varsigma_{\varphi}C_1}{(q^{1/k_1})^{\frac{(d_\ell+k_1)(d_\ell+k_1-1)}{2}}C_{P}(r_{Q,R_{D_1}})^{1/d_{D_1}}(2\pi)^{1/2}}\left\|E(\tau,m)\right\|_{(\kappa,\beta,\mu,\alpha,\rho)}\label{e330}.
\end{align}

On the other hand, (\ref{e307}) yields

\begin{align}
&\left\|\frac{R_{D_2}(im)}{P_{m,1}(\tau)}\frac{\tau^{d_{D_2}}}{(q^{1/k_1})^{\frac{(d_{D_2}+k_1)(d_{D_2}+k_1-1)}{2}}}\sigma_{q,\tau}^{d_{D_2}\left(\frac{1}{k_2}-\frac{1}{k_1}\right)}E(\tau,m)\right\|_{(\kappa,\beta,\mu,\alpha,\rho)}\nonumber\\
&\le \sup_{m\in\mathbb{R}}\frac{|R_{D_2}(im)|}{|R_{D_1}(im)|}\frac{C_1}{(q^{1/k_1})^{\frac{(d_{D_2}+k_1)(d_{D_2}+k_1-1)}{2}}C_P(r_{Q,R_{D_1}})^{1/d_{D_1}}}\left\|E(\tau,m)\right\|_{(\kappa,\beta,\mu,\alpha,\rho)}.\label{e337}
\end{align}

We choose $r_{Q,R_{D_1}}>0$, $\varsigma_{\varphi}>0$ such that
\begin{align}
&\sum_{\ell=1}^{D-1}\epsilon_0^{\Delta_{\ell}-d_\ell}\frac{C_3\varsigma_{\varphi}C_1}{(q^{1/k_1})^{\frac{(d_\ell+k_1)(d_\ell+k_1-1)}{2}}C_{P}(r_{Q,R_{D_1}})^{1/d_{D_1}}(2\pi)^{1/2}}\nonumber\\
&+\sup_{m\in\mathbb{R}}\frac{|R_{D_2}(im)|}{|R_{D_1}(im)|}\frac{C_1}{(q^{1/k_1})^{\frac{(d_{D_2}+k_1)(d_{D_2}+k_1-1)}{2}}C_P(r_{Q,R_{D_1}})^{1/d_{D_1}}}\le \frac{1}{2}.\label{e341}
\end{align}

The statement (\ref{e131}) is a direct consequence of condition (\ref{e341}) applied to (\ref{e330}) and (\ref{e337}).\\ \par 

Let us finish the proof of the proposition. At this point, in view of (\ref{e130}) and (\ref{e131}), one can choose $\varpi>0$ such that $\overline{B}(0,\varpi)\subseteq\hbox{Exp}^{q}_{(\kappa,\beta,\mu,\alpha,\rho)}$, which defines a complete metric space for the norm $\left\|\cdot\right\|_{(\kappa,\beta,\mu,\alpha,\rho)}$. The map $\mathcal{H}^1_{\epsilon}$ is contractive from $\overline{B}(0,\varpi)$ into itself. The fixed point theorem states that $\mathcal{H}^1_{\epsilon}$ admits a unique fixed point $w_{k_1}^{d}(\tau,m,\epsilon)\in\overline{B}(0,\varpi)\subseteq\hbox{Exp}^{q}_{(\kappa,\beta,\mu,\alpha,\rho)}$, for every $\epsilon\in D(0,\epsilon_0)$. The construction of $w_{k_1}^{d}(\tau,m,\epsilon)$ allows us to conclude that it turns out to be a solution of (\ref{e87}).
\end{proof}

The next step consists of studying the solutions of a second auxiliary problem. This problem lies in a second $q$-Borel plane and its solution would guarantee the extension, with appropriate growth, of the acceleration of the solution to our first auxiliary problem, described in (\ref{e87}). %For this purpose, we describe the acceleration of the elements of $\hbox{Exp}^{q}_{(\kappa,\beta,\mu,\alpha,\rho)}$. 

We set
\begin{equation}\label{lema5b}
\Psi_{k_2}(\tau,m,\epsilon) = \sum_{n \geq 0} F_{n}(m,\epsilon)
\frac{\tau^{n}}{(q^{1/k_{2}})^{n(n-1)/2}}
\end{equation}
the $q$-Borel transform of order $k_2$ of $\hat{F}(T,m,\epsilon)$. According to the second
condition of (\ref{e119}), the expression $\Psi_{k_{2}}(\tau,m,\epsilon)$ stands for an entire
function of $q$-exponential growth of order $k_2$ which belongs to the Banach space
$\mathrm{Exp}_{(k_{2},\beta,\mu,\nu)}^{q}$ provided that $\nu \in \mathbb{R}$
satisfies $T_{0} > q^{\frac{1}{2k_{2}}}/q^{\nu/k_{2}}$ for any unbounded sector
$S_{d}$. More precisely, we have
\begin{equation}\label{e1139}
\left\|\Psi_{k_{2}}(\tau,m,\epsilon)\right\|_{(k_{2},\beta,\mu,\nu)} \leq C_{\Psi_{k_2}} 
\end{equation}
for some constant $C_{\Psi_{k_2}} >0$, for all $\epsilon \in D(0,\epsilon_{0})$.\medskip

\subsection{Analytic solutions of a second auxiliary problem in the $q$-Borel plane}\label{seccion42}

We multiply both sides of equation (\ref{e59}) by $T^{k_2}$ and apply formal $q$-Borel transformation of order $k_2$ with respect to $T$. In view of the properties of $q$-Borel transformation, the resulting problem is determined by

\begin{multline}
Q(im)\frac{\tau^{k_2}}{(q^{1/k_2})^{\frac{k_2(k_2-1)}{2}}}\hat{w}_{k_2}(\tau,m,\epsilon)\\
=R_{D_1}(im)\frac{\tau^{d_{D_1}+k_2}}{(q^{1/k_2})^{\frac{(d_{D_1}+k_2)(d_{D_1}+k_2-1)}{2}}}\sigma_{q,\tau}^{d_{D_1}\left(\frac{1}{k_1}-\frac{1}{k_2}\right)}\hat{w}_{k_2}(\tau,m,\epsilon)\\
+R_{D_2}(im)\frac{\tau^{d_{D_2}+k_2}}{(q^{1/k_2})^{\frac{(d_{D_2}+k_2)(d_{D_2}+k_2-1)}{2}}}\hat{w}_{k_2}(\tau,m,\epsilon)+\frac{\tau^{k_2}}{(q^{1/k_2})^{\frac{k_2(k_2-1)}{2}}}\Psi_{k_2}(\tau,m,\epsilon)\\
+\sum_{\ell=1}^{D-1}\epsilon^{\Delta_\ell-d_\ell}\frac{\tau^{d_\ell+k_2}}{(q^{1/k_2})^{\frac{(d_{\ell}+k_2)(d_{\ell}+k_2-1)}{2}}}\sigma_{q,\tau}^{\delta_\ell-\frac{d_\ell}{k_2}-1}\left(\frac{1}{(2\pi)^{1/2}}\varphi_{k_2,\ell}(\tau,m,\epsilon)\star^{R_\ell}_{q;1/k_2}\hat{w}_{k_2}(\tau,m,\epsilon)\right).\label{e412}
\end{multline}

Here $\hat{w}_{k_2}(\tau,m,\epsilon)$,
$\Psi_{k_2}(\tau,m,\epsilon)$ and $\varphi_{k_{2},l}(\tau,m,\epsilon)$ stand for the
formal $q$-Borel transform of order $k_2$ of $U(T,m,\epsilon)$, $\hat{F}(T,m,\epsilon)$ and
$\hat{C}_{l}(T,m,\epsilon)$.

%Here, $\hat{w}_{k_2}(\tau,m,\epsilon)$ (resp. $\hat{\Psi}_{k_2}(\tau,m,\epsilon)$) stands for the formal $q$-Borel transformation of order $k_2$ of $\hat{U}(T,m,\epsilon)$ (resp. $F(T,m,\epsilon)$) with respect to $T$. Observe from (\ref{e332}) that $\hat{\Psi}_{k_2}(\tau,m,\epsilon)$ has null radius of convergence, and for that reason, the formal power series solution of (\ref{e412}) has also null radius of convergence. 

We %substitute $\hat{\Psi}_{k_2}$ by $\Psi_{k_2}$, constructed in (\ref{lema5b}) in equation (\ref{e412}) and 
consider our second auxiliary problem, namely

\begin{multline}
Q(im)\frac{\tau^{k_2}}{(q^{1/k_2})^{\frac{k_2(k_2-1)}{2}}}w_{k_2}(\tau,m,\epsilon)\\
=R_{D_1}(im)\frac{\tau^{d_{D_1}+k_2}}{(q^{1/k_2})^{\frac{(d_{D_1}+k_2)(d_{D_1}+k_2-1)}{2}}}\sigma_{q,\tau}^{d_{D_1}\left(\frac{1}{k_1}-\frac{1}{k_2}\right)}w_{k_2}(\tau,m,\epsilon)\\
+R_{D_2}(im)\frac{\tau^{d_{D_2}+k_2}}{(q^{1/k_2})^{\frac{(d_{D_2}+k_2)(d_{D_2}+k_2-1)}{2}}}w_{k_2}(\tau,m,\epsilon)+\frac{\tau^{k_2}}{(q^{1/k_2})^{\frac{k_2(k_2-1)}{2}}}\Psi_{k_2}(\tau,m,\epsilon)\\
+\sum_{\ell=1}^{D-1}\epsilon^{\Delta_\ell-d_\ell}\frac{\tau^{d_\ell+k_2}}{(q^{1/k_2})^{\frac{(d_{\ell}+k_2)(d_{\ell}+k_2-1)}{2}}}\sigma_{q,\tau}^{\delta_\ell-\frac{d_\ell}{k_2}-1}\left(\frac{1}{(2\pi)^{1/2}}\varphi_{k_2,\ell}(\tau,m,\epsilon)\star^{R_\ell}_{q;1/k_2}w_{k_2}(\tau,m,\epsilon)\right).\label{e412b}
\end{multline}

We assume an unbounded sector of bisecting direction $d_{Q,R_{D_2}}\in\mathbb{R}$ exists,
$$S_{Q,R_{D_2}}=\left\{z\in\mathbb{C}: |z|\ge r_{Q,R_{D_2}}, |\arg(z)-d_{Q,R_{D_2}}|\le \nu_{Q,R_{D_2}}\right\},$$
for some $\nu_{Q,R_{D_2}}>0$, in such a way that 
$$
\frac{Q(im)}{R_{D_2}(im)}\in S_{Q,R_{D_2}},
$$
for every $m\in\mathbb{R}$. We factorize
$$P_{m,2}(\tau)=\frac{Q(im)}{(q^{1/k_2})^{\frac{k_2(k_2-1)}{2}}}-\frac{R_{D_{2}}(im)}{(q^{1/k_2})^{\frac{(d_{D_2}+k_2)(d_{D_2}+k_2-1)}{2}}}\tau^{d_{D_2}}$$
in the form
$$P_{m,2}(\tau)=-\frac{R_{D_2}(im)}{(q^{1/k_2})^{\frac{(d_{D_2}+k_2)(d_{D_2}+k_2-1)}{2}}}\prod_{\ell=0}^{d_{D_2}-1}(\tau-q_{\ell,2}(m)).$$
Let $S_d$ be an unbounded sector with small enough aperture in such a way that:\label{cond42}
\begin{enumerate}
\item[1)] There exists $M_{12}>0$ such that $|\tau-q_{\ell,2}(m)|\ge M_{12}(1+|\tau|)$ for every $0\le \ell\le d_{D_2}-1$, $m\in\mathbb{R}$, and $\tau\in S_d$.
\item[2)] There exists $M_{22}>0$ such that we have $|\tau-q_{\ell,2}(m)|\ge M_{22}|q_{\ell,2}(m)|$ for all ${\ell\in\{0,\ldots,d_{D_2}-1\}}$, $m\in\mathbb{R}$ and $\tau\in S_d$.
\end{enumerate}

In the following estimates, we apply 1) in the previous assumption to all indices $\ell\in\{0,\ldots,d_{D_2}-1\}$, except for one of them, say $\ell_0$, that we apply 2). This yields the existence of $C_{P,2}>0$ such that
\begin{equation}\label{e463}
|P_{m,2}(\tau)|\ge C_{P,2}(r_{Q,R_{D_2}})^{1/d_{D_2}}|R_{D_2}(im)|(1+|\tau|)^{d_{D_2}-1},
\end{equation}
for every $\tau\in S_d$, and $m\in\mathbb{R}$.

\begin{prop}\label{prop11}
Under the hypotheses  (\ref{e107}), (\ref{e108a}), (\ref{e237}),  (\ref{e115b}), (\ref{e115a}) and those on the geometry of the problem, there exist $r_{Q,R_{D_2}},\varpi>0$ and $\varsigma_{\varphi},\varsigma_{\Psi}>0$ under (\ref{e256}) such that for every $\epsilon\in D(0,\epsilon_0)$, the equation (\ref{e412b}) admits a unique solution $w_{k_2}^d(\tau,m,\epsilon)$ in the space $\hbox{Exp}^q_{(k_2,\beta,\mu,\nu)}$ for $\nu\in\mathbb{R}$ %determined in Lemma \ref{lema5b}, 
and depends holomorphically with respect to $\epsilon\in D(0,\epsilon_0)$. Moreover, $\left\|w^d_{k_2}(\tau,m,\epsilon)\right\|_{(k_2,\beta,\mu,\nu)}\le\varpi$.
\end{prop}

\begin{proof}
Let $\epsilon\in D(0,\epsilon_0)$. We consider the map $\mathcal{H}_{\epsilon}^{2}$ defined by 
\begin{multline}
\mathcal{H}_{\epsilon}^{2}(w(\tau,m))
:= \frac{R_{D_1}(im)}{P_{m,2}(\tau)}\frac{\tau^{d_{D_1}}}{(q^{1/k_2})^{\frac{(d_{D_1}+k_2)(d_{D_1}+k_2-1)}{2}}}\sigma_{q,\tau}^{d_{D_1}\left(\frac{1}{k_1}-\frac{1}{k_2}\right)}w(\tau,m)\\
+\sum_{\ell=1}^{D-1}\epsilon^{\Delta_\ell-d_{\ell}}\frac{\tau^{d_{\ell}}}{P_{m,2}(\tau)(q^{1/k_2})^{\frac{(d_\ell+k_2)(d_\ell+k_2-1)}{2}}}\sigma_{q,\tau}^{\delta_\ell-\frac{d_\ell}{k_2}-1}\left(\frac{1}{(2\pi)^{1/2}}\varphi_{k_2,\ell}(\tau,m,\epsilon)\star^{R_\ell}_{q;1/k_2}w(\tau,m)\right)\\
+\frac{1}{P_{m,2}(\tau)(q^{1/k_2})^{\frac{k_2(k_2-1)}{2}}}\Psi_{k_2}(\tau,m,\epsilon).\label{e594}
\end{multline}
Note that a fixed point of $\mathcal{H}^{2}_{\epsilon}(w(\tau,m))$ will lead to a convenient  solution of (\ref{e412b}).
To apply the fixed point theorem, we are going to prove successively two facts.
\begin{enumerate}
\item One may choose small enough $\varsigma_\varphi,\varsigma_{\Psi},\varpi>0$, and large enough $r_{Q,R_{D_2}}>0$ such that 
\begin{equation}\label{e130b}
\mathcal{H}_{\epsilon}^2(\overline{B}(0,\varpi))\subseteq \overline{B}(0,\varpi),
\end{equation}
where $\overline{B}(0,\varpi)$ stands for the closed disc centered at 0, with radius $\varpi$ in the Banach space $\hbox{Exp}^{q}_{(k_2,\beta,\mu,\nu)}$.
\item It holds 
\begin{equation}\label{e131b}
\left\|\mathcal{H}^2_\epsilon(w_1(\tau,m))-\mathcal{H}^2_\epsilon(w_2(\tau,m))\right\|_{(k_2,\beta,\mu,\nu)}\le \frac{1}{2}\left\|w_1(\tau,m)-w_2(\tau,m)\right\|_{(k_2,\beta,\mu,\nu)},
\end{equation} 
for every $w_1(\tau,m),w_2(\tau,m)\in\overline{B}(0,\varpi)$.\end{enumerate}

\begin{center}
Proof of (\ref{e130b}).
\end{center}
 We first check (\ref{e130b}). Let $w(\tau,m)\in\hbox{Exp}^{q}_{(k_2,\beta,\mu,\nu)}$.

With (\ref{e108a}), we find that 
$d_{D_{2}}-1+k_{2} (d_{\ell}/k_{2}+1-\delta_{\ell})\geq d_{\ell}$.
Taking into account assumptions (\ref{e107}), (\ref{e115b}), (\ref{e115a}), regarding (\ref{e463}) together with Lemma~\ref{lema2}, Proposition~\ref{prop4} and Proposition~\ref{prop3b} we get
\begin{align}
&\left\|\epsilon^{\Delta_\ell-d_{\ell}}\frac{\tau^{d_{\ell}}}{P_{m,2}(\tau)(q^{1/k_2})^{\frac{(d_\ell+k_2)(d_\ell+k_2-1)}{2}}}\sigma_{q,\tau}^{\delta_\ell-\frac{d_\ell}{k_2}-1}\left(\frac{1}{(2\pi)^{1/2}}\varphi_{k_2,\ell}(\tau,m,\epsilon)\star^{R_\ell}_{q;1/k_2}w(\tau,m)\right)\right\|_{(k_2,\beta,\mu,\nu)}\nonumber\\
&\le \epsilon_0^{\Delta_{\ell}-d_\ell}\frac{C_4C_3\varsigma_\varphi}{(q^{1/k_2})^{\frac{(d_\ell+k_2)(d_\ell+k_2-1)}{2}}C_{P,2}(r_{Q,R_{D_2}})^{1/d_{D_2}}(2\pi)^{1/2}}\left\|w(\tau,m)\right\|_{(k_2,\beta,\mu,\nu)}\label{e257b}.
\end{align}
Gathering Lemma~\ref{lema2} we get
\begin{align}
&\left\|\frac{1}{P_{m,2}(\tau)(q^{1/k_2})^{k_2(k_2-1)/2}}\Psi_{k_2}(\tau,m,\epsilon)\right\|_{(k_2,\beta,\mu,\nu)}\nonumber\\
&\le\frac{1}{(q^{1/k_2})^{k_2(k_2-1)/2}C_{P,2}(r_{Q,R_{D_2}})^{1/d_{D_2}}}\sup_{m\in\mathbb{R}}\frac{1}{|R_{D_2}(im)|}\varsigma_{\Psi_2},\label{e287b}
\end{align}
for some $\varsigma_{\Psi_2}$. Observe that %, in view of the proof of Lemma \ref{lema5b}, one derives that 
$\varsigma_{\Psi_2}$ tends to 0 when $\varsigma_{\Psi}$ does.

Condition (\ref{e115b}), and the application of Proposition~\ref{prop4} and Lemma~\ref{lema2} yields
\begin{align}
&\left\|\frac{R_{D_1}(im)}{P_{m,2}(\tau)}\frac{\tau^{d_{D_1}}}{(q^{1/k_2})^{\frac{(d_{D_1}+k_2)(d_{D_1}+k_2-1)}{2}}}\sigma_{q,\tau}^{d_{D_1}\left(\frac{1}{k_1}-\frac{1}{k_2}\right)}w(\tau,m)\right\|_{(k_2,\beta,\mu,\nu)}\nonumber\\
&\le \sup_{m\in\mathbb{R}}\frac{|R_{D_1}(im)|}{|R_{D_2}(im)|}\frac{C_4}{(q^{1/k_2})^{\frac{(d_{D_1}+k_2)(d_{D_1}+k_2-1)}{2}}C_{P,2}(r_{Q,R_{D_2}})^{1/d_{D_2}}}\varpi. \label{e307b}
\end{align}

An appropriate choice of $r_{Q,R_{D_2}}>0$, $\varpi,\varsigma_\varphi,\varsigma_\Psi>0$ gives
\begin{align}
&\sum_{\ell=1}^{D-1}\epsilon_0^{\Delta_{\ell}-d_\ell}\frac{C_3\varsigma_\varphi C_4}{(q^{1/k_2})^{\frac{(d_\ell+k_2)(d_\ell+k_2-1)}{2}}C_{P,2}(r_{Q,R_{D_2}})^{1/d_{D_2}}(2\pi)^{1/2}}\varpi\nonumber\\
&+\frac{1}{(q^{1/k_2})^{k_2(k_2-1)/2}C_{P,2}(r_{Q,R_{D_2}})^{1/d_{D_2}}}\sup_{m\in\mathbb{R}}\frac{1}{|R_{D_2}(im)|}\varsigma_{\Psi}\nonumber\\
&+\sup_{m\in\mathbb{R}}\frac{|R_{D_1}(im)|}{|R_{D_2}(im)|}\frac{C_4\varpi}{(q^{1/k_2})^{\frac{(d_{D_1}+k_2)(d_{D_1}+k_2-1)}{2}}C_{P,2}(r_{Q,R_{D_2}})^{1/d_{D_2}}}\le \varpi.\label{e292b}
\end{align}

Regarding (\ref{e257b}), (\ref{e287b}), (\ref{e307b}) and (\ref{e292b}) one concludes (\ref{e130b}).
\pagebreak[2]
\begin{center}
Proof of (\ref{e131b}).
\end{center}
 We proceed to prove (\ref{e131b}). Let $w_1,w_2\in\hbox{Exp}^{q}_{(k_2,\beta,\mu,\nu)}$. We assume $\left\|w_\ell(\tau,m)\right\|_{(k_2,\beta,\mu,\nu)}\le \varpi,$ $\ell=1,2$, for some $\varpi>0$. Let $E(\tau,m)=w_1(\tau,m)-w_2(\tau,m)$. On one hand, from (\ref{e257b}) one has

\begin{align}
&\left\|\epsilon^{\Delta_\ell-d_{\ell}}\frac{\tau^{d_{\ell}}}{P_{m,2}(\tau)(q^{1/k_2})^{\frac{(d_\ell+k_2)(d_\ell+k_2-1)}{2}}}\sigma_{q,\tau}^{\delta_\ell-\frac{d_\ell}{k_2}-1}\left(\frac{1}{(2\pi)^{1/2}}\varphi_{k_2,\ell}(\tau,m,\epsilon)\star^{R_\ell}_{q;1/k_2}E(\tau,m)\right)\right\|_{(k_2,\beta,\mu,\nu)}\nonumber\\
&\le \epsilon_0^{\Delta_{\ell}-d_\ell}\frac{C_3\varsigma_\varphi C_4}{(q^{1/k_2})^{\frac{(d_\ell+k_2)(d_\ell+k_2-1)}{2}}C_{P,2}(r_{Q,R_{D_2}})^{1/d_{D_2}}(2\pi)^{1/2}}\left\|E(\tau,m)\right\|_{(k_2,\beta,\mu,\nu)}\nonumber.
\end{align}

On the other hand, (\ref{e307b}) yields

\begin{align}
&\left\|\frac{R_{D_1}(im)}{P_{m,2}(\tau)}\frac{\tau^{d_{D_1}}}{(q^{1/k_2})^{\frac{(d_{D_1}+k_2)(d_{D_1}+k_2-1)}{2}}}\sigma_{q,\tau}^{d_{D_1}\left(\frac{1}{k_1}-\frac{1}{k_2}\right)}E(\tau,m)\right\|_{(k_2,\beta,\mu,\nu)}\nonumber\\
&\le \sup_{m\in\mathbb{R}}\frac{|R_{D_1}(im)|}{|R_{D_2}(im)|}\frac{C_4}{(q^{1/k_2})^{\frac{(d_{D_1}+k_2)(d_{D_1}+k_2-1)}{2}}C_{P,2}(r_{Q,R_{D_2}})^{1/d_{D_2}}}\left\|E(\tau,m)\right\|_{(k_2,\beta,\mu,\nu)}.\nonumber
\end{align}

We choose $r_{Q,R_{D_2}}>0$, $\varsigma_\varphi>0$ such that
\begin{align}
&\sum_{\ell=1}^{D-1}\epsilon_0^{\Delta_{\ell}-d_\ell}\frac{C_3\varsigma_\varphi C_4}{(q^{1/k_2})^{\frac{(d_\ell+k_2)(d_\ell+k_2-1)}{2}}C_{P,2}(r_{Q,R_{D_2}})^{1/d_{D_2}}(2\pi)^{1/2}}\nonumber\\
&+\sup_{m\in\mathbb{R}}\frac{|R_{D_1}(im)|}{|R_{D_2}(im)|}\frac{C_4}{(q^{1/k_2})^{\frac{(d_{D_1}+k_2)(d_{D_1}+k_2-1)}{2}}C_{P,2}(r_{Q,R_{D_2}})^{1/d_{D_2}}}\le \frac{1}{2}.\nonumber
\end{align}
We conclude (\ref{e131b}). Let us finish the proof of the proposition. At this point, in view of (\ref{e130b}) and (\ref{e131b}), one can choose $\varpi>0$ such that $\overline{B}(0,\varpi)\subseteq\hbox{Exp}^{q}_{(k_2,\beta,\mu,\nu)}$, which defines a complete metric space for the norm $\left\|\cdot\right\|_{(k_2,\beta,\mu,\nu)}$. The map $\mathcal{H}^{2}_{\epsilon}$ is contractive from $\overline{B}(0,\varpi)$ into itself. The fixed point theorem states that $\mathcal{H}^{2}_{\epsilon}$ admits a unique fixed point $w_{k_2}^{d}(\tau,m,\epsilon)\in\overline{B}(0,\varpi)\subseteq\hbox{Exp}^{q}_{(k_2,\beta,\mu,\nu)}$, for every $\epsilon\in D(0,\epsilon_0)$. The construction of $w_{k_2}^{d}(\tau,m,\epsilon)$ allow us to conclude it turns out to be a solution of (\ref{e412b}).
\end{proof}
The existing link between the acceleration of $w^d_{k_1}$ and $w^d_{k_2}$ is now provided. Both functions coincide in the intersection of their domain of definition. This fact assures the extension of the acceleration of $w^d_{k_1}$ along direction $d$, with appropriate $q$-exponential growth in order to apply $q$-Laplace transformation of that order to recover the analytic solution of the main problem under study.

\begin{prop}\label{prop12}
We consider $w_{k_1}^d(\tau,m,\epsilon)$ constructed in Proposition~\ref{prop347}. The function
$$\tau\mapsto\mathcal{L}_{q;1/\kappa}^d(w^d_{k_1}(\tau,m,\epsilon)):=\mathcal{L}_{q;1/\kappa}^d(h\mapsto w^d_{k_1}(h,m,\epsilon))(\tau)$$
defines a bounded holomorphic function in $\mathcal{R}_{d,\tilde{\delta}}\cap D(0,r_1)$, for $0<r_1\le q^{\left(\frac{1}{2}-\alpha\right)/\kappa}/2$. Moreover, it holds that
\begin{equation}\label{e657}
\mathcal{L}_{q;1/\kappa}^d(w^d_{k_1}(\tau,m,\epsilon))=w_{k_2}^d(\tau,m,\epsilon),\qquad (\tau,m,\epsilon)\in  S_{d}^{b}\times \mathbb{R}\times D(0,\epsilon_0), 
\end{equation}
where $S_{d}^{b}$ is a finite sector of bisecting direction $d$.
\end{prop}
\begin{proof}
We recall from Proposition~\ref{prop347} that $w^d_{k_1}\in\hbox{Exp}^{q}_{(\kappa,\beta,\mu,\alpha,\rho)}$. This guarantees appropriate bounds on $\tau\in U_d$ in order to apply $q$-Laplace transformation of order $\kappa$ along direction $d$. This yields that for every $\tilde{\delta}>0$, the function $\mathcal{L}_{q;1/\kappa}^d(w^d_{k_1}(\tau,m,\epsilon))$ defines a bounded and holomorphic function in $\mathcal{R}_{d,\tilde{\delta}}\cap D(0,r_1)$ for $0<r_1\le q^{\left(\frac{1}{2}-\alpha\right)/\kappa}/2$. 

In order to prove that (\ref{e657}) holds, it is sufficient to prove that $\mathcal{L}_{q;1/\kappa}^d(w_{k_1}^d(\tau,m,\epsilon))$ and $w^d_{k_2}$ are both solutions of some problem, with unique solution in certain Banach space, so they must coincide. For that purpose, we multiply both sides of equation (\ref{e87}) by $\tau^{-k_1}$ and take $q$-Laplace transformation of order $\kappa$ along direction $d$.

The properties of $q$-Laplace transformation yield
\begin{equation}\label{e668}
\mathcal{L}_{q;1/\kappa}^d\left(\tau^{d_{D_1}}w_{k_1}^d(\tau,m,\epsilon)\right)=(q^{1/\kappa})^{d_{D_1}(d_{D_1}-1)/2}\tau^{d_{D_1}}\sigma_{q,\tau}^{\frac{d_{D_1}}{\kappa}}\mathcal{L}_{q;1/\kappa}^d(w_{k_1}^d)(\tau,m,\epsilon),
\end{equation}
\begin{equation}\label{e669}
\mathcal{L}_{q;1/\kappa}^d\left(\tau^{d_{D_2}}\sigma_{q,\tau}^{d_{D_2}\left(\frac{1}{k_2}-\frac{1}{k_1}\right)}w_{k_1}^d(\tau,m,\epsilon)\right)=(q^{1/\kappa})^{d_{D_2}(d_{D_2}-1)/2}\tau^{d_{D_2}}\mathcal{L}_{q;1/\kappa}^d(w_{k_1}^d)(\tau,m,\epsilon),
\end{equation}
and
\begin{multline}\label{e677}
\mathcal{L}_{q;1/\kappa}^d\left(\tau^{d_{\ell}}\sigma_{q,\tau}^{\delta_\ell+\frac{d_\ell}{k_1}-1}\left(\frac{1}{(2\pi)^{1/2}}\varphi_{k_1,\ell}(\tau,m,\epsilon)\star^{R_{\ell}}_{q;1/k_1}w_{k_1}^d(\tau,m,\epsilon)\right)\right)\\
=(q^{1/\kappa})^{d_\ell(d_\ell-1)/2}\tau^{d_{\ell}}\sigma_{q,\tau}^{\delta_{\ell}-\frac{d_\ell}{k_2}-1}\mathcal{L}_{q;1/\kappa}^d\left(\frac{1}{(2\pi)^{1/2}}\varphi_{k_1,\ell}(\tau,m,\epsilon)\star^{R_{\ell}}_{q;1/k_1}w_{k_1}^d(\tau,m,\epsilon)\right).
\end{multline}

We claim that we have
\begin{equation}\label{e683}
\mathcal{L}_{q;1/\kappa}^d\left(\varphi_{k_1,\ell}(\tau,m,\epsilon)\star^{R_{\ell}}_{q;1/k_1}w_{k_1}^d(\tau,m,\epsilon)\right)=\varphi_{k_2,\ell}(\tau,m,\epsilon)\star^{R_{\ell}}_{q;1/k_2}\mathcal{L}_{q;1/\kappa}^d(w_{k_1}^d(\tau,m,\epsilon)).
\end{equation}

This is a consequence of the change in the order of integration in the operators involved in (\ref{e683}). This situation is different from that of (60) in the proof of Proposition 12 in~\cite{lamaq}. Assume the variable of integration with respect to Laplace operator is $r$. After the change of variable $\tilde{r}=r/q^{h/k_1}$, we reduce the study to that of $\Xi$ in the proof of Proposition 12 in~\cite{lamaq}, with $r$ replaced by $r^{1-h}$. This last argument guarantees the availability of the change of order in the integration operators involved in (\ref{e683}). We now give a proof of (\ref{e683}) under this consideration.

We have
\begin{center}
$\mathcal{L}_{q;1/\kappa}^d\left(\varphi_{k_1,\ell}(\tau,m,\epsilon)\star^{R_{\ell}}_{q;1/k_1}w_{k_1}^d(\tau,m,\epsilon)\right)$
\end{center}
\vspace{-1cm}
\begin{multline*}
=\frac{1}{\pi_{q^{1/\kappa}}}\int_{0}^{\infty}\left(\varphi_{k_1,\ell}(re^{id},m,\epsilon)\star^{R_\ell}_{q;1/k_1}w_{k_1}^{d}(re^{id},m,\epsilon)\right)\frac{1}{\Theta_{q^{1/\kappa}}\left(\frac{re^{id}}{\tau}\right)}\frac{dr}{r}\\
=\frac{1}{\pi_{q^{1/\kappa}}}\int_{0}^{\infty}\left(\sum_{n\ge0}\frac{(re^{id})^n}{(q^{1/k_1})^{n(n-1)/2}}C_{\ell,n}(m,\epsilon)\star^{R_{\ell}}(\sigma_{q,\tau}^{-\frac{n}{k_1}}w^d_{k_1})(re^{id},m,\epsilon)\right)\frac{1}{\Theta_{q^{1/\kappa}}\left(\frac{re^{id}}{\tau}\right)}\frac{dr}{r}\\
=\frac{1}{\pi_{q^{1/\kappa}}}\int_{0}^{\infty}\left(\sum_{n\ge0}\frac{(re^{id})^n}{(q^{1/k_1})^{n(n-1)/2}}\int_{-\infty}^{\infty}C_{\ell,n}(m-m_1,\epsilon)R_{\ell}(im_1)w^d_{k_1}(re^{id}q^{-\frac{n}{k_1}},m_1,\epsilon)dm_1\right)\\
\qquad\times\frac{1}{\Theta_{q^{1/\kappa}}\left(\frac{re^{id}}{\tau}\right)}\frac{dr}{r}.
\end{multline*}
We make the change of variable $\tilde{r}=r/q^{n/k_1}$ to get that the previous expression equals
\begin{align*}
&\frac{1}{\pi_{q^{1/\kappa}}}\int_{0}^{\infty}\left(\int_{-\infty}^{\infty}\sum_{n\ge0}\frac{(\tilde{r}e^{id})^nq^{n^2/k_1}}{(q^{1/k_1})^{n(n-1)/2}}C_{\ell,n}(m-m_1,\epsilon)R_{\ell}(im_1)w^d_{k_1}(\tilde{r}e^{id},m_1,\epsilon)dm_1\right)\\
&\times\frac{1}{\Theta_{q^{1/\kappa}}\left(\frac{\tilde{r}e^{id}q^{n/k_1}}{\tau}\right)}\frac{d\tilde{r}}{\tilde{r}}.
\end{align*}
In view of (\ref{e245}), $k_{1}^{-1}=\kappa^{-1}+k_{2}^{-1}$, the change of order of the integrals and the dominated convergence theorem, the previous equation equals 
\begin{align*}
&=\frac{1}{\pi_{q^{1/\kappa}}}\int_{0}^{\infty}\left(\int_{-\infty}^{\infty}\sum_{n\ge0}\frac{(\tilde{r}e^{id})^nq^{n^2/k_1}}{(q^{1/k_1})^{n(n-1)/2}}C_{\ell,n}(m-m_1,\epsilon)R_{\ell}(im_1)w^d_{k_1}(\tilde{r}e^{id},m_1,\epsilon)dm_1\right)\\
&\qquad\times\frac{1}{\Theta_{q^{1/\kappa}}\left(\frac{\tilde{r}e^{id}q^{n/k_2}}{\tau}\right)q^{\frac{n(n+1)}{2\kappa}}\left(\frac{\tilde{r}e^{id}q^{n/k_2}}{\tau}\right)^{n}}\frac{d\tilde{r}}{\tilde{r}}\\
&=\frac{1}{\pi_{q^{1/\kappa}}}\int_{0}^{\infty}\left(\int_{-\infty}^{\infty}\sum_{n\ge0}\frac{\tau^n q^{n(n-1)/(2\kappa)}}{(q^{1/k_1})^{n(n-1)/2}}C_{\ell,n}(m-m_1,\epsilon)R_{\ell}(im_1)w^d_{k_1}(\tilde{r}e^{id},m_1,\epsilon)dm_1\right)\\
&\qquad\times\frac{1}{\Theta_{q^{1/\kappa}}\left(\frac{\tilde{r}e^{id}q^{n/k_2}}{\tau}\right)}\frac{d\tilde{r}}{\tilde{r}}\\
&=\int_{-\infty}^{\infty}\left(\sum_{n\ge0}\frac{\tau^n }{(q^{1/k_2})^{n(n-1)/2}}C_{\ell,n}(m-m_1,\epsilon)\right)R_{\ell}(im_1)\left[\frac{1}{\pi_{q^{1/\kappa}}}\int_{0}^{\infty}\frac{w^d_{k_1}(\tilde{r}e^{id},m_1,\epsilon)}{\Theta_{q^{1/\kappa}}\left(\frac{\tilde{r}e^{id}q^{n/k_2}}{\tau}\right)}\frac{d\tilde{r}}{\tilde{r}}\right]dm_1\\
&\qquad=\varphi_{k_2,\ell}(\tau,m,\epsilon)\star^{R_{\ell}}_{q;1/k_2}\mathcal{L}_{q;1/\kappa}^d(w_{k_1}^d(\tau,m,\epsilon)),
\end{align*}
from where we conclude (\ref{e683}).

On the other hand, we observe by direct computation that
\begin{equation}\label{e1053}
\mathcal{L}_{q;1/\kappa}^{d}(\Psi_{k_1}(\tau,m,\epsilon)) = \Psi_{k_2}(\tau,m,\epsilon)
\end{equation}
for every $(\tau,m,\epsilon) \in (\mathcal{R}_{d,\tilde{\delta}} \cap D(0,r_{1})) \times
\mathbb{R} \times D(0,\epsilon_{0})$.

In view of (\ref{e668}), (\ref{e669}), (\ref{e677}), (\ref{e683}), and the last formula above (\ref{e1053}), we derive that 
\begin{multline*}
\frac{Q(im)}{(q^{1/k_1})^{\frac{k_1(k_1-1)}{2}}}\mathcal{L}^{d}_{q;1/\kappa}(w^d_{k_1})(\tau,m,\epsilon)\\
=R_{D_1}(im)\frac{(q^{1/\kappa})^{d_{D_1}(d_{D_1}-1)/2}}{(q^{1/k_1})^{\frac{(d_{D_1}+k_1)(d_{D_1}+k_1-1)}{2}}}\tau^{d_{D_1}}\sigma_{q,\tau}^{\frac{d_{D_1}}{\kappa}}\mathcal{L}_{q;1/\kappa}^d(w_{k_1}^d)(\tau,m,\epsilon)\\
+R_{D_2}(im)\frac{(q^{1/\kappa})^{\frac{d_{D_2}(d_{D_2}-1)}{2}}}{(q^{1/k_1})^{\frac{(d_{D_2}+k_1)(d_{D_2}+k_1-1)}{2}}}\tau^{d_{D_2}}\mathcal{L}_{q;1/\kappa}^d(w_{k_1}^d)(\tau,m,\epsilon)+\frac{1}{(q^{1/k_1})^{\frac{k_1(k_1-1)}{2}}}\Psi_{k_2}(\tau,m,\epsilon)\\
+\sum_{\ell=1}^{D-1}\epsilon^{\Delta_\ell-d_\ell}\tau^{d_\ell}\frac{(q^{1/\kappa})^{\frac{d_\ell(d_{\ell}-1)}{2}}}{(q^{1/k_1})^{\frac{(d_{\ell}+k_1)(d_{\ell}+k_1-1)}{2}}}\sigma_{q,\tau}^{\delta_\ell-\frac{d_\ell}{k_2}-1}\left(\frac{1}{(2\pi)^{1/2}}\varphi_{k_2,\ell}(\tau,m,\epsilon)\star^{R_\ell}_{q;1/k_2}\mathcal{L}^{d}_{q;1/\kappa}(w^d_{k_1})(\tau,m,\epsilon)\right),
\end{multline*} 
for every $(\tau,m,\epsilon)\in (\mathcal{R}_{d,\tilde{\delta}}\cap D(0,r_1))\times\mathbb{R}\times D(0,\epsilon_0)$. We multiply at both sides of the previous equation by $(q^{1/k_1})^{k_1(k_1-1)/2}/(q^{1/k_2})^{k_2(k_2-1)/2}$. The fact that
$$\frac{(q^{1/\kappa})^{\frac{\mathcal{D}(\mathcal{D}-1)}{2}}(q^{1/k_1})^{\frac{k_1(k_1-1)}{2}}}{(q^{1/k_1})^{\frac{(\mathcal{D}+k_1)(\mathcal{D}+k_1-1)}{2}}(q^{1/k_2})^{\frac{k_2(k_2-1)}{2}}}=\frac{1}{(q^{1/k_2})^{\frac{(\mathcal{D}+k_2)(\mathcal{D}+k_2-1)}{2}}},$$
with $\mathcal{D}\in\{d_{D_1},d_{D_2},d_{\ell}\}$ entails that $\mathcal{L}_{q;1/\kappa}^d(w_{k_1}^d(\tau,m,\epsilon))$ is a solution of (\ref{e412b}) in its domain of definition.

Let $S_{d}^{b}$ be a bounded sector of bisecting direction $d$ such that $S_{d}^{b}\subseteq(\mathcal{R}_{d,\tilde{\delta}}\cap D(0,r_1))\cap S_d$, which is a nonempty set due to the assumptions on the construction of these sets. The functions $\mathcal{L}_{q;1/\kappa}^d(w_{k_1}^d(\tau,m,\epsilon))$ and $w_{k_2}^d(\tau,m,\epsilon)$ are continuous complex functions defined on $S^{b}_d\times\mathbb{R}\times D(0,\epsilon_0)$ and holomorphic with respect to $\tau$ (resp. $\epsilon$) on $S_{d}^b$ (resp. $D(0,\epsilon_0)$). 

Let $\epsilon\in D(0,\epsilon_0)$ and put $\Omega=\min\{\alpha,\nu\}$. It is straight to check that both functions belong to the complex Banach space $H_{(k_2,\beta,\mu,\Omega)}$, of all continuous functions $(\tau,m)\mapsto h(\tau,m)$, defined on $S^b_d\times\mathbb{R}$, holomorphic with respect to $\tau$ in $S^{b}_d$ such that
$$\left\|h(\tau,m)\right\|_{H_{(k_2,\beta,\mu,\Omega)}}=\sup_{\tau\in S^{b}_d,m\in\mathbb{R}}(1+|m|)^{\mu}e^{\beta|m|}\exp\left(-\frac{k_2}{2}\frac{\log^2|\tau|}{\log(q)}-\Omega\log|\tau|\right)|h(\tau,m)|$$
is finite. It holds that $\mathcal{L}_{q;1/\kappa}^d(w_{k_1}^d(\tau,m,\epsilon))$, $w_{k_2}^d(\tau,m,\epsilon)$ and $\Psi_{k_2}(\tau,m,\epsilon)$ belong to $H_{(k_2,\beta,\mu,\Omega)}$ due to Proposition~\ref{prop347}, Proposition~\ref{prop11}. As we can see in the proof of Proposition~\ref{prop11}, the operator $\mathcal{H}^{2}_{\epsilon}$ defined in (\ref{e594}) has a unique fixed point in $H_{(k_2,\beta,\mu,\Omega)}$ provided small enough constants $\varsigma_{\Psi},\varsigma_{\varphi}>0$, for $1\le \ell\le D-1$.  Indeed, this fixed point is a solution of the auxiliary problem (\ref{e412b}) in the disc $D(0,\varsigma)$ of $H_{(k_2,\beta,\mu,\Omega)}$, whilst $\mathcal{L}_{q;1/\kappa}^d(w_{k_1}^d(\tau,m,\epsilon))$, $w_{k_2}^d(\tau,m,\epsilon)$ are both solutions of the same problem, in the disc $D(0,\varsigma)$ of $H_{(k_2,\beta,\mu,\Omega)}$, so they do coincide in the domain $S_{d}^b\times\mathbb{R}\times D(0,\epsilon_0)$. Identity (\ref{e657}) follows from here.
\end{proof}

\section{Analytic solutions to a $q$-difference-differential equation}\label{sec5}

This section is devoted to determine in detail the main problem under study, and provide an analytic solution to it. It is worth mentioning that, although the techniques developed in previous sections are essentially novel, once the tools have been implemented, the procedure of construction of the solution coincides with that explained in Section 5 of~\cite{lamaq}. For the sake of completeness and a self contained work, we describe every step of the construction in detail, whilst we have decided to pass over the proofs which can be found in~\cite{lamaq}.

Let $1\le k_1<k_2$. We define $1/\kappa=1/k_1-1/k_2$ and take integers $D,D_1,D_2$ larger than 3. Let $q>1$ be a real number. We also consider positive integers $d_{D_1},d_{D_2}$, and for every $1\le \ell\le D-1$ we choose non negative integers $d_\ell,\delta_\ell\ge1$ and $\Delta_\ell\ge0$. We make the following assumptions on the previous constants:

\textbf{Assumption (A):} $\delta_1=1$ and $\delta_\ell<\delta_{\ell+1}$ for every $1\le \ell\le D-2$.

\textbf{Assumption (B):} We have
$$\Delta_\ell\ge d_\ell,\quad \frac{d_{D_1}-1}{\kappa}+\frac{d_\ell}{k_2}+1\ge\delta_\ell,\quad \frac{d_\ell}{k_1}+1\ge \delta_\ell,\quad  \frac{d_{D_2}-1}{k_2}\ge\delta_\ell-1,$$
for every $1\le \ell\le D-1$, and 
$$ k_1(d_{D_2}-1)>k_2d_{D_1}.$$

Let $Q, R_{D_1},R_{D_2}$, and $R_{\ell}$ for $1\le \ell\le D-1$ be polynomials with complex coefficients such that

\textbf{Assumption (C):} $\deg(R_{D_2})= \deg(R_{D_1})$, $\deg(Q)\ge \deg(R_{D_1})\ge \deg(R_\ell)$. Moreover, we assume $Q(im)\neq0$ and $R_{D_j}(im)\neq0$ for all $m\in\mathbb{R}$, $1\le \ell\le D-1$.

Let $S_{Q,R_{D_1}}$ and $S_{Q,R_{D_2}}$ be unbounded sectors of bisecting directions ${d_{Q,R_{D_1}}\in\mathbb{R}}$ and ${d_{Q,R_{D_2}}\in\mathbb{R}}$ respectively, with
$$S_{Q,R_{D_j}}=\left\{z\in\mathbb{C}: |z|\ge r_{Q,R_{D_j}}, |\arg(z)-d_{Q,R_{D_j}}|\le \nu_{Q,R_{D_j}}\right\},$$
for some $\nu_{Q,R_{D_j}}>0$, and such that 
$$\frac{Q(im)}{R_{D_j}(im)}\in S_{Q,R_{D_j}},$$
for every $m\in\mathbb{R}$.

\begin{defin}\label{defin111}
Let $\varsigma\ge2$ be an integer. A family $(\mathcal{E}_{p})_{0\le p\le \varsigma-1}$ is said to be a good covering in $\mathbb{C}^{\star}$ (in the $\epsilon$ plane) if the next hypotheses hold:
\begin{itemize}
\item $\mathcal{E}_{p}$ is an open sector of finite radius $\epsilon_0>0$, and vertex at the origin for every $0\le p\le\varsigma-1$.
\item $\mathcal{E}_{j}\cap\mathcal{E}_{k}\neq\emptyset$ for $0\le j,k\le \varsigma-1$ if and only if $|j-k|\le 1$ (we put $\mathcal{E}_{\varsigma}:=\mathcal{E}_{0}$).
\item $\cup_{p=0}^{\varsigma-1}\mathcal{E}_p=\mathcal{U}\setminus\{0\}$ for some neighborhood of the origin $\mathcal{U}$.
\end{itemize}
\end{defin}

\begin{defin}\label{defi2}
Let $(\mathcal{E}_{p})_{0\le p\le \varsigma-1}$ be a good covering. Let $\mathcal{T}$ be an open bounded sector with vertex at the origin and radius $r_{\mathcal{T}}>0$. Given $\alpha\in\mathbb{R}$ and $\nu\in\mathbb{R}$ we assume that
$$0<\epsilon_0,r_{\mathcal{T}}<1,\quad \nu+\frac{k_2}{\log(q)}\log(r_{\mathcal{T}})<0,\quad \alpha+\frac{\kappa}{\log(q)}\log(\epsilon_0 r_{\mathcal{T}})<0,\quad \epsilon_0r_{\mathcal{T}}\le q^{\left(\frac{1}{2}-\nu\right)/k_2}/2.$$

We consider a family of unbounded sectors $U_{\mathfrak{d}_p}$, $0\le p\le \varsigma-1$, with bisecting direction $\mathfrak{d}_p\in\mathbb{R}$, and a family of open domains $\mathcal{R}_{\mathfrak{d}_p}^b:=\mathcal{R}_{\mathfrak{d}_p,\tilde{\delta}}\cap D(0,\epsilon_0 r_{\mathcal{T}})$, with 
$$\mathcal{R}_{\mathfrak{d}_p,\tilde{\delta}}:=\left\{T\in\mathbb{C}^{\star}:\left|1+\frac{re^{i\mathfrak{d}_p}}{T}\right|>\tilde{\delta},\hbox{ for every }r\ge0\right\},$$
for some $\tilde{\delta}<1$.
We assume $\mathfrak{d}_{p}$, $0\le p\le \varsigma-1$ is chosen to satisfy the following conditions: there exist $S_{\mathfrak{d}_p}\cup \overline{D}(0,\rho)$ and $\rho>0$ such that
\begin{itemize}
\item Conditions 1), 2), Page \pageref{cond41} in Section~\ref{seccion41} hold. Observe that, under this assumption, Conditions 1), 2), Page \pageref{cond42} in Section~\ref{seccion42} hold for $S_{\mathfrak{d}_p}$.
\item For every $0\le p\le \varsigma-1$ we have $\mathcal{R}_{\mathfrak{d}_p}^{b}\cap \mathcal{R}_{\mathfrak{d}_{p+1}}^{b}\neq\emptyset$, and for every $t\in\mathcal{T}$ and $\epsilon\in \mathcal{E}_{p}$ we have $\epsilon t\in \mathcal{R}_{\mathfrak{d}_p}^{b}$ (where $\mathcal{R}_{\mathfrak{d}_\varsigma}:=\mathcal{R}_{\mathfrak{d}_0}$).
\end{itemize}
The family $\{(\mathcal{R}_{\mathfrak{d}_p,\tilde{\delta}})_{0\le p\le \varsigma-1},D(0,\rho),\mathcal{T}\}$ is said to be associated to the good covering $(\mathcal{E}_{p})_{0\le p\le \varsigma-1}$.
\end{defin}

Let $(\mathcal{E}_{p})_{0\le p\le \varsigma-1}$ be a good covering, and a family $\{(\mathcal{R}_{\mathfrak{d}_p,\tilde{\delta}})_{0\le p\le \varsigma-1},D(0,\rho),\mathcal{T}\}$ associated to it. For every $0\le p\le \varsigma-1$ we study the following equation
\begin{multline}\label{epral}
Q(\partial_z)\sigma_{q,t}u^{\mathfrak{d}_{p}}(t,z,\epsilon)\\
=(\epsilon t)^{d_{D_1}}\sigma_{q,t}^{\frac{d_{D_1}}{k_1}+1}R_{D_1}(\partial_{z})u^{\mathfrak{d}_{p}}(t,z,\epsilon)+(\epsilon t)^{d_{D_2}}\sigma_{q,t}^{\frac{d_{D_2}}{k_2}+1}R_{D_2}(\partial_{z})u^{\mathfrak{d}_{p}}(t,z,\epsilon)\\
+\sum_{\ell=1}^{D-1}\epsilon^{\Delta_\ell}t^{d_{\ell}}\sigma_{q,t}^{\delta_{\ell}}(c_{\ell}(t,z,\epsilon)R_{\ell}(\partial_z)u^{\mathfrak{d}_{p}}(t,z,\epsilon))+\sigma_{q,t}f(t,z,\epsilon).
\end{multline}

The terms $c_{\ell}(t,z,\epsilon)$ are determined as follows, for every $1\le\ell\le D-1$. Let $C_{\ell}(T,m,\epsilon)$ be the entire function in $T$, with coefficients in $E_{(\beta,\mu)}$ for some $\beta>0$ and $\mu\in  \R$, given by
$$C_{\ell}(T,m,\epsilon)=\sum_{n\ge0}C_{\ell,n}(m,\epsilon)T^n,$$
such that ${\mu-1\geq \deg(R_{D_j})}$, for $j\in \{1,2\}$. Assume this function depends holomorphically on $\epsilon\in D(0,\epsilon_0)$ and also the existence of $\tilde{C}_\ell, T_0>0$ such that the left-hand side of (\ref{e119}) holds for all $n\ge0$ and $\epsilon\in D(0,\epsilon_0)$. We put 
$$c_{\ell}(t,z,\epsilon):=\mathcal{F}^{-1}\left(m\mapsto C_{\ell}(\epsilon t,m,\epsilon)\right)(z),$$
which is a holomorphic and bounded function on $\mathcal{T}\times H_{\beta'}\times D(0,\epsilon_0)$, with $0<\beta'<\beta$. Indeed, one can substitute $\mathcal{T}$ by any bounded set in $\mathbb{C}$ in the previous product domain.

The function $f(t,z,\epsilon)$ is constructed as follows. Let $m\mapsto F_{n}(m,\epsilon)$ be a function in $E_{(\beta,\mu)}$ for every $n\ge0$, depending holomorphically on $\epsilon\in D(0,\epsilon_0)$. We also assume there exist $C_{F}, T_0$ such that (\ref{e119}) holds and define $\hat{F}(T,m,\epsilon)=\sum_{n\ge0}F_nT^n$. %Then, the formal $q$-Borel transform of $k_1$ of the formal power series $\hat{F}(t,m,\epsilon)$ with respect to $T$, $\Psi_{k_1}^{\mathfrak{d}_p}(\tau,m,\epsilon)$, is holomorphic in $D(0,T_0)$. Assume this function can be prolonged along an infinite sector of bisecting direction $\mathfrak{d}_{p}$ such that 
%\begin{equation}\label{e1139}
%\left\|\Psi_{k_1}^{\mathfrak{d}_p}(\tau,m,\epsilon)\right\|_{(\kappa,\beta,\mu,\alpha,\rho)}=C_{\Psi_{k_1}}<\infty,
%\end{equation}
%but this extension is not an entire function. More precisely, assume that (\ref{e332}) holds. We define $\Psi_{k_2}^{\mathfrak{d}_p}(\tau,m,\epsilon)$ following Lemma~\ref{lema5b}. Property (\ref{e413}) allow us to compute 
%\begin{equation}\label{e1140}
%F^{\mathfrak{d}_{p}}(T,m,\epsilon):=\mathcal{L}^{\mathfrak{d}_p}_{q;1/k_2}(\tau\mapsto\Psi_{k_2}^{\mathfrak{d}_p}(\tau,m,\epsilon))(T),
%\end{equation}
%which defines a holomorphic function with respect to $T$ in the domain $\mathcal{R}_{\mathfrak{d}_p,\tilde{\delta}}\cap D(0,r_1)$, for all ${0<r_1\le q^{\left(\frac{1}{2}-\nu\right)/k_2}/2}$. We define
%\begin{equation}\label{e821}
%f^{\mathfrak{d}_p}(t,z,\epsilon):=\mathcal{F}^{-1}(m\mapsto F^{\mathfrak{d}_{p}}(\epsilon t, m, \epsilon))(z),
%\end{equation}
%which defines a holomorphic and bounded function on $\mathcal{T}\times H_{\beta'}\times \mathcal{E}_p$ for all $0<\beta'<\beta$ under the assumptions at the beginning of this section.

By construction, $\hat{F}(T,m,\epsilon)$ represents a holomorphic function in $T$ on the disc
$D(0,T_{0}/2)$ with values in the Banach space $E_{(\beta,\mu)}$, for all
$\epsilon \in D(0,\epsilon_{0})$. We define
$$ f(t,z,\epsilon) = \mathcal{F}^{-1}(m \mapsto \hat{F}(\epsilon t,m,\epsilon) )(z)$$
which stands for a holomorphic and bounded function on
$D(0,\epsilon_{0}T_{0}/2) \times H_{\beta'} \times D(0,\epsilon_{0})$, for all
$0 < \beta' < \beta$.

\begin{theo}\label{teo872}
Under the construction made at the beginning of this section of the elements involved in the problem (\ref{epral}), assume that the above conditions hold. Let $(\mathcal{E}_{p})_{0\le p\le\varsigma-1}$ be a good covering in $\mathbb{C}^{\star}$, for which a family $\{(\mathcal{R}_{\mathfrak{d}_{p},\tilde{\delta}})_{0\le p\le \varsigma-1},D(0,\rho),\mathcal{T}\}$ associated to this covering is considered.

Then, there exist large enough $r_{Q,R_{D_1}},r_{Q,R_{D_2}}>0$ and constants $\varsigma_{\Psi}>0$ and $\varsigma_\varphi>0$ such that if
$$\tilde{C}_\ell\le \varsigma_{\varphi},\qquad C_{\Psi_1}\le \varsigma_{\psi}$$ 
for all $1\le \ell\le D-1$, then for every $0\le p\le\varsigma-1$, one can construct a solution $u^{\mathfrak{d}_{p}}(t,z,\epsilon)$ of (\ref{epral}),  which defines a holomorphic function on $\mathcal{T}\times H_{\beta'}\times\mathcal{E}_p$, for every $0<\beta'<\beta$.
\end{theo}
\begin{proof}
Let $0\le p\le \varsigma-1$ and consider the equation

\begin{align}
Q(im)\sigma_{q,T}U^{\mathfrak{d}_{p}}(T,m,\epsilon)=T^{d_{D_1}}\sigma_{q,T}^{\frac{d_{D_1}}{k_1}+1}R_{D_1}(im)U^{\mathfrak{d}_{p}}(T,m,\epsilon)+T^{d_{D_2}}\sigma_{q,T}^{\frac{d_{D_2}}{k_2}+1}R_{D_2}(im)U^{\mathfrak{d}_{p}}(T,m,\epsilon)\nonumber\\
+\sum_{\ell=1}^{D-1}\epsilon^{\Delta_\ell-d_\ell}T^{d_{\ell}}\sigma_{q,T}^{\delta_{\ell}}\left(\frac{1}{(2\pi)^{1/2}}\int_{-\infty}^{+\infty}C_{\ell}(T,m-m_1,\epsilon)R_\ell(im_1)U^{\mathfrak{d}_{p}}(T,m_1,\epsilon)dm_1\right)\label{e59c}\\
+\sigma_{q,T}\hat{F}(T,m,\epsilon).\nonumber
\end{align}
Under an appropriate choice of the constants $\varsigma_{\Psi}$ and $\varsigma_{\varphi}$ one can follow the construction in Section~\ref{seccion41} and apply Proposition~\ref{prop347} to obtain a solution $U^{\mathfrak{d}_{p}}(T,m,\epsilon)$ of (\ref{e59c}).

Regarding the properties of $q$-Laplace transformation, and from the results obtained in Section~\ref{seccion42}, $U^{\mathfrak{d}_{p}}(T,m,\epsilon)$ is the $q$-Laplace transformation of order $k_2$ of a function $w_{k_2}^{\mathfrak{d}_p}$ along direction $\mathfrak{d}_{p}$, which depends on $T$. Indeed, 
\begin{equation}\label{e861}
U^{\mathfrak{d}_p}(T,m,\epsilon)=\frac{1}{\pi_{q^{1/k_2}}}\int_{L_{\mathfrak{d}_p}}\frac{w^{\mathfrak{d}_p}_{k_2}(u,m,\epsilon)}{\Theta_{q^{1/k_2}}\left(\frac{u}{T}\right)}\frac{du}{u},
\end{equation}
for some $L_{\mathfrak{d}_p}\subseteq S_{\mathfrak{d}_{p}}\cup\{0\}$, and $w_{k_2}^{\mathfrak{d}_p}(\tau,m,\epsilon)$ defines a continuous function on $S_{\mathfrak{d}_p}\times\mathbb{R}\times D(0,\epsilon_0)$, and holomorphic with respect to $(\tau,\epsilon)$ in $S_{\mathfrak{d}_p}\times D(0,\epsilon_0)$. In addition to this, there exists $C_{w_{k_2}^{\mathfrak{d}_p}}>0$ such that
\begin{equation}\label{e865}
|w_{k_2}^{\mathfrak{d}_p}(\tau,m,\epsilon)|\le C_{w_{k_2}^{\mathfrak{d}_p}}\frac{1}{(1+|m|)^{\mu}}e^{-\beta|m|}\exp\left(\frac{k_2}{2\log(q)}\log^2|\tau|+\nu\log|\tau|\right),
\end{equation}
for some $\nu\in\mathbb{R}$. This holds for $\tau\in S_{\mathfrak{d}_{p}}$, $m\in\mathbb{R}$, and $\epsilon\in D(0,\epsilon_0)$. Moreover, in view of Proposition~\ref{prop12}, the function $w_{k_2}^{\mathfrak{d}_p}(\tau,m,\epsilon)$ and the $q$-Laplace transformation of order $\kappa$ of the function $w_{k_1}^{\mathfrak{d}_p}(\tau,m,\epsilon)$ along direction $\mathfrak{d}_{p}^{1}$, where $e^{i\mathfrak{d}_{p}^1}\mathbb{R}_+\subseteq S_{\mathfrak{d}_{p}}\cup\{0\}$, depending on $\tau$, coincide in $(S_{\mathfrak{d}_{p}}\cap D(0,r_1))\times\mathbb{R}\times D(0,\epsilon_0)$, for $0<r_1\le q^{\left(\frac{1}{2}-\alpha\right)/\kappa}/2$, for some $\alpha\in\mathbb{R}$. The function $w_{k_1}^{\mathfrak{d}_p}(\tau,m,\epsilon)$ is such that
\begin{equation}\label{e869}
|w_{k_1}^{\mathfrak{d}_p}(\tau,m,\epsilon)|\le C_{w_{k_1}^{\mathfrak{d}_p}}\frac{1}{(1+|m|)^{\mu}}e^{-\beta|m|}\exp\left(\frac{\kappa}{2\log(q)}\log^2|\tau+\delta|+\alpha\log|\tau+\delta|\right),
\end{equation}
for some $C_{w_{k_1}^{\mathfrak{d}_p}},\delta>0$, valid for $\tau\in (D(0,\rho)\cup U_{\mathfrak{d}_p})$, $m\in\mathbb{R}$ and $\epsilon\in D(0,\epsilon_0)$. This function is the extension of a function $w_{k_1}(\tau,m,\epsilon)$, common for every $0\le p\le \varsigma-1$, continuous on $\overline{D}(0,\rho)\times\mathbb{R}\times D(0,\epsilon_0)$ and holomorphic with respect to $(\tau,\epsilon)$ in $D(0,\rho)\times D(0,\epsilon_0)$.

The bounds in (\ref{e865}) with respect to $m$ variable are transmitted to $U^{\mathfrak{d}_{p}}(T,m,\epsilon)$ as defined in (\ref{e861}). This allows to define the function
\begin{align*}
u^{\mathfrak{d}_{p}}(t,z,\epsilon)&:=\mathcal{F}^{-1}(m\mapsto U^{\mathfrak{d}_{p}}(\epsilon t,m,\epsilon))(z)\nonumber\\
&=\frac{1}{(2\pi)^{1/2}}\frac{1}{\pi_q^{1/k_2}}\int_{-\infty}^{\infty}\int_{L_{\mathfrak{d}_p}}\frac{w_{k_2}^{\mathfrak{d}_p}(u,m,\epsilon)}{\Theta_{q^{1/k_2}}\left(\frac{u}{\epsilon t}\right)}\frac{du}{u}\exp(izm)dm,
\end{align*}
which turns out to be holomorphic on $\mathcal{T}\times H_{\beta'}\times\mathcal{E}_p$. The properties of inverse Fourier transform allow us to conclude that $u^{\mathfrak{d}_{p}}(t,z,\epsilon)$ is a solution of equation (\ref{epral}) defined on $\mathcal{T}\times H_{\beta'}\times \mathcal{E}_{p}$.
\end{proof}

\begin{prop}\label{prop925}
Let $0\le p \le \varsigma -1$. Under the hypotheses of Theorem~\ref{teo872}, assume that the unbounded sectors $U_{\mathfrak{d}_{p}}$ and $U_{\mathfrak{d}_{p+1}}$ are wide enough so that $U_{\mathfrak{d}_{p}}\cap U_{\mathfrak{d}_{p+1}}$ contains the sector $U_{\mathfrak{d}_{p},\mathfrak{d}_{p+1}}=\{\tau\in\C^{\star}:\arg(\tau)\in[\mathfrak{d}_{p},\mathfrak{d}_{p+1}]\}$. Then, there exist $K_1>0$ and $K_2\in\R$ such that
\begin{equation}\label{e927}
|u^{\mathfrak{d}_{p+1}}(t,z,\epsilon)-u^{\mathfrak{d}_{p}}(t,z,\epsilon)|\le K_{1}\exp\left(-\frac{k_2}{2\log(q)}\log^2|\epsilon|\right)|\epsilon|^{K_2},
\end{equation}
%&\hspace{3cm}|f^{\mathfrak{d}_{p+1}}(t,z,\epsilon)-f^{\mathfrak{d}_{p}}(t,z,\epsilon)|\le K_{1}\exp\left(-\frac{k_2}{2\log(q)}\log^2|\epsilon|\right)|\epsilon|^{K_2}, 
%\end{align}
for every $t\in\mathcal{T}$, $z\in H_{\beta'}$, and $\epsilon\in\mathcal{E}_{p}\cap\mathcal{E}_{p+1}$. 

\end{prop}
\begin{proof}
Let $0\le p \le \varsigma -1$. Taking into account that $U_{\mathfrak{d}_{p},\mathfrak{d}_{p+1}}\subseteq U_{\mathfrak{d}_{p}}\cap U_{\mathfrak{d}_{p+1}}$, we observe from the construction of the functions $U^{\mathfrak{d}_{p}}$ and $U^{\mathfrak{d}_{p+1}}$ that $\mathcal{L}^{\mathfrak{d}_{p}}_{q;1/\kappa}(w_{k_1}^{\mathfrak{d}_{p}})(\tau,m,\epsilon)$ and $\mathcal{L}^{\mathfrak{d}_{p+1}}_{q;1/\kappa}(w_{k_1}^{\mathfrak{d}_{p+1}})(\tau,m,\epsilon)$ coincide in the domain $(\mathcal{R}_{\mathfrak{d}_{p}}^{b}\cap \mathcal{R}_{\mathfrak{d}_{p+1}}^{b})\times \R\times D(0,\epsilon_0)$. This entails the existence of $w_{k_2}^{\mathfrak{d}_{p},\mathfrak{d}_{p+1}}(\tau,m,\epsilon)$, holomorphic with respect to $\tau$ on $\mathcal{R}_{\mathfrak{d}_{p}}^{b}\cup \mathcal{R}_{\mathfrak{d}_{p+1}}^{b}$, continuous with respect to $m\in\R$ and holomorphic with respect to $\epsilon$ in $D(0,\epsilon_0)$ which coincides with $\mathcal{L}^{\mathfrak{d}_{p}}_{q;1/\kappa}(w_{k_1}^{\mathfrak{d}_{p}})(\tau,m,\epsilon)$ on $\mathcal{R}^{b}_{\mathfrak{d}_{p}}\times\R\times D(0,\epsilon_0)$ and also with $\mathcal{L}^{\mathfrak{d}_{p+1}}_{q;1/\kappa}(w_{k_1}^{\mathfrak{d}_{p+1}})(\tau,m,\epsilon)$ on $\mathcal{R}^{b}_{\mathfrak{d}_{p+1}}\times\R\times D(0,\epsilon_0)$.

Let $\widetilde{\rho}>0$ be such that $\widetilde{\rho}e^{i\mathfrak{d}_{p}}\in\mathcal{R}^{b}_{\mathfrak{d}_{p}}$ and $\widetilde{\rho}e^{i\mathfrak{d}_{p+1}}\in\mathcal{R}^{b}_{\mathfrak{d}_{p+1}}$. The function 
$$u\mapsto \frac{w_{k_2}^{\mathfrak{d}_{p},\mathfrak{d}_{p+1}}(u,m,\epsilon)}{\Theta_{q^{1/k_2}}\left(\frac{u}{\epsilon t}\right)}$$
is holomorphic on $\mathcal{R}^{b}_{\mathfrak{d}_{p}}\cup \mathcal{R}^{b}_{\mathfrak{d}_{p+1}}$ for all $(m,\epsilon)\in\R\times (\mathcal{E}_{p}\cap \mathcal{E}_{p+1})$  and its integral along the closed path constructed by concatenation of the segment starting at the origin and with ending point fixed at $\widetilde{\rho}e^{i\mathfrak{d}_{p}}$, the arc of circle with radius $\widetilde{\rho}$ connecting $\widetilde{\rho}e^{i\mathfrak{d}_{p}}$ with ${\widetilde{\rho}e^{i\mathfrak{d}_{p+1}}\subseteq\mathcal{R}^{b}_{\mathfrak{d}_{p+1}}}$, and the segment from $\widetilde{\rho}e^{i\mathfrak{d}_{p+1}}$ to 0, vanishes. The difference $u^{\mathfrak{d}_{p+1}}-u^{\mathfrak{d}_{p}}$ can be written in the form

\begin{multline}
u^{\mathfrak{d}_{p+1}}(t,z,\epsilon)-u^{\mathfrak{d}_{p}}(t,z,\epsilon) \\
=\frac{1}{(2\pi)^{1/2}}\frac{1}{\pi_{q^{1/k_2}}}\int_{-\infty}^{\infty}\int_{L_{\mathfrak{d}_{p+1},\widetilde{\rho}}}\frac{w^{\mathfrak{d}_{p+1}}_{k_2}(u,m,\epsilon)}{\Theta_{q^{1/k_2}}\left(\frac{u}{\epsilon t}\right)}\exp(izm)\frac{du}{u}dm,\hspace{3cm}\\ 
-\frac{1}{(2\pi)^{1/2}}\frac{1}{\pi_{q^{1/k_2}}}\int_{-\infty}^{\infty}\int_{L_{\mathfrak{d}_p,\widetilde{\rho}}}\frac{w^{\mathfrak{d}_{p}}_{k_2}(u,m,\epsilon)}{\Theta_{q^{1/k_2}}\left(\frac{u}{\epsilon t}\right)}\exp(izm)\frac{du}{u}dm\\
+\frac{1}{(2\pi)^{1/2}}\frac{1}{\pi_{q^{1/k_2}}}\int_{-\infty}^{\infty}\int_{\mathcal{C}_{\widetilde{\rho},\mathfrak{d}_p,\mathfrak{d}_{p+1}}}\frac{w_{k_2}^{\mathfrak{d}_p,\mathfrak{d}_{p+1}}(u,m,\epsilon)}{\Theta_{q^{1/k_2}}\left(\frac{u}{\epsilon t}\right)}\exp(izm)\frac{du}{u}dm,\label{e943}
\end{multline}
where $L_{\mathfrak{d}_{j},\widetilde{\rho}}=[\widetilde{\rho},+\infty)e^{i\mathfrak{d}_{j}}$ for $j\in\{p,p+1\}$ and $\mathcal{C}_{\widetilde{\rho},\mathfrak{d}_p,\mathfrak{d}_{p+1}}$ is the arc of circle connecting $\widetilde{\rho}e^{i\mathfrak{d}_{p}}$ with $\widetilde{\rho}e^{i\mathfrak{d}_{p+1}}$ (see Figure \ref{fig1}).

\begin{figure}[h]
	\centering
		\label{fig:conf0}
		\includegraphics[width=0.50\textwidth]{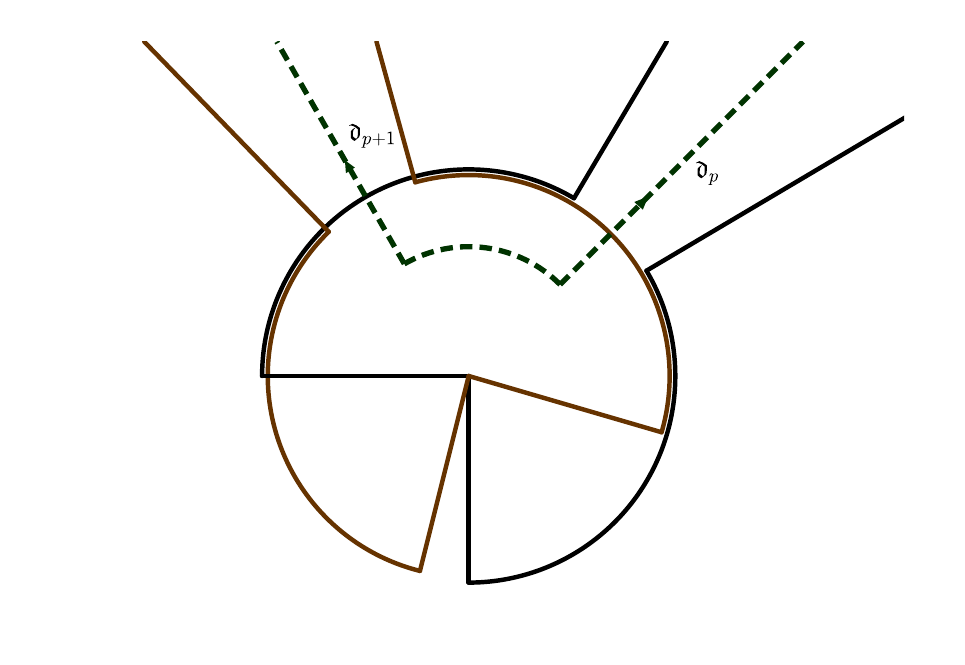}
	\caption{Deformation of the path of integration, first case.}\label{fig1}
\end{figure}

Let us put
$$I_{1}:=\left|\frac{1}{(2\pi)^{1/2}}\frac{1}{\pi_{q^{1/k_2}}}\int_{-\infty}^{\infty}\int_{L_{\mathfrak{d}_{p+1},\widetilde{\rho}}}\frac{w^{\mathfrak{d}_{p+1}}_{k_2}(u,m,\epsilon)}{\Theta_{q^{1/k_2}}\left(\frac{u}{\epsilon t}\right)}\exp(izm)\frac{du}{u}dm\right|.$$
In view of (\ref{e865}) and (\ref{eq3}), one has
\begin{align*}
I_{1}&\le \frac{C_{w_{k_2}^{\mathfrak{d}_{p+1}}}}{C_{q,k_2}\tilde{\delta}(2\pi)^{1/2}}\frac{|\epsilon t|^{1/2}}{\pi_{q^{1/k_2}}}\int_{-\infty}^{\infty}e^{-\beta|m|-m\Im(z)}\frac{dm}{(1+|m|)^{\mu}}\nonumber\\
&\times\int_{\widetilde{\rho}}^{\infty}\exp\left(\frac{k_2\log^2|u|}{2\log(q)}+\nu\log|u|\right)|u|^{-3/2}\exp\left(-\frac{k_2\log^2\left(\frac{|u|}{|\epsilon t|}\right)}{2\log(q)}\right)d|u|.
\end{align*}
We recall that we have restricted the domain on the variable $z$ such that $|\Im(z)|\le\beta'<\beta$. Then, the first integral in the previous expression in convergent, and one derives
$$I_{1}\le \frac{\tilde{C}_{w_{k_2}^{\mathfrak{d}_{p+1}}}}{(2\pi)^{1/2}}\frac{ (\epsilon_0  r_{\mathcal{T}})^{1/2}}{\pi_{q^{1/k_2}}}\int_{\widetilde{\rho}}^{\infty}\exp\left(\frac{k_2\log^2|u|}{2\log(q)}\right)\exp\left(-\frac{k_2\log^2\left(\frac{|u|}{|\epsilon t|}\right)}{2\log(q)}\right)|u|^{\nu-3/2}d|u|,$$
for some $\tilde{C}_{w_{k_2}^{\mathfrak{d}_{p+1}}}>0$.
We derive
\begin{align*}
\exp\left(\frac{k_2\log^2|u|}{2\log(q)}\right)\exp\left(-\frac{k_2\log^2\left(\frac{|u|}{|\epsilon t|}\right)}{2\log(q)}\right)&=\exp\left(\frac{k_2}{2\log(q)}(-\log^2|\epsilon|-2\log|\epsilon|\log|t|-\log^2|t|)\right)\\
&\times \exp\left(\frac{k_2}{\log(q)}(\log|u|\log|\epsilon|+\log|u|\log|t|)\right).
\end{align*}

From the assumption that $0<\epsilon_0<1$ and $0<r_{\mathcal{T}}<1$, we get
\begin{align}\label{e964}
\exp\left(-\frac{k_2}{\log(q)}\log|\epsilon|\log|t|\right)\le|\epsilon|^{-\frac{k_2}{\log(q)}\log(r_{\mathcal{T}})},\\ \exp\left(\frac{k_2}{\log(q)}\log|u|\log|\epsilon|\right)\le |\epsilon|^{\frac{k_2}{\log(q)}\log(\widetilde{\rho})},\nonumber
\end{align}
for $t\in\mathcal{T}$, $\epsilon\in\mathcal{E}_{p}\cap\mathcal{E}_{p+1}$, $|u|\ge \widetilde{\rho},$ and also
\begin{multline}
\exp\left(\frac{k_2}{\log(q)}\log|u|\log|t|\right)\le|t|^{\frac{k_2}{\log(q)}\log(\widetilde{\rho})}, \hbox{ if } \widetilde{\rho}\le |u|\le 1 \\
\exp\left(\frac{k_2}{\log(q)}\log|u|\log|t|\right)\le |u|^{\frac{k_2}{\log(q)}\log(r_{\mathcal{T}})},\hbox{ if } |u|\ge 1,\label{e968}
\end{multline}
for $t\in\mathcal{T}$. In addition to that, there exists $K_{k_2,\widetilde{\rho},q}>0$ such that
\begin{equation}\label{e973}
\sup_{x>0}x^{\frac{k_2}{\log(q)}\log(\widetilde{\rho})}\exp\left(-\frac{k_2}{2\log(q)}\log^2(x)\right)\le K_{k_2,\widetilde{\rho},q}.
\end{equation}

In view of (\ref{e964}), (\ref{e968}), (\ref{e973}), and bearing in mind that the inequalities of Definition \ref{defi2} hold, we deduce there exist $\tilde{K}^{1}\in\R$, $\tilde{K}^{2}>0$ such that
$$ \exp\left(\frac{k_2\log^2|u|}{2\log(q)}\right)\exp\left(-\frac{k_2\log^2\left(\frac{|u|}{|\epsilon t|}\right)}{2\log(q)}\right)|u|^{\nu}\le \tilde{K}^2\exp\left(-\frac{k_2}{2\log(q)}\log^2|\epsilon|\right)|\epsilon|^{\tilde{K}^{1}},$$
for $t\in\mathcal{T}$, $r\ge \widetilde{\rho}$, and $\epsilon\in\mathcal{E}_{p}\cap\mathcal{E}_{p+1}$. Provided this last inequality, 
we arrive at
\begin{multline}
I_{1}\le\frac{\tilde{K}^{2}\tilde{C}_{w_{k_2}^{\mathfrak{d}_{p+1}}}}{(2\pi)^{1/2}}\frac{(\epsilon_0 r_{\mathcal{T}})^{1/2}}{\pi_{q^{1/k_2}}}\int_{\widetilde{\rho}}^{\infty}\frac{d|u|}{|u|^{3/2}}\exp\left(-\frac{k_2}{2\log(q)}\log^2|\epsilon|\right)|\epsilon|^{\tilde{K}^{1}} \\
=\tilde{K}^3\exp\left(-\frac{k_2}{2\log(q)}\log^2|\epsilon|\right)|\epsilon|^{\tilde{K}^{1}}\label{e983},
\end{multline}
for some $\tilde{K}^{3}>0$, for all $t\in\mathcal{T}$, $z\in H_{\beta'}$, and $\epsilon\in\mathcal{E}_{p}\cap\mathcal{E}_{p+1}$.

We can estimate in the same manner the expression
$$I_{2}:=\left|\frac{1}{(2\pi)^{1/2}}\frac{1}{\pi_{q^{1/k_2}}}\int_{-\infty}^{\infty}\int_{L_{\mathfrak{d}_{p},\widetilde{\rho}}}\frac{w^{\mathfrak{d}_{p}}_{k_2}(u,m,\epsilon)}{\Theta_{q^{1/k_2}}\left(\frac{u}{\epsilon t}\right)}\exp(izm)\frac{du}{u}dm\right|$$
to arrive at the existence of $\tilde{K}^{4}>0$ such that
\begin{equation}\label{e990}
I_{2}\le\tilde{K}^4\exp\left(-\frac{k_2}{2\log(q)}\log^2|\epsilon|\right)|\epsilon|^{\tilde{K}^{1}},
\end{equation}
for all $t\in\mathcal{T}$, $z\in H_{\beta'}$, and $\epsilon\in\mathcal{E}_{p}\cap\mathcal{E}_{p+1}$. We now provide upper bounds for the quantity
$$I_{3}:=\left|\frac{1}{(2\pi)^{1/2}}\frac{1}{\pi_{q^{1/k_2}}}\int_{-\infty}^{\infty}\int_{\mathcal{C}_{\widetilde{\rho},\mathfrak{d}_p,\mathfrak{d}_{p+1}}}\frac{w_{k_2}^{\mathfrak{d}_p,\mathfrak{d}_{p+1}}(u,m,\epsilon)}{\Theta_{q^{1/k_2}}\left(\frac{u}{\epsilon t}\right)}\exp(izm)\frac{du}{u}dm\right|.$$
From the construction of $w_{k_2}^{\mathfrak{d}_p,\mathfrak{d}_{p+1}}(\tau,m,\epsilon)$, we have
$$
|w_{k_2}^{\mathfrak{d}_p,\mathfrak{d}_{p+1}}(u,m,\epsilon)| \leq \tilde{C}_{w_{k_1}^{\mathfrak{d}_p}}\frac{1}{(1+|m|)^{\mu}}e^{-\beta|m|},
$$
for some $\tilde{C}_{w_{k_1}^{\mathfrak{d}_p}}>0$, valid for $u\in \mathcal{C}_{\widetilde{\rho},\mathfrak{d}_p,\mathfrak{d}_{p+1}}$, $m\in\mathbb{R}$ and $\epsilon\in D(0,\epsilon_0)$.

The estimates with (\ref{eq3}) allow us to obtain the existence of $\tilde{C}^{\mathfrak{d}_{p},\mathfrak{d}_{p+1}}_{w_{k_2}}>0$ such that
$$I_3\le \tilde{C}_{w_{k_2}}^{\mathfrak{d}_{p},\mathfrak{d}_{p+1}}\int_{-\infty}^{\infty}\frac{e^{-\beta|m|-m\Im(z)}}{(1+|m|)^{\mu}}dm|\mathfrak{d}_{p+1}-\mathfrak{d}_{p}||t|^{1/2}\exp\left(-\frac{k_2\log^2\left(\frac{\widetilde{\rho}}{|\epsilon t|}\right)}{2\log(q)}\right),$$
for all $t\in\mathcal{T}$, $z\in H_{\beta'}$, and $\epsilon\in\mathcal{E}_{p}\cap\mathcal{E}_{p+1}$. We can follow analogous arguments as in the previous steps to provide upper estimates of the expression
$$|t|^{1/2}\exp\left(-\frac{k_2\log^2\left(\frac{\widetilde{\rho}}{|\epsilon t|}\right)}{2\log(q)}\right).$$
Indeed, 
\begin{align*}
|t|^{1/2}\exp\left(-\frac{k_2\log^2\left(\frac{\widetilde{\rho}}{|\epsilon t|}\right)}{2\log(q)}\right)&=\exp\left(-\frac{k_2\log^2(\widetilde{\rho})}{2\log(q)}\right)|\epsilon|^{\frac{k_2\log(\widetilde{\rho})}{\log(q)}}|t|^{\frac{k_2\log(\widetilde{\rho})}{\log(q)}}\\
&\times \exp\left(\frac{k_2}{2\log(q)}(-\log^2|\epsilon|-2\log|\epsilon|\log|t|-\log^2|t|)\right)|t|^{1/2}.
\end{align*}

From the assumption $0\le \epsilon_0 <1$ we check that
$$\exp\left(-\frac{k_2}{\log(q)}\log|\epsilon|\log|t|\right)\le|\epsilon|^{-\frac{k_2}{\log(q)}\log(r_{\mathcal{T}})},$$
for $t\in\mathcal{T}$, $\epsilon\in\mathcal{E}_{p}\cap\mathcal{E}_{p+1}$. Gathering (\ref{e973}), we get the existence of $\tilde{K}^{5}\in\R$, $\tilde{K}^{6}>0$ such that
$$|t|^{1/2}\exp\left(-\frac{k_2\log^2\left(\frac{\widetilde{\rho}}{|\epsilon t|}\right)}{2\log(q)}\right)\le \tilde{K}^{6}\exp\left(-\frac{k_2}{2\log(q)}\log^2|\epsilon|\right)|\epsilon|^{\tilde{K}^{5}},$$
to conclude that 
\begin{equation}\label{e1009}
I_{3}\le \tilde{K}^{7}\exp\left(-\frac{k_2}{2\log(q)}\log^2|\epsilon|\right)|\epsilon|^{\tilde{K}^{5}},
\end{equation}
for some $\tilde{K}^{7}>0$, all $t\in\mathcal{T}$, $z\in H_{\beta'}$, and $\epsilon\in\mathcal{E}_{p}\cap\mathcal{E}_{p+1}$.
We conclude the proof of this result in view of (\ref{e983}), (\ref{e990}), (\ref{e1009}) and the decomposition (\ref{e943}).

%In order to obtain analogous estimates for the forcing term $f^{\mathfrak{d}_{p}}$, one can follow analogous estimates as for $u^{\mathfrak{d}_{p}}$ under the consideration of the estimates (\ref{e197}) and by the application of Lemma~\ref{lema5b} with $Y_{k_1}=\Psi_{k_1}$. Observe that $\mathcal{R}_{\mathfrak{d}_p,\tilde{\delta}}$ (resp. $\mathcal{R}_{\mathfrak{d}_{p+1},\tilde{\delta}}$) contains an infinite sector of small opening and bisecting direction $\mathfrak{d}_p$ (resp. $\mathfrak{d}_{p+1}$). 
\end{proof}
 \begin{lemma}\label{lema1031}
Let $0\le p \le \varsigma -1$. Under the hypotheses of Theorem~\ref{teo872}, assume that ${U_{\mathfrak{d}_{p}}\cap U_{\mathfrak{d}_{p+1}}=\emptyset}$. Then, there exist $K_p^{\mathcal{L}}>0$, $M_p^{\mathcal{L}}\in\R$ such that
\begin{multline*}
\left|\mathcal{L}_{q;1/\kappa}^{\mathfrak{d}_{p+1}}(w_{k_1}^{\mathfrak{d}_{p+1}})(\tau,m,\epsilon)-\mathcal{L}_{q;1/\kappa}^{\mathfrak{d}_{p}}(w_{k_1}^{\mathfrak{d}_{p}})(\tau,m,\epsilon)\right| \\
\le K_p^{\mathcal{L}}e^{-\beta|m|}(1+|m|)^{-\mu}\exp\left(-\frac{\kappa}{2\log(q)}\log^2|\tau|\right)|\tau|^{M_p^{\mathcal{L}}},
\end{multline*}
for every $\epsilon\in(\mathcal{E}_{p}\cap\mathcal{E}_{p+1})$, $\tau\in(\mathcal{R}^{b}_{\mathfrak{d}_{p}}\cap\mathcal{R}^{b}_{\mathfrak{d}_{p+1}})$ and $m\in\R$.
\end{lemma}
\begin{proof}
We first recall that, without loss of generality, the intersection $\mathcal{R}_{\mathfrak{d}_{p}}^b\cap \mathcal{R}_{\mathfrak{d}_{p+1}}^b$ can be assumed to be a nonempty set because one  can vary $\tilde{\delta}$ in advance to be as close to 0 as desired.

Analogous arguments as in the beginning of the proof of Proposition~\ref{prop925} allow us to write
\begin{multline}\label{e943b}
\mathcal{L}_{q;1/\kappa}^{\mathfrak{d}_{p+1}}(w_{k_1}^{\mathfrak{d}_{p+1}})(\tau,m,\epsilon)-\mathcal{L}_{q;1/\kappa}^{\mathfrak{d}_{p}}(w_{k_1}^{\mathfrak{d}_{p}})(\tau,m,\epsilon)\\
=\frac{1}{\pi_{q^{1/\kappa}}}\int_{L_{\mathfrak{d}_{p+1},\widetilde{\rho}}}\frac{w^{\mathfrak{d}_{p+1}}_{k_1}(u,m,\epsilon)}{\Theta_{q^{1/\kappa}}\left(\frac{u}{\tau}\right)}\frac{du}{u}\\ 
-\frac{1}{\pi_{q^{1/\kappa}}}\int_{L_{\mathfrak{d}_p,\widetilde{\rho}}}\frac{w^{\mathfrak{d}_{p}}_{k_1}(u,m,\epsilon)}{\Theta_{q^{1/\kappa}}\left(\frac{u}{\tau}\right)}\frac{du}{u}\\
+\frac{1}{\pi_{q^{1/\kappa}}}\int_{\mathcal{C}_{\widetilde{\rho},\mathfrak{d}_p,\mathfrak{d}_{p+1}}}\frac{w_{k_1}(u,m,\epsilon)}{\Theta_{q^{1/\kappa}}\left(\frac{u}{\tau}\right)}\frac{du}{u},
\end{multline}
where $\widetilde{\rho}$, $L_{\mathfrak{d}_{p},\widetilde{\rho}}$, $L_{\mathfrak{d}_{p+1},\widetilde{\rho}}$ and $\mathcal{C}_{\widetilde{\rho},\mathfrak{d}_p,\mathfrak{d}_{p+1}}$ are constructed in Proposition~\ref{prop925}.

In view of (\ref{e869}) and (\ref{eq3}), one has

\begin{multline*}
I^{\mathcal{L}}_1:=\left|\frac{1}{\pi_{q^{1/\kappa}}}\int_{L_{\mathfrak{d}_{p},\widetilde{\rho}}}\frac{w^{\mathfrak{d}_{p}}_{k_1}(u,m,\epsilon)}{\Theta_{q^{1/\kappa}}\left(\frac{u}{\tau}\right)}\frac{du}{u}
\right|\\
\le \frac{C_{w^{\mathfrak{d}_{p}}_{k_1}}}{C_{q,\kappa}\tilde{\delta}}\frac{|\tau|^{1/2}}{(1+|m|)^{\mu}}e^{-\beta|m|}\int_{\widetilde{\rho}}^{\infty}\frac{\exp\left(\frac{\kappa\log^2|re^{i\mathfrak{d}_{p}}+\delta|}{2\log(q)}+\alpha\log|re^{i\mathfrak{d}_{p}}+\delta|\right)}{\exp\left(\frac{\kappa}{2}\frac{\log^2\left(\frac{r}{|\tau|}\right)}{\log(q)}\right)}\frac{dr}{r^{3/2}}\\
\le K_{p,1}^{\mathcal{L}}|\tau|^{1/2}(1+|m|)^{-\mu}e^{-\beta|m|}\int_{\widetilde{\rho}}^{\infty}\frac{\exp\left(\frac{\kappa\log^2r}{2\log(q)}+\alpha\log r\right)}{\exp\left(\frac{\kappa}{2}\frac{\log^2\left(\frac{r}{|\tau|}\right)}{\log(q)}\right)}\frac{dr}{r^{3/2}}
\end{multline*}
for some $K_{p,1}^{\mathcal{L}}>0$. Usual calculations, and taking into account the choice of $\alpha$ in Definition \ref{defi2}, one derives the previous expression equals
$$K_{p,1}^{\mathcal{L}}|\tau|^{1/2}(1+|m|)^{-\mu}e^{-\beta|m|}\exp\left(-\frac{\kappa}{2\log(q)}\log^2|\tau|\right)\int_{\widetilde{\rho}}^{\infty}r^{\frac{\kappa\log|\tau|}{\log(q)}+\alpha-3/2}dr.$$
Besides, we observe from direct computations that
$$\int_{\tilde{\rho}}^{\infty}r^{\frac{\kappa \log|\tau|}{\log(q)}+\alpha-\frac{3}{2}}dr$$
is upper bounded by a constant times $|\tau|^{\Delta}$, for every $\Delta<\frac{\kappa}{\log(q)}\log(\tilde{\rho})$. This yields
\begin{equation}\label{e1058}
I^{\mathcal{L}}_1\le K_{p,2}^{\mathcal{L}}(1+|m|)^{-\mu}e^{-\beta|m|}\exp\left(-\frac{\kappa}{2\log(q)}\log^2|\tau|\right)|\tau|^{\Delta+\frac{1}{2}},
\end{equation}
for some $K_{p,2}^{\mathcal{L}}>0$.
Analogous arguments allow us to obtain the existence of $K_{p,3}^{\mathcal{L}}>0$
such that
\begin{multline}\label{e1064}
\quad\quad\quad\quad I^{\mathcal{L}}_{2}:=\left|\frac{1}{\pi_{q^{1/\kappa}}}\int_{L_{\mathfrak{d}_{p+1},\widetilde{\rho}}}\frac{w^{\mathfrak{d}_{p+1}}_{k_1}(u,m,\epsilon)}{\Theta_{q^{1/\kappa}}\left(\frac{u}{\tau}\right)}\frac{du}{u}
\right|\\
\le K_{p,3}^{\mathcal{L}}(1+|m|)^{-\mu}e^{-\beta|m|}\exp\left(-\frac{\kappa}{2\log(q)}\log^2|\tau|\right)|\tau|^{\Delta+\frac{1}{2}},
\end{multline}
for $\Delta$ as above. 

We write 
$$I^{\mathcal{L}}_{3}:=\left|\frac{1}{\pi_{q^{1/\kappa}}}\int_{\mathcal{C}_{\widetilde{\rho},\mathfrak{d}_p,\mathfrak{d}_{p+1}}}\frac{w_{k_1}(u,m,\epsilon)}{\Theta_{q^{1/\kappa}}\left(\frac{u}{\tau}\right)}\frac{du}{u}\right|.$$
Regarding  (\ref{e869}) and (\ref{eq3}), one derives that
\begin{align*}
I^{\mathcal{L}}_{3}&\le\frac{C_{w^{\mathfrak{d}_{p}}_{k_1}}}{\pi_{q^{1/\kappa}}}\frac{e^{-\beta|m|}}{(1+|m|)^{\mu}}\frac{|\tau|^{1/2}}{\widetilde{\rho}^{1/2}C_{q,\kappa}\tilde{\delta}}\int_{\mathfrak{d}_p}^{\mathfrak{d}_{p+1}}\frac{\exp\left(\frac{\kappa\log^2|\widetilde{\rho}e^{i\theta}+\delta|}{2\log(q)}+\alpha\log|\widetilde{\rho}e^{i\theta}+\delta|\right)}{\exp\left(\frac{\kappa}{2}\frac{\log^2\left(\frac{\widetilde{\rho}}{|\tau|}\right)}{\log(q)}\right)}d\theta\\
&\le K_{p,4}^{\mathcal{L}}|\tau|^{1/2}\frac{e^{-\beta|m|}}{(1+|m|)^{\mu}}\exp\left(-\frac{\kappa}{2}\frac{\log^2\left(\frac{\widetilde{\rho}}{|\tau|}\right)}{\log(q)}\right)
\end{align*}
with 
$$K_{p,4}^{\mathcal{L}}=|\mathfrak{d}_{p+1}-\mathfrak{d}_{p}|\frac{C_{w^{\mathfrak{d}_{p}}_{k_1}}}{\pi_{q^{1/\kappa}}}\frac{1}{\widetilde{\rho}^{1/2}C_{q,\kappa}\tilde{\delta}}\exp\left(\frac{\kappa\log^2(\widetilde{\rho}+\delta)}{2\log(q)}+\alpha\log(\widetilde{\rho}+\delta)\right).$$
Let $K_{p,5}^{\mathcal{L}}=K_{p,4}^{\mathcal{L}}\exp(-\frac{\kappa}{2\log(q)}\log^2(\widetilde{\rho}))$. It is straightforward to check that
\begin{equation}\label{e1076}
I^{\mathcal{L}}_{3}\le K_{p,5}^{\mathcal{L}}|\tau|^{1/2+\frac{\kappa\log(\widetilde{\rho})}{\log(q)}}\frac{e^{-\beta|m|}}{(1+|m|)^{\mu}}\exp\left(-\frac{\kappa}{2}\frac{\log^2|\tau|}{\log(q)}\right).
\end{equation}

From (\ref{e1058}), (\ref{e1064}) and (\ref{e1076}), put into (\ref{e943b}), we conclude the result.
\end{proof}

\begin{prop}\label{prop1047}
Let $0\le p \le \varsigma -1$. Under the hypotheses of Theorem~\ref{teo872}, assume that $U_{\mathfrak{d}_{p}}\cap U_{\mathfrak{d}_{p+1}}=\emptyset$. Then, there exist $K_3>0$ and $K_4\in\R$ such that
\begin{equation}\label{e1049aa}
|u^{\mathfrak{d}_{p+1}}(t,z,\epsilon)-u^{\mathfrak{d}_{p}}(t,z,\epsilon)|\le K_{3}\exp\left(-\frac{k_1}{2\log(q)}\log^2|\epsilon|\right)|\epsilon|^{K_4},
\end{equation}
%&\hspace{3cm}|f^{\mathfrak{d}_{p+1}}(t,z,\epsilon)-f^{\mathfrak{d}_{p}}(t,z,\epsilon)|\le K_{3}\exp\left(-\frac{k_1}{2\log(q)}\log^2|\epsilon|\right)|\epsilon|^{K_4},
%\end{align}
for every $t\in\mathcal{T}$, $z\in H_{\beta'}$, and $\epsilon\in\mathcal{E}_{p}\cap\mathcal{E}_{p+1}$. 
\end{prop}
\begin{proof}
Let $0\le p\le \varsigma-1$. Under the assumptions of the statement, we observe that one can not proceed as in the proof of Proposition~\ref{prop925} for there does not exist a common function for both indices $p$ and $p+1$, defined in $\mathcal{R}^{b}_{\mathfrak{d}_{p}}\cup \mathcal{R}^{b}_{\mathfrak{d}_{p+1}}$ in the variable of integration, when applying $q$-Laplace transform. However, one can use the analytic continuation property and write the difference $u^{\mathfrak{d}_{p+1}}-u^{\mathfrak{d}_{p}}$ as follows. Let $\widetilde{\rho}>0$ be such that $\widetilde{\rho}e^{i\mathfrak{d}_{p}}\in\mathcal{R}^{b}_{\mathfrak{d}_{p}}$ and $\widetilde{\rho}e^{i\mathfrak{d}_{p+1}}\in\mathcal{R}^{b}_{\mathfrak{d}_{p+1}}$, and let $\theta_{p,p+1}\in\R$ be such that $\widetilde{\rho}e^{\theta_{p,p+1}}$ lies in both $\mathcal{R}^{b}_{\mathfrak{d}_{p}}$ and  $\mathcal{R}^{b}_{\mathfrak{d}_{p+1}}$. We write $u^{\mathfrak{d}_{p+1}}(t,z,\epsilon)-u^{\mathfrak{d}_{p}}(t,z,\epsilon)$ as follows

\begin{multline}
u^{\mathfrak{d}_{p+1}}(t,z,\epsilon)-u^{\mathfrak{d}_{p}}(t,z,\epsilon)\\
=\frac{1}{(2\pi)^{1/2}}\frac{1}{\pi_{q^{1/k_2}}}\int_{-\infty}^{\infty}\int_{L_{\mathfrak{d}_{p+1},\widetilde{\rho}}}\frac{w^{\mathfrak{d}_{p+1}}_{k_2}(u,m,\epsilon)}{\Theta_{q^{1/k_2}}\left(\frac{u}{\epsilon t}\right)}\exp(izm)\frac{du}{u}dm\hspace{3cm}\\ 
\hspace{2cm}-\frac{1}{(2\pi)^{1/2}}\frac{1}{\pi_{q^{1/k_2}}}\int_{-\infty}^{\infty}\int_{L_{\mathfrak{d}_p,\widetilde{\rho}}}\frac{w^{\mathfrak{d}_{p}}_{k_2}(u,m,\epsilon)}{\Theta_{q^{1/k_2}}\left(\frac{u}{\epsilon t}\right)}\exp(izm)\frac{du}{u}dm\\
\hspace{4cm}-\frac{1}{(2\pi)^{1/2}}\frac{1}{\pi_{q^{1/k_2}}}\int_{-\infty}^{\infty}\int_{\mathcal{C}_{\widetilde{\rho},\theta_{p,p+1},\mathfrak{d}_{p+1}}}\frac{w^{\mathfrak{d}_p,\mathfrak{d}_{p+1}}_{k_2}(u,m,\epsilon)}{\Theta_{q^{1/k_2}}\left(\frac{u}{\epsilon t}\right)}\exp(izm)\frac{du}{u}dm\label{e1049}\\
\hspace{4cm}+\frac{1}{(2\pi)^{1/2}}\frac{1}{\pi_{q^{1/k_2}}}\int_{-\infty}^{\infty}\int_{\mathcal{C}_{\widetilde{\rho},\theta_{p,p+1},\mathfrak{d}_{p}}}\frac{w^{\mathfrak{d}_p,\mathfrak{d}_{p+1}}_{k_2}(u,m,\epsilon)}{\Theta_{q^{1/k_2}}\left(\frac{u}{\epsilon t}\right)}\exp(izm)\frac{du}{u}dm\\
+\frac{1}{(2\pi)^{1/2}}\frac{1}{\pi_{q^{1/k_2}}}\int_{-\infty}^{\infty}\int_{L_{0,\widetilde{\rho},\theta_{p,p+1}}}\frac{\mathcal{L}_{q;1/\kappa}^{\mathfrak{d}_{p+1}}(w_{k_1}^{\mathfrak{d}_{p+1}})(\tau,m,\epsilon)-\mathcal{L}_{q;1/\kappa}^{\mathfrak{d}_{p}}(w_{k_1}^{\mathfrak{d}_{p}})(\tau,m,\epsilon)}{\Theta_{q^{1/k_2}}\left(\frac{u}{\epsilon t}\right)}\exp(izm)\frac{du}{u}dm.
\end{multline}
Here, we have denoted $L_{\mathfrak{d}_{j},\widetilde{\rho}}=[\widetilde{\rho},+\infty)e^{i\mathfrak{d}_{j}}$ for $j\in\{p,p+1\}$, $\mathcal{C}_{\widetilde{\rho},\theta_{p,p+1},\mathfrak{d}_{p+1}}$ is the arc of circle connecting $\widetilde{\rho}e^{i\mathfrak{d}_{p+1}}$ with $\widetilde{\rho}e^{i\theta_{p,p+1}}$, $\mathcal{C}_{\widetilde{\rho},\theta_{p,p+1},\mathfrak{d}_{p}}$ is the arc of circle connecting $\widetilde{\rho}e^{i\mathfrak{d}_{p}}$ with $\widetilde{\rho}e^{i\theta_{p,p+1}}$, $L_{0,\widetilde{\rho},\theta_{p,p+1}}=[0,\widetilde{\rho}]e^{i\theta_{p,p+1}}$, as it is shown in the following figure.

\begin{figure}[h]
	\centering
		\includegraphics[width=0.50\textwidth]{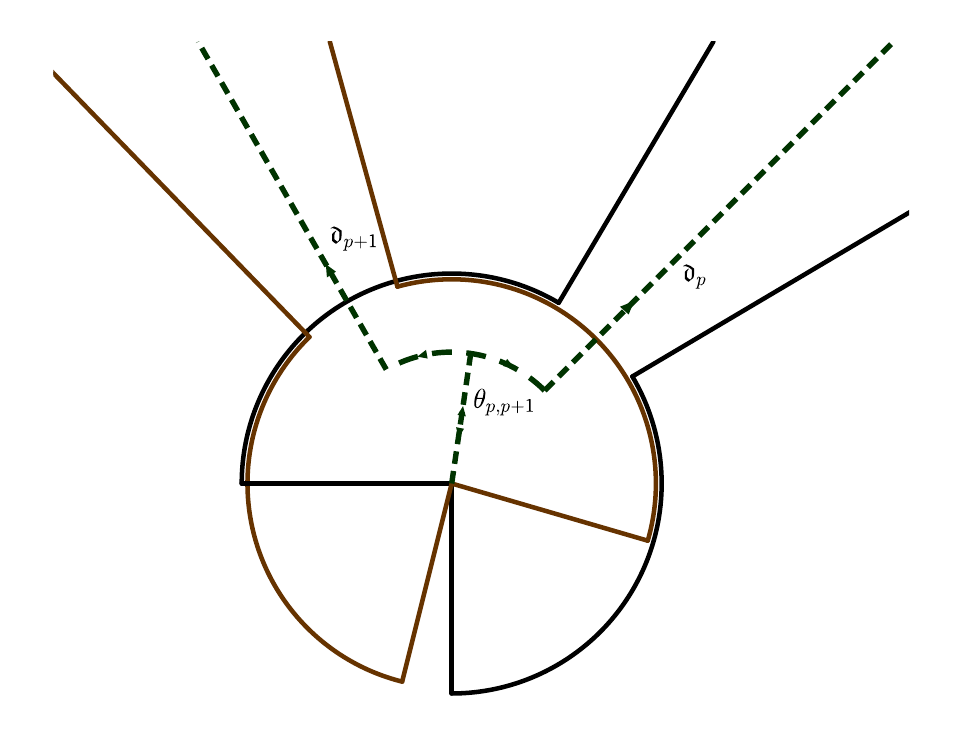}
	\label{fig3}\label{fig2}
	\caption{Deformation of the path of integration, second case.}
\end{figure}

Following the same line of arguments as those in the proof of Proposition~\ref{prop925}, we can guarantee the existence of $\hat{K}^{j}>0$ and $\hat{K}^{k}\in\R$ for $1\le j\le 4$ and $5\le k\le 8$ such that

\begin{multline*}
J_{1}:=\left|\frac{1}{(2\pi)^{1/2}}\frac{1}{\pi_{q^{1/k_2}}}\int_{-\infty}^{\infty}\int_{L_{\mathfrak{d}_{p+1},\widetilde{\rho}}}\frac{w^{\mathfrak{d}_{p+1}}_{k_2}(u,m,\epsilon)}{\Theta_{q^{1/k_2}}\left(\frac{u}{\epsilon t}\right)}\exp(izm)\frac{du}{u}dm\right|\\
\le \hat{K}_{1}\exp\left(-\frac{k_2}{2\log(q)}\log^2|\epsilon|\right)|\epsilon|^{\hat{K}^{5}},
\end{multline*}

\begin{multline*}
J_{2}:=\left|\frac{1}{(2\pi)^{1/2}}\frac{1}{\pi_{q^{1/k_2}}}\int_{-\infty}^{\infty}\int_{L_{\mathfrak{d}_p,\widetilde{\rho}}}\frac{w^{\mathfrak{d}_{p}}_{k_2}(u,m,\epsilon)}{\Theta_{q^{1/k_2}}\left(\frac{u}{\epsilon t}\right)}\exp(izm)\frac{du}{u}dm\right|\\
\le \hat{K}_{2}\exp\left(-\frac{k_2}{2\log(q)}\log^2|\epsilon|\right)|\epsilon|^{\hat{K}^{6}},
\end{multline*}

\begin{multline*}
J_{3}:=\left|\frac{1}{(2\pi)^{1/2}}\frac{1}{\pi_{q^{1/k_2}}}\int_{-\infty}^{\infty}\int_{\mathcal{C}_{\widetilde{\rho},\theta_{p,p+1},\mathfrak{d}_{p+1}}}\frac{w^{\mathfrak{d}_p,\mathfrak{d}_{p+1}}_{k_2}(u,m,\epsilon)}{\Theta_{q^{1/k_2}}\left(\frac{u}{\epsilon t}\right)}\exp(izm)\frac{du}{u}dm\right|\\
\le \hat{K}_{3}\exp\left(-\frac{k_2}{2\log(q)}\log^2|\epsilon|\right)|\epsilon|^{\hat{K}^{7}},
\end{multline*}

\begin{multline*}
J_{4}:=\left|\frac{1}{(2\pi)^{1/2}}\frac{1}{\pi_{q^{1/k_2}}}\int_{-\infty}^{\infty}\int_{\mathcal{C}_{\widetilde{\rho},\theta_{p,p+1},\mathfrak{d}_{p}}}\frac{w^{\mathfrak{d}_p,\mathfrak{d}_{p+1}}_{k_2}(u,m,\epsilon)}{\Theta_{q^{1/k_2}}\left(\frac{u}{\epsilon t}\right)}\exp(izm)\frac{du}{u}dm\right|\\
\le \hat{K}_{4}\exp\left(-\frac{k_2}{2\log(q)}\log^2|\epsilon|\right)|\epsilon|^{\hat{K}^{8}}
\end{multline*}

We now give estimates for 
\begin{multline*}
J_{5}:=\frac{1}{(2\pi)^{1/2}}\frac{1}{\pi_{q^{1/k_2}}} \\
\times \left| \int_{-\infty}^{\infty}\int_{L_{0,\widetilde{\rho},\theta_{p,p+1}}}\frac{\mathcal{L}_{q;1/\kappa}^{\mathfrak{d}_{p+1}}(w_{k_1}^{\mathfrak{d}_{p+1}})(u,m,\epsilon)-\mathcal{L}_{q;1/\kappa}^{\mathfrak{d}_{p}}(w_{k_1}^{\mathfrak{d}_{p}})(u,m,\epsilon)}{\Theta_{q^{1/k_2}}\left(\frac{u}{\epsilon t}\right)}\exp(izm)\frac{du}{u}dm\right|.
\end{multline*}
In view of Lemma~\ref{lema1031} and (\ref{eq3}), one has 
$$J_{5}\le\frac{K_p^{\mathcal{L}}}{(2\pi)^{1/2}}\frac{1}{\pi_{q^{1/k_2}}}\int_{-\infty}^{\infty}e^{-\beta|m|-\Im(z)m}\frac{dm}{(1+|m|)^{\mu}}\int_{0}^{\widetilde{\rho}}\frac{\exp\left(-\frac{\kappa}{2\log(q)}\log^2|u|\right)|u|^{M_p^{\mathcal{L}}}}{C_{q,k_2}\tilde{\delta}\exp\left(\frac{k_2}{2}\frac{\log^2\left|\frac{u}{\epsilon t}\right|}{\log(q)}\right)\left|\frac{u}{\epsilon t}\right|^{1/2}}\frac{d|u|}{|u|}.$$
We recall that $z\in H_{\beta'}$ for some $\beta'<\beta$. Then, there exists $K_{31}>0$ such that
$$J_{5}\le\frac{K_p^{\mathcal{L}}K_{31}}{(2\pi)^{1/2}}\frac{|\epsilon|^{1/2}r_{\mathcal{T}}^{1/2}}{\pi_{q^{1/k_2}}C_{q,k_2}\tilde{\delta}}\int_{0}^{\widetilde{\rho}}\frac{\exp\left(-\frac{\kappa}{2\log(q)}\log^2|u|\right)|u|^{M_p^{\mathcal{L}}}}{\exp\left(\frac{k_2}{2}\frac{\log^2\left|\frac{u}{\epsilon t}\right|}{\log(q)}\right)}\frac{d|u|}{|u|^{3/2}}.$$

We now proceed to prove the expression
$$\int_{0}^{\widetilde{\rho}}\frac{\exp\left(-\frac{\kappa}{2\log(q)}\log^2|u|\right)}{\exp\left(\frac{k_2}{2}\frac{\log^2\left|\frac{u}{\epsilon t}\right|}{\log(q)}\right)}\exp\left(\frac{k_1}{2\log(q)}\log^2|\epsilon|\right)\frac{d|u|}{|u|^{3/2-M_p^{\mathcal{L}}}}$$
is upper bounded by a positive constant times a certain power of $|\epsilon|$ for every $\epsilon\in(\mathcal{E}_{p}\cap \mathcal{E}_{p+1})$ and $t\in\mathcal{T}$. This concludes the existence of $K_{32}>0$ such that
\begin{equation}\label{e1109}J_{5}\le K_{32}|\epsilon|^{1/2}\exp\left(-\frac{k_1}{2\log(q)}\log^2|\epsilon|\right),
\end{equation}
for every $\epsilon\in(\mathcal{E}_{p}\cap \mathcal{E}_{p+1})$, $t\in\mathcal{T}$ and $z\in H_{\beta'}$.

Indeed, we have 
$$\int_{0}^{\widetilde{\rho}}\frac{\exp\left(-\frac{\kappa}{2\log(q)}\log^2|u|\right)}{\exp\left(\frac{k_2}{2}\frac{\log^2\left(\frac{|u|}{|\epsilon t|}\right)}{\log(q)}\right)}\exp\left(\frac{k_1}{2\log(q)}\log^2|\epsilon |\right)\frac{d|u|}{|u|^{3/2-M_p^{\mathcal{L}}}}$$
equals
\begin{multline}\label{e1111}
\exp\left(\frac{k_1}{2\log(q)}\log^2|\epsilon|-\frac{k_2}{2\log(q)}\log^2|\epsilon t|\right)\\
\times \int_{0}^{\widetilde{\rho}}\exp\left(-\frac{(\kappa+k_2)}{2\log(q)}\log^2|u|\right)|u|^{\frac{k_2\log|\epsilon t|}{\log(q)}-\frac{3}{2}+M_p^{\mathcal{L}}}d|u|.
\end{multline}
Given $m_1\in\R$ and $m_2>0$, the function $[0,\infty)\ni x\mapsto H(x)=x^{m_1}\exp(-m_2\log^2(x))$ attains its maximum value at $x_0=\exp(\frac{m_1}{2m_2})$ with $H(x_0)=\exp(\frac{m_1^2}{4m_2})$. This yields an upper bound for the integrand in (\ref{e1111}); the expression in (\ref{e1111}) is estimated from above by
\begin{multline}
\widetilde{\rho}\exp\left(\frac{(M_{p}^{\mathcal{L}}-3/2)^2\log(q)}{2(\kappa+k_2)}\right) \exp\left(\frac{1}{2\log(q)}(\frac{k_2^2}{\kappa+k_2}-k_2+k_1)\log^2|\epsilon|\right)\\
\times\exp\left(\frac{1}{2\log(q)}(\frac{k_2^2}{\kappa+k_2}-k_2)\log^2|t|\right)|t|^{\frac{k_2(M_{p}^{\mathcal{L}}-3/2)}{\kappa+k_2}}\\
\times \exp\left(\frac{1}{\log(q)}(\frac{k_2^2}{\kappa+k_2}-k_2)\log|\epsilon|\log|t|\right)|\epsilon|^{\frac{k_2(M_{p}^{\mathcal{L}}-3/2)}{\kappa+k_2}}.\label{e1156}
\end{multline}
The second line in (\ref{e1156}) is upper bounded for every $t$ because $\frac{k_2^2}{\kappa+k_2}<k_2$ and also, one has an upper bound for $\exp\left(\frac{1}{\log(q)}(\frac{k_2^2}{\kappa+k_2}-k_2)\log|\epsilon|\log|t|\right)$ is 1. Regarding Definition \ref{defi2}, and taking into account that
$$\frac{k_2^2}{\kappa+k_2}-k_2=-k_1,$$
the expression (\ref{e1156}) is upper bounded by
$$K_{33}|\epsilon|^{\frac{k_2(M_{p}^{\mathcal{L}}-3/2)}{\kappa+k_2}}$$
for some $K_{33}>0$. The conclusion is achieved. The result follows from (\ref{e1049}), the estimates $J_{1}$ to $J_{4}$, and (\ref{e1109}). %The proof for the estimates of $f^{\mathfrak{d}_{p}}$ is analogous as that for $u^{\mathfrak{d}_{p}}$: take into account (\ref{e1139}), (\ref{e413}) and (\ref{e1140}).
\end{proof}

\section{Existence of formal series solutions in the complex parameter and asymptotic expansion in two levels}\label{seccion6}

In the first part of this section, we remind two $q$-analogs of Ramis-Sibuya theorem from \cite{lamaq,ma}. This result provides the tool to guarantee the existence of a formal power series in the perturbation parameter which formally solves the main problem and such that it asymptotically represents the analytic solution of that equation.

This asymptotic representation is held in the sense of $q$-asymptotic expansions of certain positive order.

\begin{defin}
Let $V$ be a bounded open sector with vertex at 0 in $\C$. Let $(\mathbb{F},\left\|\cdot\right\|_{\mathbb{F}})$ be a complex Banach space. Let $q\in\R$ with $q>1$ and let $k$ be a positive integer. We say that a holomorphic function $f:V\to\mathbb{F}$ admits the formal power series $\hat{f}(\epsilon)=\sum_{n\ge0}f_n\epsilon^n\in\mathbb{F}[[\epsilon]]$ as its $q$-Gevrey asymptotic expansion of order $1/k$ if for every open subsector $U$ with $(\overline{U}\setminus\{0\})\subseteq V$, there exist $A,C>0$ such that
$$\left\|f(\epsilon)-\sum_{n=0}^{N}f_n\epsilon^n\right\|_{\mathbb{F}}\le CA^{N+1}q^{\frac{N(N+1)}{2k}}|\epsilon|^{N+1},$$
for every $\epsilon\in U$, and $N\ge0$.
\end{defin}

The set of functions which admit null $q$-Gevrey asymptotic expansion of certain positive order are characterized as follows. The proof of this result, already stated in~\cite{ma}, provides the $q$-analog of Theorem XI-3-2 in~\cite{hssi}.

\begin{lemma}
A holomorphic function $f:V\to\mathbb{F}$ admits the null formal power series $\hat{0}\in\mathbb{F}[[\epsilon]]$ as its $q$-Gevrey asymptotic expansion of order $1/k$ if and only if for every open subsector $U$ with $(\overline{U}\setminus\{0\})\subseteq V$ there exist constants $K_1\in\R$ and $K_2>0$ with
$$\left\|f(\epsilon)\right\|_{\mathbb{F}}\le K_2\exp\left(-\frac{k}{2\log(q)}\log^2|\epsilon|\right)|\epsilon|^{K_1},$$
for all $\epsilon\in U$.
\end{lemma}

%We recall the one-level version of the $q$-analog of Ramis-Sibuya theorem proved in~\cite{ma}.

%\begin{theo}($q$-RS)\label{teoqrs}
%Let $(\mathbb{F},\left\|\cdot\right\|_{\mathbb{F}})$ be a Banach space and $(\mathcal{E}_{p})_{0\le p\le \varsigma-1}$ be a good covering in $\C^{\star}$. For every $0\le p\le \varsigma-1$, let $G_{p}(\epsilon)$ be a holomorphic function from $\mathcal{E}_{p}$ into $\mathbb{F}$ and let the cocycle $\Delta_{p}(\epsilon)=G_{p+1}(\epsilon)-G_{p}(\epsilon)$ be a holomorphic function from $Z_{p}=\mathcal{E}_{p}\cap\mathcal{E}_{p+1}$ into $\mathbb{F}$ (we put $\mathcal{E}_{\varsigma}=\mathcal{E}_{0}$ and $G_{\varsigma}=G_{0}$). We also make the further assumptions:
%\begin{enumerate}
%\item[\textbf{1)}] The functions $G_p(\epsilon)$ are bounded as $\epsilon$ tends to 0 on $\mathcal{E}_{p}$ for every $0\le p\le \varsigma -1$.
%\item[\textbf{2)}] For all $0\le p\le \varsigma-1$, the function $\Delta_{p}(\epsilon)$ is $q$-exponentially flat of order $k$ on $Z_p$, i.e., there exist constants $C_{p}^1\in\R$ and $C_p^2>0$ such that
%$$\left\|\Delta_p(\epsilon)\right\|_{\mathbb{F}}\le C_p^2|\epsilon|^{C_p^1}\exp\left(-\frac{k}{2\log(q)}\log^2|\epsilon|\right),$$
%for every $\epsilon\in Z_p$, all $0\le p\le \varsigma-1$.
%\end{enumerate}
%Then, there exists a formal power series $\hat{G}(\epsilon)\in\mathbb{F}[[\epsilon]]$ which is the common $q$-Gevrey asymptotic expansion of order $1/k$ of the functions $G_p(\epsilon)$ on $\mathcal{E}_{p}$, which is common for all $0\le p\le \varsigma-1$.
%\end{theo}

The next result leans on the one level version of the $q$-analog of Ramis Sibuya theorem, stated in~\cite{ma}, provides a two level result in this framework. See \cite{lamaq} for a proof.

\begin{theo}\label{teo1215}
Let $(\mathbb{F},\left\|\cdot\right\|_{\mathbb{F}})$ be a Banach space and $(\mathcal{E}_{p})_{0\le p\le \varsigma-1}$ be a good covering in $\C^{\star}$. Let $0<k_1<k_2$, consider a holomorphic function $G_{p}:\mathcal{E}_{i}\to\mathbb{F}$ for every $0\le p\le \varsigma-1$ and put $\Delta_{p}(\epsilon)=G_{p+1}(\epsilon)-G_{p}( \epsilon)$ for every $\epsilon\in Z_{p}:=\mathcal{E}_{p}\cap\mathcal{E}_{p+1}$. Moreover, we assume:
\begin{enumerate}
\item[\textbf{1)}] The functions $G_p(\epsilon)$ are bounded as $\epsilon$ tends to 0 on $\mathcal{E}_{p}$ for every $0\le p\le \varsigma -1$.
\item[\textbf{2)}] There exist nonempty sets $I_1,I_2\subseteq\{0,1,\dots,\varsigma-1\}$ such that $I_1\cup I_2=\{0,1,\dots,\varsigma-1\}$ and $I_1\cap I_2=\emptyset$. Also,
\begin{itemize}
\item[-] for every $p\in I_1$ there exist constants $K_1>0$, $M_1\in\R$ such that
$$\left\|\Delta_p(\epsilon)\right\|_{\mathbb{F}}\le K_1|\epsilon|^{M_1}\exp\left(-\frac{k_1}{2\log(q)}\log^2|\epsilon|\right),\quad \epsilon\in Z_p,$$
\item[-] and, for every $p\in I_2$ there exist constants $K_2>0$, $M_2\in\R$ such that
$$\left\|\Delta_p(\epsilon)\right\|_{\mathbb{F}}\le K_2|\epsilon|^{M_2}\exp\left(-\frac{k_2}{2\log(q)}\log^2|\epsilon|\right),\quad \epsilon\in Z_p.$$
\end{itemize}
\end{enumerate}
Then, there exists a convergent power series $a(\epsilon)\in\mathbb{F}\{\epsilon\}$ defined on some neighborhood of the origin and $\hat{G}^1(\epsilon),\hat{G}^2(\epsilon)\in\mathbb{F}[[\epsilon]]$ such that $G_p$ can be written in the form
$$G_p(\epsilon)=a(\epsilon)+G_{p}^{1}(\epsilon)+G_{p}^{2}(\epsilon).$$
$G_{p}^1(\epsilon)$ is holomorphic on $\mathcal{E}_{p}$ and admits $\hat{G}^{1}(\epsilon)$ as its $q$-Gevrey asymptotic expansion of order $1/k_1$ on $\mathcal{E}_{p}$, for every $p\in I_1$; whilst $G_{p}^2(\epsilon)$ is holomorphic on $\mathcal{E}_{p}$ and admits $\hat{G}^{2}(\epsilon)$ as its $q$-Gevrey asymptotic expansion of order $1/k_2$ on $\mathcal{E}_{p}$, for every $p\in I_2$.
\end{theo}

We conclude this section with the main result in the work in which we guarantee the existence of a formal solution of the main problem (\ref{epral}), written as a formal power series in the perturbation parameter, with coefficients in an appropriate Banach space, say $\hat{u}(t,z,\epsilon)$. Moreover, it represents, in some sense to be precised, each solution $u^{\mathfrak{d}_{p}}(t,z,\epsilon)$ of the problem (\ref{epral}). 

%This result is based on the existence of a common formal power series $\hat{f}(t,z,\epsilon)$ which is the $q$-Gevrey asymptotic expansion of order $1/k_1$, seen as a formal power series in the perturbation parameter $\epsilon$ with coefficients in a certain Banach space, of every $f^{\mathfrak{d}_{p}}$ on $\mathcal{E}_{p}$.

From now on, $\mathbb{F}$ stands for the Banach space of bounded holomorphic functions defined on $\mathcal{T}\times H_{\beta'}$, with the supremum norm, where $\beta'<\beta$, as above.

%\begin{lemma}\label{lema1271}
%Under the hypotheses of Theorem~\ref{teo872}, there exists a formal power series 
%$$\hat{f}(t,z,\epsilon)=\sum_{m\ge0}f_m(t,z)\frac{\epsilon^m}{m!},$$
%with $f_m(t,z)\in\mathbb{F}$ for $m\ge0$, which is the common $q$-Gevrey asymptotic expansion of order $1/k_1$ on $\mathcal{E}_{p}$ of the functions $f^{\mathfrak{d}_{p}}$, seen as holomorphic functions from $\mathcal{E}_{p}$ to $\mathbb{F}$, for all $0\le p\le \varsigma-1$.
%\end{lemma}
%\begin{proof}
%Let $0\le p\le \varsigma-1$. We consider the function $f^{\mathfrak{d}_{p}}$ constructed in (\ref{e821}), and define $G^{f}_p(\epsilon):=(t,z)\mapsto f^{\mathfrak{d}_{p}}(t,z,\epsilon)$, which is a holomorphic and bounded function from $\mathcal{E}_{p}$ into $\mathbb{F}$. Regarding (\ref{e927}) and (\ref{e1049a}), and taking into account that $k_1<k_2$, we have that (\ref{e1049a}) holds for every $0\le p\le \varsigma-1$. This yields the cocycle $\Delta_p^f(\epsilon):=G_{p+1}^f(\epsilon)-G_{p}^f(\epsilon)$ satisfies the conditions of Theorem~\ref{teoqrs} for $k=k_1$, and one concludes the result by the application of Theorem~\ref{teoqrs}.
%\end{proof}

\begin{theo}\label{teo1281}
Under the hypotheses of Theorem~\ref{teo872},  there exists a formal power series 
\begin{equation}\label{e1526}
\hat{u}(t,z,\epsilon)=\sum_{m\ge0}h_m(t,z)\frac{\epsilon^m}{m!}\in\mathbb{F}[[\epsilon]],
\end{equation}
formal solution of the equation

\begin{multline}
Q(\partial_z)\sigma_{q,t}\hat{u}(t,z,\epsilon)\\
=(\epsilon t)^{d_{D_1}}\sigma_{q,t}^{\frac{d_{D_1}}{k_1}+1}R_{D_1}(\partial_{z})\hat{u}(t,z,\epsilon)+(\epsilon t)^{d_{D_2}}\sigma_{q,t}^{\frac{d_{D_2}}{k_2}+1}R_{D_2}(\partial_{z})\hat{u}(t,z,\epsilon)\\
+\sum_{\ell=1}^{D-1}\epsilon^{\Delta_\ell}t^{d_{\ell}}\sigma_{q,t}^{\delta_{\ell}}(c_{\ell}(t,z,\epsilon)R_{\ell}(\partial_z)\hat{u}(t,z,\epsilon))+\sigma_{q,t}f(t,z,\epsilon).\label{e771b}
\end{multline}
Moreover, $\hat{u}(t,z,\epsilon)$ turns out to be the common $q$-Gevrey asymptotic expansion of order $1/k_1$ on $\mathcal{E}_{p}$ of the function $u^{\mathfrak{d}_{p}}$, seen as holomorphic function from $\mathcal{E}_{p}$ into $\mathbb{F}$, for $0\le p\le \varsigma-1$. In addition to that, $\hat{u}$ is of the form
$$\hat{u}(t,z,\epsilon)=a(t,z,\epsilon)+\hat{u}_1(t,z,\epsilon)+\hat{u}_2(t,z,\epsilon),$$
where $a(t,z,\epsilon)\in\mathbb{F}\{\epsilon\}$ and $\hat{u}_1(t,z,\epsilon),\hat{u}_2(t,z,\epsilon)\in\mathbb{F}[[\epsilon]]$ and such that for every $0\le p\le \varsigma-1$, the function $u^{\mathfrak{d}_p}$ can be written in the form
$$u^{\mathfrak{d}_{p}}(t,z,\epsilon)=a(t,z,\epsilon)+u^{\mathfrak{d}_{p}}_1(t,z,\epsilon)+u^{\mathfrak{d}_{p}}_2(t,z,\epsilon),$$
where $\epsilon\mapsto u_1^{\mathfrak{d}_{p}}(t,z,\epsilon)$ is a $\mathbb{F}$-valued function that admits $\hat{u}_1(t,z,\epsilon)$ as its $q$-Gevrey asymptotic expansion of order $1/k_1$ on $\mathcal{E}_{p}$ and also $\epsilon\mapsto u_2^{\mathfrak{d}_{p}}(t,z,\epsilon)$ is a $\mathbb{F}$-valued function that admits $\hat{u}_2(t,z,\epsilon)$ as its $q$-Gevrey asymptotic expansion of order $1/k_2$ on $\mathcal{E}_{p}$.
\end{theo}
\begin{proof}
For every $0\le p\le \varsigma -1$, one can consider the function $u^{\mathfrak{d}_{p}}(t,z,\epsilon)$ constructed in Theorem~\ref{teo872}. We define $G_{p}(\epsilon):=(t,z)\mapsto u^{\mathfrak{d}_{p}}(t,z,\epsilon)$, which is a holomorphic and bounded function from $\mathcal{E}_{p}$ into $\mathbb{F}$. In view of Proposition~\ref{prop925} and Proposition~\ref{prop1047}, one can split the set $\{0,1,\dots,\varsigma-1\}$ in two nonempty subsets of indices, $I_1$ and $I_2$ with $\{0,1,\dots,\varsigma-1\}=I_1\cup I_2$ and such that $I_1$ (resp. $I_2$) consists of all the elements in $\{0,1,\dots,\varsigma-1\}$ such that $U_{\mathfrak{d}_{p}}\cap U_{\mathfrak{d}_{p+1}}$ contains the sector $U_{\mathfrak{d}_{p},\mathfrak{d}_{p+1}}$, as defined in Proposition~\ref{prop925} (resp. $U_{\mathfrak{d}_{p}}\cap U_{\mathfrak{d}_{p+1}}=\emptyset$). From (\ref{e927}) and (\ref{e1049aa}) one can apply Theorem~\ref{teo1215} and deduce the existence of formal power series $\hat{G}^1(\epsilon),\hat{G}^2(\epsilon)\in\mathbb{F}[[\epsilon]]$, a convergent power series $a(\epsilon)\in\mathbb{F}\{\epsilon\}$ and holomorphic functions $G_{p}^1(\epsilon),G_{p}^2(\epsilon)$ defined on $\mathcal{E}_{p}$ and with values in $\mathbb{F}$ such that 
$$G_p(\epsilon)=a(\epsilon)+G^1_p(\epsilon)+G^2_p(\epsilon),$$
and for $j=1,2$, one has $G_{p}^j(\epsilon)$ admits $\hat{G}^j(\epsilon)$ as its $q$-Gevrey asymptotic expansion or order $1/k_j$ on $\mathcal{E}_{p}$. 
We put 
$$\hat{u}(t,z,\epsilon)=\sum_{m\ge0}h_m(t,z)\frac{\epsilon^m}{m!}:=a(\epsilon)+\hat{G}^1_p(\epsilon)+\hat{G}^2_p(\epsilon).$$

It only rests to prove that $\hat{u}(t,z,\epsilon)$ is the solution of (\ref{e771b}). Indeed, since $u^{\mathfrak{d}_{p}}$ admits $\hat{u}(t,z,\epsilon)$ as its $q$-Gevrey asymptotic expansion of order $1/k_1$ on $\mathcal{E}_{p}$, we have that
$$\lim_{\epsilon\to 0,\epsilon\in\mathcal{E}_{p}}\sup_{t\in\mathcal{T},z\in H_{\beta'}}|\partial_\epsilon^m u^{\mathfrak{d}_{p}}(t,z,\epsilon)- h_m(t,z)|=0,$$
for every $0\le p\le \varsigma -1$ and $m\ge0$. Let $p\in\{0,1,\dots,\varsigma-1\}$. By construction, the function $u^{\mathfrak{d}_{p}}(t,z,\epsilon)$ solves equation (\ref{e771b}). We take derivatives of order $m\ge 0$ with respect to $\epsilon$ at both sides of equation (\ref{epral}) and deduce that 

\begin{multline}
Q(\partial_z)\sigma_{q,t}(\partial_\epsilon^m u^{\mathfrak{d}_{p}})(t,z,\epsilon)\\
=\sum_{m_1+m_2=m}\frac{m!}{m_1!m_2!}\partial_\epsilon^{m_1}(\epsilon^{d_{D_1}})t^{d_{D_{1}}}\sigma_{q,t}^{\frac{d_{D_{1}}}{k_1}+1}R_{D_{1}}(\partial_{z})(\partial_{\epsilon}^{m_2} u^{\mathfrak{d}_{p}})\\
+\sum_{m_1+m_2=m}\frac{m!}{m_1!m_2!}\partial_\epsilon^{m_1}(\epsilon^{d_{D_2}})t^{d_{D_{2}}}\sigma_{q,t}^{\frac{d_{D_{2}}}{k_2}+1}R_{D_{2}}(\partial_{z})(\partial_{\epsilon}^{m_2} u^{\mathfrak{d}_{p}})\\
+\sum_{\ell=1}^{D-1}\sum_{m_1+m_2+m_3=m}\frac{m!}{m_1!m_2!m_3!}(\partial_{\epsilon}^{m_1}\epsilon^{\Delta_{\ell}})t^{d_{\ell}}
\sigma_{q,t}^{\delta_{\ell}}(\partial_{\epsilon}^{m_2}c_{\ell}(t,z,\epsilon)R_{\ell}(\partial_z)\partial_{\epsilon}^{m_3}u^{\mathfrak{d}_{p}}(t,z,\epsilon))\\
+\sigma_{q,t}(\partial_\epsilon^m f)(t,z,0),\label{e771c}
\end{multline}
for every $(t,z,\epsilon)\in\mathcal{T}\times H_{\beta'}\times\mathcal{E}_{p}$. We let $\epsilon\to0$ in (\ref{e771c}) and obtain the recursion formula
\begin{multline}
Q(\partial_z)\sigma_{q,t} h_{m}(t,z)\\
=\frac{m!}{(m-d_{D_{1}})!}t^{d_{D_{1}}}\sigma_{q,t}^{\frac{d_{D_{1}}}{k_1}+1}R_{D_{1}}(\partial_{z})(h_{m-d_{D_{1}}}(t,z))+\frac{m!}{(m-d_{D_{2}})!}t^{d_{D_{2}}}\sigma_{q,t}^{\frac{d_{D_{2}}}{k_2}+1}R_{D_{2}}(\partial_{z})(h_{m-d_{D_{2}}}(t,z))\\
+\sum_{\ell=1}^{D-1}\sum_{m_2+m_3=m-\Delta_{\ell}}\frac{m!}{m_2!m_3!}t^{d_{\ell}}
\sigma_{q,t}^{\delta_{\ell}}(\partial_{\epsilon}^{m_2}c_{\ell}(t,z,0)R_{\ell}(\partial_z)h_{m_{3}}(t,z))+
\sigma_{q,t}(\partial_\epsilon^mf)(t,z,0),\label{e771d}
\end{multline}
for every $m\ge \max\{d_{D_{1}},d_{D_{2}},\max_{1\le \ell\le D-1}\Delta_{\ell}\}$, and all $(t,z)\in\mathcal{T}\times H_{\beta'}$. %Bearing in mind that $c_{\ell}$  is holomorphic with respect to $\epsilon$ in a neighborhood of the origin, in such neighborhood one has
%\begin{equation}\label{e1316}
%c_{\ell}(t,z,\epsilon)=\sum_{m\ge0}\frac{(\partial_\epsilon^m c_\ell)(t,z,0)}{m!}\epsilon^m,
%\end{equation}
%for every $1\le \ell\le D-1$. 

Bearing in mind that both $c_{l}$ and $f$ are holomorphic w.r.t $\epsilon$ in a neighborhood
of the origin, in such neighborhood one has
\begin{equation}
c_{\ell}(t,z,\epsilon) = \sum_{m \geq 0}
\frac{(\partial_{\epsilon}^{m}c_{m})(t,z,0)}{m!} \epsilon^{m} \ \ , \ \
f(t,z,\epsilon) = \sum_{m \geq 0}
\frac{(\partial_{\epsilon}^{m}f)(t,z,0)}{m!} \epsilon^{m} \label{e1316}
\end{equation}
for every $1 \leq \ell \leq D-1$.

By plugging (\ref{e1526}) into (\ref{e771b}) and bearing in mind (\ref{e771d}) and (\ref{e1316}) one concludes that the formal power series $\hat{u}(t,z,\epsilon)=\sum_{m\ge0}h_m(t,z)\epsilon^m/m!$ is a solution of equation (\ref{e771b}).  

\end{proof}

\vspace{1cm}

\textbf{Acknowledgements:} We want to express our gratitude to the anonimous referee for the valuable comments and suggestions made which helped to improve the work.

\noindent \textbf{Funding:} This project has received funding from the European Research Council (ERC) under the European Union's Horizon 2020 research and innovation programme under the Grant Agreement No 648132. The second and third authors are partially supported by the research project MTM2016-77642-C2-1-P of Ministerio de Econom\'ia, Industria y Competitividad, Spain.

\noindent \textbf{Competing interests:} The authors declare that they have no competing interests.

\noindent \textbf{Availability of data and material:} Data sharing is not applicable to this article as no datasets were generated or analysed during the current study.

\noindent \textbf{Authors' contributions:} All authors contributed equally and significantly in writing this paper and typed, read, and approved the final manuscript.

\providecommand{\bysame}{\leavevmode\hbox to3em{\hrulefill}\thinspace}
\providecommand{\MR}{\relax\ifhmode\unskip\space\fi MR }
\providecommand{\MRhref}[2]{%
  \href{http://www.ams.org/mathscinet-getitem?mr=#1}{#2}
}
\providecommand{\href}[2]{#2}


\providecommand{\bysame}{\leavevmode\hbox to3em{\hrulefill}\thinspace}
\providecommand{\MR}{\relax\ifhmode\unskip\space\fi MR }
\providecommand{\MRhref}[2]{%
  \href{http://www.ams.org/mathscinet-getitem?mr=#1}{#2}
}
\providecommand{\href}[2]{#2}
\begin{thebibliography}{DVZ09}

\bibitem{ba2}
W. Balser, \emph{Formal power series and linear systems of meromorphic
  ordinary differential equations}, Universitext, Springer-Verlag, New York,
  2000.

\bibitem{cota2}
O.~Costin and S.~Tanveer, \emph{Short time existence and {B}orel summability in
  the {N}avier-{S}tokes equation in {$\Bbb R^3$}}, Comm. Partial Differential
  Equations \textbf{34} (2009), no.~7-9, 785--817.

\bibitem{dreyfuseloy}
T. Dreyfus and A. Eloy, \emph{{$q$}-{B}orel-{L}aplace summation for
  {$q$}-difference equations with two slopes}, J. Difference Equ. Appl.
  \textbf{22} (2016), no.~10, 1501--1511.

\bibitem{dreyfus}
T. Dreyfus, \emph{Building meromorphic solutions of {$q$}-difference
  equations using a {B}orel-{L}aplace summation}, Int. Math. Res. Not. IMRN
  (2015), no.~15, 6562--6587.

\bibitem{viziozhang}
L. Di~Vizio and C. Zhang, \emph{On {$q$}-summation and confluence},
  Ann. Inst. Fourier (Grenoble) \textbf{59} (2009), no.~1, 347--392.

\bibitem{hssi}
P.-F. Hsieh and Y. Sibuya, \emph{Basic theory of ordinary differential
  equations}, Universitext, Springer-Verlag, New York, 1999.

\bibitem{lamaq2}
A. Lastra and S. Malek, \emph{On {$q$}-{G}evrey asymptotics for
  singularly perturbed {$q$}-difference-differential problems with an irregular
  singularity}, Abstr. Appl. Anal. (2012), Art. ID 860716, 35.

\bibitem{lama2}
\bysame, \emph{On parametric {G}evrey asymptotics for singularly perturbed
  partial differential equations with delays}, Abstr. Appl. Anal. (2013), Art.
  ID 723040, 18. \MR{3129346}

\bibitem{lama}
\bysame, \emph{On parametric {G}evrey asymptotics for singularly perturbed
  partial differential equations with delays}, Abstr. Appl. Anal. (2013), Art.
  ID 723040, 18.

\bibitem{lamaq}
\bysame, \emph{On parametric multilevel {$q$}-{G}evrey asymptotics for some
  linear {$q$}-difference-differential equations}, Adv. Difference Equ. (2015),
  2015:344, 52.

\bibitem{lama0}
\bysame, \emph{On parametric multisummable formal solutions to some nonlinear
  initial value {C}auchy problems}, Adv. Difference Equ. (2015), 2015:200, 78.

\bibitem{lm1}
\bysame, \emph{On multiscale {G}evrey and q-{G}evrey asymptotics for some linear
  {$q$}-difference differential initial value Cauchy problems}, J. Difference Equ. Appl. 23 (2017), no. 8, 1397--1457.

\bibitem{lamasa2}
A. Lastra, S. Malek, and J. Sanz, \emph{On {$q$}-asymptotics
  for linear {$q$}-difference-differential equations with {F}uchsian and
  irregular singularities}, J. Differential Equations \textbf{252} (2012),
  no.~10, 5185--5216.

\bibitem{ma2}
S.~Malek, \emph{On {G}evrey asymptotics for some nonlinear integro-differential
  equations}, J. Dyn. Control Syst. \textbf{16} (2010), no.~3, 377--406.

\bibitem{ma0}
\bysame, \emph{On singularly perturbed {$q$}-difference-differential equations
  with irregular singularity}, J. Dyn. Control Syst. \textbf{17} (2011), no.~2,
  243--271.

\bibitem{ma}
S. Malek, \emph{Parametric {G}evrey asymptotics for a $q$-analog of
  some linear initial value problem}, Funkcial. Ekvac. 60 (2017), no. 1, 21--63.

%\bibitem{MZ}
%Fabienne Marotte and Changgui Zhang, \emph{Multisommabilit\'e des s\'eries
%  enti\`eres solutions formelles d'une \'equation aux {$q$}-diff\'erences
%  lin\'eaire analytique}, Ann. Inst. Fourier (Grenoble) \textbf{50} (2000),
%  no.~6, 1859--1890 (2001).

\bibitem{pr1}
D. Pravica, N. Randriampiry, M. Spurr, \emph{On q-advanced spherical Bessel functions of the first kind and perturbations of the Haar wavelet.} Appl. Comput. Harmon. Anal. 44 (2018), no. 2, 350--413.


\bibitem{pr2}
D. Pravica, N. Randriampiry, M. Spurr, \emph{$q$-Advanced models for tsunamis and rogue waves.}Abstr. Appl. Anal., 2012 (2012), Article 414060.

\bibitem{pr3}
D. Pravica, N. Randriampiry, M. Spurr,  \emph{Solutions of a class of multiplicatively advanced differential equations.} C. R. Math. Acad. Sci. Paris 356 (2018), no. 7, 776--817.

\bibitem{ra}
J.-P. Ramis, \emph{About the growth of entire functions solutions of
  linear algebraic {$q$}-difference equations}, Ann. Fac. Sci. Toulouse Math.
  (6) \textbf{1} (1992), no.~1, 53--94.

\bibitem{ta} 
H. Tahara, \emph{{$q$}-{A}nalogues of Laplace and Borel transforms by means of {$q$}-exponentials}, Ann. Inst. Fourier (Grenoble) 67 (2017), no. 5, 1865--1903.

\bibitem{taya13}
H. Tahara and H. Yamazawa, \emph{Multisummability of formal
  solutions to the {C}auchy problem for some linear partial differential
  equations}, J. Differential Equations \textbf{255} (2013), no.~10,
  3592--3637.

\bibitem{taya2}
\bysame, \emph{{$q$}-analogue of summability of formal solutions of some linear
  {$q$}-difference-differential equations}, Opuscula Math. \textbf{35} (2015),
  no.~5, 713--738.


\bibitem{ya}
H. Yamazawa, \emph{Holomorphic and singular solutions of
  {$q$}-difference-differential equations of {B}riot-{B}ouquet type}, Funkcial.
  Ekvac. \textbf{59} (2016), no.~2, 185--197.

\bibitem{ya2}
\bysame, \emph{Gevrey and {$q$}-Gevrey asymptotics for some linear {$q$}-difference-differential equations}, communication in ``Formal and analytic solutions of functional equations on the complex domain'', Kyoto, 2018.

%\bibitem{Z02}
%Changgui Zhang, \emph{Une sommation discr\`ete pour des \'equations aux
%  {$q$}-diff\'erences lin\'eaires et \`a coefficients analytiques: th\'eorie
%  g\'en\'erale et exemples}, Differential equations and the {S}tokes
%  phenomenon, World Sci. Publ., River Edge, NJ, 2002, pp.~309--329.

\end{thebibliography}
\end{document}